\documentclass[10pt]{amsart}

\usepackage{amsmath,amssymb,amsthm,amsfonts,mathrsfs}
\usepackage[all]{xy}
\CompileMatrices
\usepackage{mathdots}
\usepackage{ifthen}
\usepackage{fullpage}
\usepackage[usenames]{color}
\usepackage{comment}
\usepackage[shortlabels]{enumitem}
\usepackage[colorlinks=true,
linkcolor=blue,
anchorcolor=blue,
citecolor=red
]{hyperref}
\allowdisplaybreaks 
\newtheorem{thm}{Theorem}[section]
\newtheorem{prop}[thm]{Proposition}
\newtheorem{cor}[thm]{Corollary}
\newtheorem{lem}[thm]{Lemma}

\theoremstyle{definition}
\newtheorem{defn}[thm]{Definition}

\theoremstyle{remark}

\newcommand{\Z}{\mathbb{Z}}

\newcommand{\C}{\mathbb{C}}

\newcommand{\A}{\mathbb{A}}

\newcommand{\Sp}{\mathrm{Sp}}
\newcommand{\GL}{\mathrm{GL}}

\newcommand{\diag}{\mathrm{diag}}

\newcommand{\h}{\mathbf{h}} 

  1

\newcommand{\transpose}[1]{\text{$^t\!#1$}}
\newcommand{\G}{\ifmmode {\mathcal{G}}\else${\mathcal{G}}$\ \fi}

\author{Yubo Jin}
\address{Institute for Advanced Study in Mathematics, Zhejiang University\\
  Hangzhou, 310058, China}\email{yubo.jin@zju.edu.cn}

\author{Pan Yan}
\address{Department of Mathematics, The University of Arizona, Tucson, AZ 85721, USA}
\email{panyan@arizona.edu}

\makeatletter
\@namedef{subjclassname@2020}{\textup{}2020 Mathematics Subject Classification}
\makeatother
\date{\today}
\title{$L$-functions for $\mathrm{Sp}(2n)\times\mathrm{GL}(k)$ via non-unique models}


\begin{document}

\begin{abstract}
Let $n$ and $k$ be positive integers such that $n$ is even.
We derive new global integrals for $\mathrm{Sp}_{2n}\times\mathrm{GL}_k$ from the generalized doubling method of Cai, Friedberg, Ginzburg and Kaplan, following a strategy and extending a previous result of Ginzburg and Soudry on the case $n=k=2$. We show that these new integrals unfold to non-unique models on $\mathrm{Sp}_{2n}$. Using the New Way method of Piatetski-Shapiro and Rallis, we show that these new global integrals represent the $L$-functions for $\mathrm{Sp}_{2n}\times\mathrm{GL}_k$, generalizing a previous result of the second-named author on $\mathrm{Sp}_{4}\times\mathrm{GL}_2$ and a previous work of Piatetski-Shapiro and Rallis on $\mathrm{Sp}_{2n}\times\mathrm{GL}_1$.
\end{abstract}

\maketitle
\tableofcontents

\section{Introduction}

In a recent work, Ginzburg and Soudry \cite{GinzburgSoudry2020} introduced a new method for constructing integral representations of certain $L$-functions, aiming to explain the existence of many global Rankin-Selberg integrals. This method builds upon the generalized doubling integral of Cai, Friedberg, Ginzburg and Kaplan \cite{CaiFriedbergGinzburgKaplan2019}, which represents the standard $L$-function for $G\times \GL_k$ where $G$ is a split classical group. 
 We briefly review the construction of \cite{CaiFriedbergGinzburgKaplan2019} for the symplectic group $G=\Sp_{2n}$, defined over a number field $F$, with the ring of ad\`{e}les denoted by $\A$. Let $\pi$ and $\tau$ be irreducible cuspidal automorphic representations of $G(\mathbb{A})$ and $\GL_k(\A)$ respectively. Consider the Eisenstein series $E(\cdot;f_{2n,k,s})$ on  $H=\Sp_{4kn}$, attached to a section $f_{2n,k,s}$ of the parabolic induction $\mathrm{Ind}_{P_{2kn}(\mathbb{A})}^{H(\mathbb{A})}(\Delta(\tau,2n)|\det\cdot|^s)$, where $P_{2kn}$ is the standard Siegel parabolic subgroup of $\Sp_{4kn}$ and $\Delta(\tau,2n)$ is the generalized Speh representation of $\GL_{2kn}(\A)$, associated to $\tau$. The generalized doubling integral introduced in \cite{CaiFriedbergGinzburgKaplan2019} takes the form
\begin{equation}
	\label{introduction-eq-CFGK1}
		\begin{aligned}
			\int_{G(F)\times G(F)\backslash G(\mathbb{A})\times G(\mathbb{A})} \phi_1(g_1)\overline{\phi_2(^{\iota}g_2)}\int_{U(F)\backslash U(\mathbb{A})}
			 E(u(g_1\times g_2);f_{2n,k,s})\psi_U(u)dudg_1dg_2.
		\end{aligned}
\end{equation}
Here, $\phi_1, \phi_2$ are cusp forms in the space of $\pi$; ${}^\iota$ is an involution of $G$; $U$ is a certain unipotent subgroup of $H$ and $\psi_U$ is a character of $U(\A)$ which is trivial on $U(F)$. 
Finally, the expression $(g_1\times g_2)$ stands for the image of an embedding of the group $G\times G$ in the normalizer of $U$ and stabilizer of $\psi_U$. By taking $k=1$, $U$ as the trivial group, and $\tau$ as the trivial representation of $\GL_1(\A)$, the integral \eqref{introduction-eq-CFGK1} recovers the doubling integral of \cite{Piatetski-ShapiroRallis1987Doubling}, which represents the standard $L$-function for $G$.
 

Starting from \eqref{introduction-eq-CFGK1}, Ginzburg and Soudry considered the following function for $g\in G(\A)$:
\begin{equation}
\label{introduction-eq-GS1}	
\xi(\phi, f_{2n,k,s})(g)=\int_{G(F)\backslash G(\A)}\int_{U(F)\backslash U(\A)} \phi(h)  E(u(g\times h);f_{2n,k,s})\psi_U(u)dudh,
\end{equation}
where $\phi=\phi_\pi$ is a cusp form on the space of $\pi$, and $f_{2n,k,s}$ is a $K$-finite holomorphic section. 
They proved that the function $\xi(\phi, f_{2n,k,s})$ is meromorphic and takes value in $\pi^\iota$, an outer conjugate of $\pi$ by an element of order 2, and that the right-hand side of \eqref{introduction-eq-GS1}	 is Eulerian, representing the same partial $L$-function that integral \eqref{introduction-eq-CFGK1} represents, after normalizing the Eisenstein series.

Moreover, Ginzburg and Soudry also proposed that all known integral representations of $L^S(s+\frac{1}{2},\pi\times\tau)$ should arise from the expression \eqref{introduction-eq-GS1}, by applying a certain Fourier coefficient or a period integral to $\xi(\phi, f_{2n,k,s})$. To support their claim, they showed that the well-known New Way integral of Piatetski-Shapiro and Rallis in \cite{Piatetski-ShapiroRallis1988} can be derived from \eqref{introduction-eq-GS1}, when $n$ is even and $k=1$, by taking the Fourier coefficient of this function with respect to $(N_n, \psi_T)$, where $N_n$ is the unipotent radical of the Siegel parabolic subgroup of $\Sp_{2n}$, $T$ is a certain element in $\GL_n(F)$, and $\psi_T$ is the character on $N_n(F)\backslash N_n(\A)$ given by 
\begin{equation*}
\psi_T \left(\left[\begin{smallmatrix}
1_n & z\\
0 & 1_n
\end{smallmatrix}\right]\right)=\psi(\mathrm{tr}(Tz)).
\end{equation*}
They also applied the same idea to the function \eqref{introduction-eq-GS1} when $n=k=2$,
to derive and obtain a new and ``simpler" global integral. They further conjectured that this new integral represents the $L$-function for $\Sp_4\times\GL_2$ via the New Way method. This conjecture is proven to hold by the second-named author in \cite{Yan2021}.

The first goal of this paper is to apply the same global arguments in \cite{GinzburgSoudry2020} to obtain new global integrals for $\Sp_{2n}\times\GL_k$ for any even integer $n$. Throughout this paper, we assume $n$ is even.  Let $\pi$ and $\tau$ be irreducible automorphic cuspidal representations of $\Sp_{2n}(\A)$ and $\GL_k(\A)$ respectively. 
Following \cite{GinzburgSoudry2020},  we consider the following Fourier coefficient of $\xi(\phi,f_{2n,k,s})$:
	\begin{equation*}
		\mathcal{L}(\phi,f_{2n,k,s})=\int_{N_n(F)\backslash N_n(\A)}\xi(\phi,f_{2n,k,s})\left(u\right)\psi_T(u)du.
	\end{equation*}
We apply global root exchange, and an identity involving Eisenstein series induced from the Speh representations (see Lemma~\ref{lem 3.2}), to derive the following new integral.
	
	\begin{thm}(Theorem~\ref{thm 3.3})
		\label{introduction-thm-1}
		Let $n$ be an even positive integer. Given a Schwartz function $\Phi\in\mathcal{S}(\mathrm{Mat}_{n}(\mathbb{A}))$, there are nontrivial choices of sections
		\[
		\begin{aligned}
			f_{2n,k,s}&\in\mathrm{Ind}_{P_{2kn}(\mathbb{A})}^{\mathrm{Sp}_{4kn}(\mathbb{A})}(\Delta(\tau,2n)|\det\cdot|^s),\\
			f_{n,k,s}&\in\mathrm{Ind}_{P_{kn}(\mathbb{A})}^{\mathrm{Sp}_{2kn}(\mathbb{A})}(\Delta(\tau\otimes\chi_T,n)|\det\cdot|^s),
		\end{aligned}
		\]
		such that $\mathcal{L}(\phi,f_{2n,k,s})$ is equal to
		\begin{equation}
		\label{introduction-eq-1}
			\begin{aligned}
				 &\mathcal{Z}(\phi,\theta_{\psi,n^2}^{\Phi}, f_{n,k,s})\\
     :=&\int_{\mathrm{Sp}_{2n}(F)\backslash\mathrm{Sp}_{2n}(\mathbb{A})}\int_{N_{n^{k-1},kn}(F)\backslash N_{n^{k-1},kn}(\mathbb{A})}\phi(h) 
				 \theta^{\Phi}_{\psi,n^2}(\alpha_T(u)i_T(1,h))E(ut(1,h);f_{n,k,s})\psi_k(u)dudh.
			\end{aligned}
		\end{equation}
Here:\\
\indent (1) $\phi\in V_\pi$ is a non-zero cusp form;\\
\indent (2) $N_{n^{k-1},kn}$ is a certain unipotent subgroup of $\Sp_{2kn}$ and $\psi_k$ is a character on $N_{n^{k-1},kn}(\mathbb{A})$ which is trivial on $N_{n^{k-1},kn}(F)$; \\
\indent (3) $\theta^{\Phi}_{\psi,n^2}$ is a theta series associated to the dual pair $(\mathrm{SO}_{T_0},\mathrm{Sp}_{2n})$, where $T_0\in \GL_n(F)\cap \mathrm{Sym}_n(F)$;\\
\indent (4)  $T\in \GL_n(F)$ depends on $T_0$, $\chi_T$ is the character $\chi_T(x)=(x, \det(T))$ given by the global Hilbert symbol, and $\Delta(\tau\otimes\chi_T,n)$ is the generalized Speh representation of $\mathrm{GL}_{kn}(\mathbb{A})$ associated to $\tau\otimes\chi
_T$;\\
\indent (5) $E$ is an Eisenstein series on $\mathrm{Sp}_{2kn}(\A)$ associated to the section $f_{n,k,s}$. 
	\end{thm}

We refer the reader to Section \ref{section-preliminaries} for the precise definitions of the notations.

Theorem~\ref{introduction-thm-1} generalizes \cite[Theorem 4]{GinzburgSoudry2020} where the case $n=k=2$ was considered, to any positive even integer $n$ and any positive integer $k$. We remark that the assumption that $n$ is even is made in order to avoid the use of the Eisenstein series on metaplectic groups (the same assumption also appeared in \cite{Piatetski-ShapiroRallis1988} and \cite[Section 4.1]{GinzburgSoudry2020}).

Our second goal is to study integral $\mathcal{Z}(\phi,\theta_{\psi,n^2}^{\Phi}, f_{n,k,s})$ and establish an integral representation for $L^S(s+\frac{1}{2},\pi\times\tau)$. This can be formulated in terms of normalized Eisenstein series. Let $S$ be a finite set of places (defined in Section \ref{section-global-zeta-integral}) and let
$
f_{v, n, k, s}^*(g)= d_{\tau_v\otimes\chi_T}^{\mathrm{Sp}_{2kn}} \cdot    f_{v, n, k, s}(g)
$ 
if $v\not\in S$,
where $d_{\tau_v\otimes\chi_T}^{\mathrm{Sp}_{2kn}}$ is given by \eqref{2.4.2} or \eqref{2.4.3} depending on the parity of $k$. For $v\in S$, let $
f_{v, n, k, s}^*(g)=     f_{v, n, k, s}(g).
$ 
Put $f^{*, S}_{n, k, s}=\prod_v f_{v, n, k, s}^*(g)$ and 
\begin{equation*}
E(g, f_{n, k, s}^{*, S})= \sum_{\gamma\in P_{kn}(F)\backslash\mathrm{Sp}_{2kn}(F)}f_{n,k,s}^{*, S}(\gamma g).
\end{equation*}
This is the partially normalized (outside $S$) Eisenstein series. Our second main result is the following integral representation formula. 

\begin{thm}
\label{thm-main}
Let $n$ be an even positive integer.
There exists a choice of a nonzero cusp form $\phi\in V_{\pi}$, a matrix $T_0$,
a theta series $\theta_{\psi,n^2}^{\Phi}$, and a section $f_{n,k,s}\in\mathrm{Ind}_{P_{kn}(\mathbb{A})}^{\mathrm{Sp}_{2kn}(\mathbb{A})}(\Delta(\tau\otimes\chi_T,n)|\det\cdot|^s)$ such that
\begin{equation}
\mathcal{Z}(\phi,\theta_{\psi,n^2}^{\Phi}, f_{n,k,s}^{*, S})=L^S(s+\frac{1}{2},\pi\times\tau)\cdot\mathcal{Z}_S(\phi,\theta_{\psi,n^2}^{\Phi},f_{n,k,s}),
\end{equation}
where  $\mathcal{Z}_S(\phi,\theta_{\psi,n^2}^{\Phi},f_{n,k,s})$ is meromorphic in $s$. Moreover, for any $s_0\in\C$, the data can be chosen such that $\mathcal{Z}_S(\phi,\theta_{\psi,n^2}^{\Phi},f_{n,k,s})$ is holomorphic and nonzero in a neighborhood of $s=s_0$.
\end{thm}

Theorem~\ref{thm-main} generalizes the New Way integrals for $\Sp_{2n}$ (when $n$ is even) in \cite{Piatetski-ShapiroRallis1988} where the case $k=1$ was treated. As an important application of the New Way integral for $\Sp_4$ (i.e., when $n=2$), Kudla, Rallis and Soudry \cite{KudlaRallisSoudry1992} used the Siegel-Weil formula (in the context of the doubling method) to prove that, for a cuspidal representation $\pi$ of $\Sp_4(\A)$, if the twisted standard $L$-function $L^S(s, \pi\times\chi_T)$ has a simple pole at $s=1$, then the theta lift of $\pi$ to an appropriate orthogonal group $O(V)$ is non-zero. Moreover, they proved that if a cuspidal representation $\pi=\otimes_v^\prime \pi_v$ of $\Sp_4(\A)$ is locally generic at every place $v$ and $L^S(s,\pi)$ has either a pole or a non-zero value at $s=1$, then $\pi$ is globally generic. Recently, Ginzburg and Soudry \cite{GinzburgSoudry2024} generalized parts of the regularized Siegel-Weil formula (in the context of the generalized doubling method), in a project to detect CAP representations of symplectic groups. We expect that Theorem~\ref{thm-main} will have similar applications as in \cite{KudlaRallisSoudry1992}, using the results of \cite{GinzburgSoudry2024}. Theorem~\ref{thm-main} also generalizes the main result of \cite{Yan2021} where the case $n=k=2$ was treated.


Now we give a summary of our proof of Theorem~\ref{thm-main}.
The first step is to unfold the integral $\mathcal{Z}(\phi,\theta_{\psi,n^2}^{\Phi}, f_{n,k,s})$. We begin by unfolding the Eisenstein series, followed by unfolding the theta series and show that $\mathcal{Z}(\phi,\theta_{\psi,n^2}^{\Phi},f_{n,k,s})$ unfolds to the Fourier coefficient of $\phi$ given by 
\begin{equation}
\label{eq-Fourier-coef}
\int_{N_n(F)\backslash N_n(\mathbb{A})}\phi\left(nh\right)\psi_T(n)dn.
\end{equation}
See Proposition~\ref{prop 4.1}.
We point out that the existence of $T$ such that the integral \eqref{eq-Fourier-coef} is non-zero is due to \cite{LiJian-Shu1992}. 
In general, the model of $\pi$ corresponding to \eqref{eq-Fourier-coef} is not unique (see \cite{Piatetski-ShapiroRallis1988}), thus the New Way method is required to analyze the global integral.  
The next step is to carry out the local unramified computation.  
This is done in Theorem~\ref{thm-unramified-computation}. The main idea is to compare the unramified integral with the one from the generalized doubling method.
We can also control the local zeta integral at a place $v\in S$. This is done in Proposition~\ref{prop-finite-non-vanishing} and Proposition~\ref{prop-archimedean-non-vanishing}. More specifically, at a finite place one can choose the data so that the local integral is equal to 1 for all $s$, and at the archimedean places the integral can be made holomorphic and nonzero in a neighborhood of any given $s$.
Then Theorem~\ref{thm-main} follows from
Theorem~\ref{thm-main-restated}, Proposition~\ref{prop-finite-non-vanishing} and Proposition~\ref{prop-archimedean-non-vanishing}.

Let us mention some other integrals for the standard $L$-function for symplectic groups. There are integrals of Shimura type for $\mathrm{Sp}_{2n}\times\mathrm{GL}_k$ \cite{GinzburgRallisSoudry1998} and integrals of Hecke type for $\mathrm{Sp}_{2n}\times\mathrm{GL}_1$ \cite{Yan2024JRMS}, using the Whittaker model. There are also integrals of \cite{GinzburgJiangRallisSoudry2011} using the Fourier-Jacobi model. For more examples of New Way integrals as well as their applications, we refer the reader to \cite{BumpFurusawaGinzburg1995, GurevichSegal2015, PollackShah2017, PollackShah2018, Ginzburg2018}. 
See also \cite{Pollack2022APAW} for a survey of the New Way method.

We give an overview of the structure of the rest of this paper. In Section~\ref{section-preliminaries}, after fixing some notations we recall the definitions of theta series, Eisenstein series, and $(k,c)$-representations. In Section~\ref{section-new-integrals}, we review the global and local integrals from the generalized doubling method and derive new Rankin-Selberg integrals following a strategy of \cite{GinzburgSoudry2020}. The main new result in this section is Theorem~\ref{thm 3.3}. In Section~\ref{section-global-zeta-integral}, we state our main results on the new integrals studied in this paper, while delaying the proofs to later sections, to give a more streamlined presentation. In Section~\ref{section-unfolding}, we give the global unfolding computation and in Section~\ref{section-unramified-computation}, we carry out the local unramified computation. 

\section{Preliminaries}
\label{section-preliminaries}
\subsection{Notation}

Let $F$ be a number field and $\mathbb{A}=\mathbb{A}_F$ the ring of ad\`{e}les. Denote $\mathrm{Mat}_{m\times n}$ for the algebraic group of all matrices of size $m\times n$ and $\mathrm{Mat}_n=\mathrm{Mat}_{n\times n}$. Let $1_n$ be the $n\times n$ identity matrix. Set $J_n$ for the $n\times n$ matrix with ones on the anti-diagonal and zeros everywhere else. For $x\in \mathrm{Mat}_{n}$ we denote $\transpose{x}$ for the transpose of $x$ and $x^{\ast}=J_n\transpose{x}J_n$. If $x\in \GL_n$, denote $\hat{x}=(x^{\ast})^{-1}=J_n\transpose{x}^{-1}J_n$. The symplectic group $\mathrm{Sp}_{2n}$ is realized as
\[
\mathrm{Sp}_{2n}=\left\{g\in\mathrm{GL}_{2n}:\transpose g\left[\begin{smallmatrix}
0 & J_n\\
-J_n & 0
\end{smallmatrix}\right] g=\left[\begin{smallmatrix}
0 & J_n\\
-J_n & 0
\end{smallmatrix}\right]\right\}.
\]
We denote the standard Borel subgroups of $\Sp_{2n}$ and $\GL_n$ consisting of upper triangular matrices by $B_{\Sp_{2n}}$ and $B_{\GL_n}$ respectively. Let $\mathbf{r}=(r_1,{\cdots},r_m)\in\Z_{\geq 0}^m$ be a $m$-tuple with $0\leq r_1+{\cdots}+r_m\leq n$ and denote $|\mathbf{r}|=r_1+{\cdots}+r_m$. 
Let $P_{\mathbf{r},n}$ be the standard parabolic subgroup of $\mathrm{Sp}_{2n}$ with Levi decomposition $P_{\mathbf{r},n}=M_{\mathbf{r},n}\ltimes N_{\mathbf{r},n}$, where $M_{\mathbf{r},n}\cong\mathrm{GL}_{r_1}\times{\cdots}\times\mathrm{GL}_{r_m}\times \mathrm{Sp}_{2(n-|\mathbf{r}|)}$. If $r_1={\cdots}=r_m=r\in\Z_{\geq 0}$ we also denote the tuple $\mathbf{r}$ by $r^m$. If $m=1$ we omit it from the notation and simply write $r$. In particular, for $m=1,r=n$ we obtain the Siegel parabolic subgroup $P_{n}:=P_{n,n}$. Let $M_n:=M_{n,n}, N_n:=N_{n,n}$. Then we have $P_n=M_n\ltimes N_n$ where
\[
\begin{aligned}
M_n=\left\{m(x)=\left[\begin{smallmatrix}
x & 0\\
0 & \hat{x}
\end{smallmatrix}\right]:x\in\mathrm{GL}_n\right\}, \qquad N_n=\left\{n(z)=\left[\begin{smallmatrix}
1_n & z\\
0 & 1_n
\end{smallmatrix}\right]:z\in\mathrm{Mat}_n^0\right\}.
\end{aligned}
\]
Here
\[
\mathrm{Mat}_n^0=\{A\in\mathrm{Mat}_n:{}^tAJ_n=J_nA\}.
\]
For an integer $k\geq 2$, we will frequently use the following two unipotent subgroups $N_{n^{k-1},kn}$ and $N_{n^k,2kn}$.

The unipotent subgroup $N_{n^{k-1},kn}$ contains elements of the form
\begin{equation}
\label{2.1.1}
\left[\begin{smallmatrix}
    1_n & u_{1,2} & \ast & \ast & \ast & \ast & \ast & \ast & \ast & \ast\\
0 & \ddots & \ddots & \ast & \ast & \ast & \ast & \ast & \ast & \ast\\
0 & 0 & 1_n & u_{k-2,k-1} & \ast & \ast & \ast & \ast & \ast & \ast\\
0 & 0 & 0 & 1_n & x & y & z & \ast & \ast & \ast\\
0 & 0 & 0 & 0 & 1_n & 0 & y^{\ast} & \ast & \ast & \ast\\
0 & 0 & 0 & 0 & 0 & 1_n & -x^{\ast} & \ast & \ast & \ast\\
0 & 0 & 0 & 0 & 0 & 0 & 1_n & -u_{k-2,k-1}^{\ast} & \ast & \ast\\
0 & 0 & 0 & 0 & 0 & 0 & 0 & \ddots & \ddots & \ast\\
0 & 0 & 0 & 0 & 0 & 0 & 0 & 0 &1_n & -u_{1,2}^{\ast}\\
0 & 0 & 0 & 0 & 0 & 0 & 0 & 0 & 0 & 1_n
\end{smallmatrix}\right] \in \Sp_{2kn}.
\end{equation}
In particular, when $k=2$, we have
\begin{equation}
\label{2.1.2}
N_{n,2n}=\left\{u(x,y,z):=\left[\begin{smallmatrix}
1_n & x & y & z\\
0 & 1_n & 0 & y^{\ast}\\
0 & 0 & 1_n & -x^{\ast}\\
0 & 0 & 0 & 1_n
\end{smallmatrix}\right]:\begin{array}{c}
x,y,z\in\mathrm{Mat}_n,\\
z+z^{\ast}+yx^{\ast}-xy^{\ast}=0\end{array}\right\}.
\end{equation}
By sending elements of the form in \eqref{2.1.1} to its central $2n\times 2n$ block we have a projection map
\begin{equation}
\label{2.1.2-map}
\mathrm{Pr}:N_{n^{k-1},kn}\to N_{n,2n}.
\end{equation}
For $u(x,y,z)\in N_{n,2n}$, we abuse the notation by using the same notation $u(x,y,z)$ to mean any elements in the pre-image $\mathrm{Pr}^{-1}(u(x,y,z))\in N_{n^{k-1},kn}$ and denote $u^0(x,y,z)$ to emphasize the one in $N_{n^{k-1},kn}$ obtained by the natural embedding $N_{n,2n}\to N_{n^{k-1},kn}$.

The unipotent subgroup $N_{n^k,2kn}$ contains elements of the form
\begin{equation}
\label{2.1.3}
\left[\begin{smallmatrix}
1_n & u_{1,2} & \ast & \ast & \ast & \ast & \ast & \ast & \ast\\
0 & \ddots & \ddots & \ast & \ast & \ast & \ast & \ast & \ast\\
0 & 0 & 1_n & u_{k-1,k} & \ast & \ast & \ast & \ast & \ast\\
0 & 0 & 0 & 1_n & y & z & \ast & \ast & \ast\\
0 & 0 & 0 & 0 & 1_{2kn} & y' & \ast & \ast & \ast\\
0 & 0 & 0 & 0 & 0 & 1_n & -u_{k-1,k}^{\ast} & \ast & \ast\\
0 & 0 & 0 & 0 & 0 & 0 & \ddots & \ddots & \ast\\
0 & 0 & 0 & 0 & 0 & 0  & 0 & 1_n & -u_{1,2}^{\ast}\\
0 & 0 & 0 & 0 & 0 & 0 & 0 & 0 & 1_n
\end{smallmatrix}\right] \in \Sp_{4kn}.
\end{equation}

\subsection{Theta series}
\label{section-theta}
We fix a nontrivial additive character $\psi:F\backslash\mathbb{A}\to\C^{\times}$. Let $T_0=\mathrm{diag}[t_1,{\cdots},t_n]\in\mathrm{GL}_n(F)$ be a diagonal matrix and set $T=J_nT_0$. We define a quadratic character $\chi_T:F^{\times}\backslash\mathbb{A}^{\times}\to\C^{\times}$ by $\chi_T(x)=(x,\det(T))$ where $(\cdot,\cdot)$ is the global Hilbert symbol. Denote $\mathcal{H}_n=\mathrm{Mat}_n\times\mathrm{Mat}_n\times\mathrm{Mat}_1$ for the Heisenberg group (of $2n^2+1$ elements) with multiplication
\[
(X_1,Y_1,z_1)(X_2,Y_2,z_2)=(X_1+X_2,Y_1+Y_2,z_1+z_2+\mathrm{tr}(T(X_1Y_2^{\ast}-Y_1X_2^{\ast})))
\]
for $X_1,Y_1,X_2,Y_2\in\mathrm{Mat}_n$ and $z_1,z_2\in \mathrm{Mat}_1$. We identify $N_{n,2n}$ with $\mathcal{H}_n$ by the map
\[
\alpha_T:N_{n,2n}\to\mathcal{H}_n, \quad u(x,y,z)\mapsto (x,y,\mathrm{tr}(Tz)),
\]
where $u(x, y, z)$ is of the form in \eqref{2.1.2}.
 For an integer $k\geq 2$, we extend $\alpha_T$ to a map
\[
\alpha_T^k:N_{n^{k-1},kn}\to N_{n,2n}\to\mathcal{H}_n
\]
by taking the composite with the map in \eqref{2.1.2-map}. 

Consider the dual pair $(\mathrm{SO}_{T_0},\mathrm{Sp}_{2n})$ inside $\mathrm{Sp}_{2n^2}$ with
\[
\mathrm{SO}_{T_0}=\{g\in\mathrm{SL}_n:\transpose{g}T_0g=T_0\}.
\]
We embed $\mathrm{SO}_{T_0}\times\mathrm{Sp}_{2n}$ inside $\mathrm{Sp}_{4n}$ via $(m,h)\mapsto\mathrm{diag}[m,h,\hat{m}]$ and further embed $t:\mathrm{Sp}_{4n}\to \mathrm{Sp}_{2kn}$ by $t(g)=\mathrm{diag}[1_{(k-2)n},g,1_{(k-2)n}]$. We also denote $(m,h)$ for its image in $\mathrm{Sp}_{4n}$ and $t(m,h)$ its image in $\mathrm{Sp}_{2kn}$. 
We always assume $n$ is even so that $\mathrm{SO}_{T_0}(\A)\times\mathrm{Sp}_{2n}(\A)$ splits in the metaplectic double cover $\widetilde{\mathrm{Sp}}_{2n^2}(\mathbb{A})$. We fix such a splitting $i_T$ and consider the restriction of the Weil representation $\omega_{\psi}:=\omega_{\psi,n^2}$ of $\widetilde{\mathrm{Sp}}_{2n^2}(\mathbb{A})$, corresponding to the character $\psi$, to $\mathrm{SO}_{T_0}(\mathbb{A})\times \mathrm{Sp}_{2n}(\mathbb{A})$ under $i_T$. For a Schwartz function $\Phi\in\mathcal{S}(\mathrm{Mat}_n(\mathbb{A}))$, we have the following formulas for $\omega_{\psi}$ (see \cite{GinzburgRallisSoudry2011}):
\begin{align}
\label{weil1}
\omega_{\psi}(\alpha^k_T(u(x,y,z)))\Phi(\xi)&=\psi(\mathrm{tr}(Txy^{\ast}))\psi(2\mathrm{tr}(T\xi y^{\ast}))\psi(\mathrm{tr}(Tz))\Phi(\xi+x),\\
\label{weil2}
\omega_{\psi}(i_T(1,m(g)))\Phi(\xi)&=\chi_T(\det(g))|\det g|^{\frac{n}{2}}\Phi(\xi g),\\
\label{weil3}
\omega_{\psi}(i_T(1,b(w)))\Phi(\xi)&=\psi(\mathrm{tr}(T\transpose{\xi}w\xi))\Phi(\xi).
\end{align}
Here, $u(x,y,z)\in N_{n^{k-1},kn}(\mathbb{A}),g\in\mathrm{GL}_n(\mathbb{A}),w\in\mathrm{Mat}_n^0(\mathbb{A})$.

Given $\Phi\in\mathcal{S}(\mathrm{Mat}_n(\mathbb{A}))$, we form the theta series 
\begin{equation}
\label{2.2.1}
\theta_{\psi}^{\Phi}(\alpha^k_T(v)i_T(m,h)):=\theta_{\psi,n^2}^{\Phi}(\alpha^k_T(v)i_T(m,h)):=\sum_{\xi\in\mathrm{Mat}_n(F)}\omega_{\psi}(\alpha^k_T(v)i_T(m,h))\Phi(\xi)
\end{equation}
with $v\in N_{n^{k-1},kn}(\mathbb{A}),m\in\mathrm{SO}_{T_0}(\mathbb{A}),h\in\mathrm{Sp}_{2n}(\mathbb{A})$. 

We also need another kind of theta series. Let $\mathcal{H}_{k,n}=\mathrm{Mat}_{n,kn}\times\mathrm{Mat}_{n,kn}\times\mathrm{Mat}_1$ be the Heisenberg group of $2kn^2+1$ variables with multiplication given by
\[
(X_1,Y_1,z_1)(X_2,Y_2,z_2)=(X_1+X_2,Y_1+Y_2,z_1+z_2+\mathrm{tr}(T(X_1 J_{kn} {}^tY_2 J_n-Y_1 J_{kn} {}^tX_2 J_n))
\]
for $X_1,Y_1,X_2,Y_2\in\mathrm{Mat}_{n,kn}$ and $z_1,z_2\in \mathrm{Mat}_1$. 
For $u$ of the form in \eqref{2.1.3} with $y=(y_1, y_2)$ where $y_1, y_2\in \mathrm{Mat}_{n,kn}$, we define a map
\[
l_T^0(u):N_{n^k,2kn}\to\mathcal{H}_{k,n},\qquad l_T^0(u)=(y_1, y_2,\mathrm{tr}(Tz)). 
\]
This provides $N_{n^k,2kn}$ a structure of Heisenberg group $\mathcal{H}_{k,n}$. Consider the dual pair $(\mathrm{SO}_{T_0},\mathrm{Sp}_{2kn})$ inside $\mathrm{Sp}_{2kn^2}$ and fix a splitting $i^0_T:\mathrm{SO}_{T_0}\times\mathrm{Sp}_{2kn}\to\widetilde{\mathrm{Sp}}_{2kn^2}$ inside the metaplectic double cover. We may realize the Weil representation $\omega_{\psi,kn^2}$ in $\mathcal{S}(\mathrm{Mat}_{n,kn}(\mathbb{A}))$ and define the theta series $\theta_{\psi,kn^2}^{\Phi}$ for $\Phi\in\mathcal{S}(\mathrm{Mat}_{n,kn}(\mathbb{A}))$ similarly as above.

\subsection{Representations of $(k,c)$ type}

We recall the definition and properties of $(k,c)$ representations in \cite{CaiFriedbergGinzburgKaplan2019, CaiFriedbergGourevitchKaplan2021, CaiFriedbergKaplan2022}, both locally and globally. Let $k$ and $c$ be positive integers. Let $P_{c^k}$ be the standard parabolic subgroup of $\mathrm{GL}_{kc}$ corresponding to the orbit $(k^c)$ so that its unipotent radical $U_{c^k}$ consists of elements of the form
\begin{equation}
\label{2.3.1}
u=\left[\begin{smallmatrix}
1_c & u_{1,2} & \ast & \ast\\
0 & \ddots & \ddots & \ast\\
0 & 0 & 1_c & u_{k-1,k}\\
0 & 0 & 0 & 1_c
\end{smallmatrix}\right]\in\mathrm{GL}_{kc}.
\end{equation}
Define a character 
\begin{equation}
\begin{aligned}
\label{2.3.2}
\psi_{c^k}:U_{c^k}(F)\backslash U_{c^k}(\mathbb{A})\to\C^{\times},\qquad
u\mapsto\psi\left(\sum_{i=1}^{k-1}\mathrm{tr}(u_{i,i+1})\right).
\end{aligned}
\end{equation}
For an automorphic function $\phi$  on $\mathrm{GL}_{kc}(F)\backslash\mathrm{GL}_{kc}(\mathbb{A})$, we consider the following Fourier coefficient
\begin{equation}
\label{2.3.4}
\Lambda(\phi)=\int_{U_{c^k}(F)\backslash U_{c^k}(\mathbb{A})}\phi(u)\psi^{-1}_{c^k}(u)du.
\end{equation}

\begin{defn}
An irreducible automorphic representation $\rho$ of $\mathrm{GL}_{kc}(\mathbb{A})$ is called a $(k,c)$ representation if the following holds.\\
(1) The Fourier coefficient $\Lambda(\phi)$ does not vanish identically on the space of $\rho$, and moreover, for all unipotent orbits greater than or non-comparable with $(k^c)$, all corresponding generic Fourier coefficients are zero for all choices of data.\\
(2) Let $\rho_v$ denote the irreducible constituent of $\rho$ at a place $v$. Then for all unipotent orbits greater than or non-comparable with $(k^c)$, the corresponding twisted Jacquet module of $\rho_v$ with respect to a generic character vanishes. Moreover, $\mathrm{Hom}_{U_{c^k}(F_v)}(\rho_v,\psi_{c^k,v})$ (continuous morphisms over archimedean fields) is one-dimensional. 
\end{defn}

For the definition of the twisted Jacquet module at an archimedean place, we refer the reader to \cite[\S 1]{CaiFriedbergGourevitchKaplan2021}.

Let $(\rho,V_{\rho})$ be a $(k,c)$ representation. Then the space $\mathcal{W}(\rho,\psi)$ of functions on $\mathrm{GL}_{kc}(\mathbb{A})$
\[
g\mapsto\Lambda(\rho(g)\phi),\quad \phi\in V_{\rho}
\]
is a unique $(k,c)$ model of $\rho$. If we write $\rho\cong\otimes_v'\rho_v$ as a restricted tensor product with respect to a system of spherical vectors $\{\xi_v^0\}_{v\notin S}$, then the space $\mathrm{Hom}_{U_{c^k}(F_v)}(\rho_v,\psi_{c^k,v})$ is one-dimensional. We fix a nonzero functional $\Lambda_v\in\mathrm{Hom}_{U_{c^k}(F_v)}(\rho_v,\psi_{c^k,v})$ and denote by $\mathcal{W}(\rho_v,\psi)$ the local unique model consisting of functions on $\mathrm{GL}_{kc}(F_v)$ given by
\[
g\mapsto\Lambda_v(\rho(g)\xi_v),\quad \xi_v\in V_{\rho,v}.
\]

\begin{prop}
\label{eulerproduct}
Let $\phi=\otimes_v\xi_v\in V_{\rho}$ be a decomposable vector. For each place $v$ of $F$, there exists a functional $\Lambda_v\in\mathrm{Hom}_{U_{c^k}(F_v)}(\rho_v,\psi_{c^k,v})$ such that $\Lambda_v(\xi_v^0)=1$ for all $v\notin S$ and for all $g\in\mathrm{GL}_{kc}(\mathbb{A})$, we have
\[
\Lambda(\rho(g)\phi)=\prod_v\Lambda_v(\rho(g_v)\xi_v).
\]
\end{prop}
\begin{proof}
The proof is similar to \cite[Theorem 3.5.2]{Bump1997} and \cite[\S 4]{Shalika1974}.
\end{proof}

Let $\tau$ be an irreducible unitary cuspidal automorphic representation of $\mathrm{GL}_k(\mathbb{A})$. Let $\Delta(\tau,c)$ be the generalized Speh representation of $\mathrm{GL}_{kc}(\mathbb{A})$ associated to $\tau$ (see \cite{Jacquet1984}). This is a $(k,c)$ representation (see \cite[Introduction]{CaiFriedbergKaplan2022} and the references there for the global result and \cite{CaiFriedbergGourevitchKaplan2021} for the local result) and we only consider such $(k,c)$ representation in this paper. 
Write $\tau\cong\otimes_v\tau_v$ and $\Delta(\tau,c)\cong\otimes_v\Delta(\tau_v,c)$. Let $v$ be a finite place such that $\tau_v$ is unramified and thus can be written in the form
\[
\tau_v=\mathrm{Ind}_{B_{\mathrm{GL}_k(F_v)}}^{\mathrm{GL}_k(F_v)}(\chi_1\otimes{\cdots}\otimes\chi_k).
\]
Here $\mathrm{B_{\mathrm{GL}_k}}$ is the standard Borel subgroup of $\mathrm{GL}_k$ consisting of upper triangular matrices and $\chi_1,{\cdots},\chi_k$ are unramified quasi-characters of $F_v^{\times}$. Then by \cite[Claim 9]{CaiFriedbergGinzburgKaplan2019}, $\Delta(\tau_v,c)$ is the unique irreducible unramified quotient of 
\[
\mathrm{Ind}^{\mathrm{GL}_{kc}(F_v)}_{P_{k^c}(F_v)}((\tau_v\otimes{\cdots}\otimes\tau_v)\delta^{1/(2k)}_{P_{k^c}}),
\]
and 
\begin{equation}
\label{localdelta1}
\Delta(\tau_v,c)\cong\mathrm{Ind}_{P_{c^k}(F_v)}^{\mathrm{GL}_{kc}(F_v)}(\chi_1\circ\det\otimes{\cdots}\otimes\chi_k\circ\det).
\end{equation}
We refer the reader to \cite[Section 3]{CaiFriedbergGourevitchKaplan2021} for the explicit realizations of $(k,c)$-functionals for $\Delta(\tau_v,c)$.

We have the following global (resp. local) invariant property for the unique functional $\Lambda$ (resp. $\Lambda_v$).
	\begin{prop}
		\label{prop-invariance-(k,c)-representation}
		(\cite[Claim 8 and Proposition 24]{CaiFriedbergGinzburgKaplan2019}) For $g\in\mathrm{GL}_c$, denote $g^{\Delta}=\mathrm{diag}[g,{\cdots},g]$ for its diagonal embedding in $\mathrm{GL}_{kc}$. Then for any $g\in\mathrm{GL}_c(\mathbb{A})$ and $\phi$ in the space of $\Delta(\tau,c)$, we have
		\[
		\Lambda(\Delta(\tau,c)(g^{\Delta})\phi)=\tau(\det(g)1_k)\Lambda(\phi).
		\]
		Similarly, for any $g\in\mathrm{GL}_c(F_v)$ and any $\xi_v$ in the space of $\Delta(\tau_v,c)$, we have
\[
\Lambda_v(\Delta(\tau_v,c)(g^{\Delta})\xi_v)=\tau_v(\det(g)1_k)\Lambda_v(\xi_v).
\]
	\end{prop}

\subsection{Eisenstein series}

Let $\tau$ be an irreducible unitary cuspidal automorphic representation of $\mathrm{GL}_k(\mathbb{A})$ and $\Delta(\tau,c)$ the generalized Speh representation of $\mathrm{GL}_{kc}(\mathbb{A})$ associated to $\tau$ and $c$. Consider the induced representation
\[
\mathrm{Ind}_{P_{kc}(\mathbb{A})}^{\mathrm{Sp}_{2kc}(\mathbb{A})}(\Delta(\tau,c)|\det\cdot|^s).
\]
Its space consists of smooth functions $\tilde{f}_{c,k,s}:\mathrm{Sp}_{2kc}(\mathbb{A})\to V$ satisfying
\begin{equation*}
\tilde{f}_{c, k, s} \left(\left[\begin{smallmatrix}
a & \\
 & \hat{a}
\end{smallmatrix}\right]
\left[\begin{smallmatrix}
1_{kc} & b\\
 & 1_{kc}
\end{smallmatrix}\right]g\right)= |\det(a)|^{s+\frac{kc+1}{2}} \Delta(\tau,c)(a)\tilde{f}_{c, k, s}(g), 
\end{equation*} 
where $V$ is the space of $\Delta(\tau,c)$. We identify it with a space of smooth functions $f_{c,k,s}:\mathrm{Sp}_{2kc}(\mathbb{A})\to\C$ by setting $f_{c,k,s}(g)=\tilde{f}_{c,k,s}(g)(1_{kc})$. For a smooth section $f_{c,k,s}$, we define an Eisenstein series on $\mathrm{Sp}_{2kc}(\mathbb{A})$ by
\[
E(g;f_{c,k,s})=\sum_{\gamma\in P_{kc}(F)\backslash\mathrm{Sp}_{2kc}(F)}f_{c,k,s}(\gamma g).
\]
In this paper we will only consider the cases where $c=2n$ and $c=n$. When choosing the sections $f_{c,k,s}$ we will need the following normalizing factors (see \cite[(1.47), (1.48)]{GinzburgSoudry2021}):
\begin{equation}
\label{2.4.2}
d_{\tau_v}^{\mathrm{Sp}_{2kn}}(s)=L(s+\frac{k}{2}+\frac{1}{2},\tau_v)\prod_{j=1}^{\frac{k}{2}}L(2s+2j,\tau_v,\wedge^2)L(2s+2j-1,\tau_v,\mathrm{Sym}^2) 
\end{equation}
if $k$ is even and
\begin{equation}
\label{2.4.3}
d_{\tau_v}^{\mathrm{Sp}_{2kn}}(s)=L(s+\frac{k}{2}+\frac{1}{2},\tau_v)\prod_{j=1}^{\frac{k+1}{2}}L(2s+2j-1,\tau_v,\wedge^2)L(2s+2j,\tau_v,\mathrm{Sym}^2) 
\end{equation}
if $k$ is odd.

\section{New integrals derived from the generalized doubling method}
	\label{section-new-integrals}
	In this section, we recall the generalized doubling construction in \cite{CaiFriedbergGinzburgKaplan2019} and explain how to derive new Rankin-Selberg integrals from the generalized doubling method following a strategy of \cite{GinzburgSoudry2020}. The main new result of this section is Theorem~\ref{thm 3.3}. We will also review the local unramified integrals from the generalized doubling method, which will be used in Section~\ref{section-unramified-computation}.

	\subsection{The generalized doubling construction}
We first recall the generalized doubling construction in \cite{CaiFriedbergGinzburgKaplan2019}. To compare the notations with \emph{loc.cit} we write $G=\mathrm{Sp}_{2n},H=\mathrm{Sp}_{4kn}$. Let $P:=P_{2kn}$ be the standard Siegel parabolic subgroup of $H$ with Levi decomposition $P=M_P\ltimes U_P$ and $Q:=P_{(2n)^{k-1},2kn}$ be the standard parabolic subgroup with  $Q=M\ltimes U$ such that $U:=N_{(2n)^{k-1},2kn}$ contains elements of the form
	\begin{equation}
		\label{3.1.2}
		\left[\begin{smallmatrix}
			1_{2n} & u_{1,2} & \ast & \ast & \ast & \ast & \ast & \ast & \ast\\
			0 &\ddots & \ddots & \ast  &\ast & \ast &\ast & \ast & \ast\\
			0 & 0 & 1_{2n} & u_{k-2,k-1} & \ast & \ast & \ast & \ast & \ast\\
			0 & 0 & 0 & 1_{2n} & u_0 & \ast & \ast & \ast & \ast\\
			0 & 0 & 0 & 0 & 1_{4n} & u_0' & \ast & \ast & \ast\\
			0& 0 & 0 & 0 & 0 & 1_{2n} &-u_{k-2,k-1}^{\ast} & \ast & \ast\\
			0 & 0 & 0 & 0 & 0 & 0 & \ddots & \ddots & \ast\\
			0 & 0 & 0 & 0 &0 & 0 & 0 & 1_{2n} & -u_{1,2}^{\ast}\\
			0 & 0 & 0 & 0 & 0 & 0 & 0 & 0 & 1_{2n}
		\end{smallmatrix}\right]\in \Sp_{4kn}.
	\end{equation}
	Fix a nontrivial additive character $\psi:F\backslash\mathbb{A}\to\C^{\times}$. For $u\in U(\A)$ as in  \eqref{3.1.2} with $u_0=\left[\begin{smallmatrix}
		a_1 & b_1 & c_1\\
		a_2 & b_2 & c_2
	\end{smallmatrix}\right],a_i,c_i\in\mathrm{Mat}_n(\A)$ and $b_i\in\mathrm{Mat}_{n\times 2n}(\mathbb{A})$, we define a character $\psi_U:U(F)\backslash U(\mathbb{A})\to\C^{\times}$ by
	\begin{equation}
		\label{3.1.3}
		\psi_U(u)=\psi\left(\sum_{i=1}^{k-2}\mathrm{tr}(u_{i,i+1})+\mathrm{tr}(a_1+c_2)\right).
	\end{equation}
	The group $G\times G$ is embedded in the stabilizer of $\psi_U$ in $M$. More explicitly, the embedding is given by
	\begin{equation}
		\label{3.1.1}
		\begin{aligned}
			G\times G& \hookrightarrow M,\\
			(g_1,g_2)&\mapsto g_1\times g_2=\mathrm{diag}[g_1,{\cdots},g_1,\left[\begin{smallmatrix}
				g_{1,1} & 0 & g_{1,2}\\
				0 & g_2 & 0\\
				g_{1,3} & 0 & g_{1,4}
			\end{smallmatrix}\right],\hat{g}_1,{\cdots},\hat{g}_1]
		\end{aligned}
	\end{equation}
	where $g_1=\left[\begin{smallmatrix}
		g_{1,1} & g_{1,2}\\
		g_{2,1} & g_{2,2}
	\end{smallmatrix}\right],g_{1,i}\in\mathrm{Mat}_n$ and $g_1$ appears $k-1$ times.

	Let $(\pi,V_{\pi})$ be an irreducible cuspidal automorphic representation of $G(\mathbb{A})$ and $(\tau,V_{\tau})$ be an irreducible unitary cuspidal automorphic representation of $\mathrm{GL}_k(\mathbb{A})$. Consider the generalized Speh representation $\Delta(\tau,2n)$ of $\mathrm{GL}_{2kn}(\mathbb{A})$ associated to $\tau$ and the induced representation
	\[
	\mathrm{Ind}_{P(\mathbb{A})}^{H(\mathbb{A})}(\Delta(\tau,2n)|\det\cdot|^s).
	\]
	For a smooth section $f_{2n,k,s}$ in the above induced representation we form an Eisenstein series on $H(\mathbb{A})$ by
	\[
	E(h;f_{2n,k,s})=\sum_{\gamma\in P(F)\backslash H(F)}f_{2n,k,s}(\gamma h,s).
	\]
	For cusp forms $\phi_1,\phi_2\in\pi$, the global zeta integral considered in \cite{CaiFriedbergGinzburgKaplan2019} is
	\begin{equation}
		\label{3.1.4}
		\begin{aligned}
			Z(s,\phi_1,\phi_2,f_{2n,k,s})&=\int_{G(F)\times G(F)\backslash G(\mathbb{A})\times G(\mathbb{A})}\int_{U(F)\backslash U(\mathbb{A})}\phi_1(g_1)\overline{\phi_2(^{\iota}g_2)}\\
			\times& E(u(g_1\times g_2);f_{2n,k,s})\psi_U(u)dudg_1dg_2.
		\end{aligned}
	\end{equation}
	Here $^{\iota}g:=\iota g\iota^{-1}$ with $\iota=\left[\begin{smallmatrix}
		0 & 1_n\\
		1_n & 0
	\end{smallmatrix}\right]$.
	
	Set $U_0^\prime=U\cap U_P$ to be a subgroup of $N_{(2n)^{k-1},2kn}$ consisting elements of the form
	\[
	\left[\begin{smallmatrix}
	1_{2n(k-1)} & 0 & \ast & \ast\\
	0 & 1_{2n} & 0 & \ast\\
	0 & 0 & 1_{2n} & 0\\
	0 & 0 & 0 & 1_{2n(k-1)}
	\end{smallmatrix}\right],
	\]
	and
	\[
	\label{3.1.5}
	\delta=\left[\begin{smallmatrix}
	0 & 0 & 1_{2n} & 0\\
	0 & 0 & 0 & 1_{2n(k-1)}\\
	-1_{2n(k-1)} & 0 & 0 & 0\\
	0 & -1_{2n} & -1_{2n} & 0
	\end{smallmatrix}\right].
	\]
	Let
	\[
	\label{3.1.6}
	\langle\phi_1,\phi_2\rangle=\int_{G(F)\backslash G(\mathbb{A})}\phi_1(g)\overline{\phi_2(g)}dg
	\]
	be the standard inner product on $G(\mathbb{A})$. 
	
	The basic properties of the global integral $Z(s,\phi_1,\phi_2,f_{2n,k,s})$ are summarized below.
	
	\begin{thm}\cite[Theorem 1]{CaiFriedbergGinzburgKaplan2019}
		The integral $Z(s,\phi_1,\phi_2,f_{2n,k,s})$ is absolutely convergent for $\mathrm{Re}(s)\gg0$ and admits meromorphic continuation to the complex plane. For $\mathrm{Re}(s)\gg 0$, it unfolds to 
		\begin{equation}
			\label{3.1.7}
			\int_{G(\mathbb{A})}\int_{U_0^\prime(\mathbb{A})}\langle\phi_1,\pi(g)\phi_2\rangle f_{\mathcal{W}(\tau,2n,\psi_{(2n)^k}),s}(\delta u_0(1\times {^{\iota}g}))\psi_U(u_0)du_0dg,
		\end{equation}
		where
		\begin{equation}
			\label{3.1.8}
			f_{\mathcal{W}(\tau,2n,\psi_{(2n)^k}),s}(h)=\int_{U_{(2n)^k}(F)\backslash U_{(2n)^k}(\mathbb{A})}f_{2n,k,s}(vh)\psi^{-1}_U(v)dv.
		\end{equation}
	\end{thm}
	
	If we assume $\phi_1$, $\phi_2$ are decomposable, then we can write $\langle \phi_1, \pi(g)\phi_2\rangle=\prod_v \omega_{\pi_v^\vee}(g)$, where $\omega_{\pi_v^\vee}$ is the matrix coefficient of $\pi_v^{\vee}$ for all $v$. If the section $f_{2n,k,s}$ is decomposable, then by Proposition \ref{eulerproduct}, we can write
	\begin{equation*}
		f_{\mathcal{W}(\tau,2n,\psi_{(2n)^k}),s}(h)=\prod_v f_{\mathcal{W}(\tau_v,2n,\psi_{(2n)^k}),s}(h_v).
	\end{equation*}
	where $f_{\mathcal{W}(\tau_v,2n,\psi_{(2n)^k}),s}\in\mathrm{Ind}_{P_{2kn}(F_v)}^{\mathrm{Sp}_{4kn}(F_v)}(\mathcal{W}(\tau_v,2n,\psi_{(2n)^k})|\det\cdot|^s)$.
	Thus the integral~\eqref{3.1.7} is Eulerian.

We now state the unramified computation from \cite{CaiFriedbergGinzburgKaplan2019}. Let $v$ be a finite place of $F$ with residue cardinality $q_v$, $\mathcal{O}_v$ be its ring of integers, $\varpi\in \mathcal{O}_v$ be a uniformizer. Let $|\cdot|$ be the absolute value on $F_v$ normalized such that $|\varpi|=q^{-1}$. Let $\psi$ be a nontrivial unramified character of $F_v$. We choose a Haar measure on $F_v$ which is self-dual with respect to $\psi$, and in particular it assigns the volume 1 to $\mathcal{O}_v$.

	\begin{thm}\cite[Theorem 29]{CaiFriedbergGinzburgKaplan2019}\label{thm29CFGK}
		Let $v$ be a finite place such that $\pi_v$ and $\tau_v$ are unramified. Assume the character $\psi$ is unramified. Let $\omega_{\pi^{\vee}_v}^0$ be the unramified matrix coefficient of $\pi^{\vee}_v$ normalized such that $\omega_{\pi_v^{\vee}}^0(1_{2n})=1$. Let 
		\[
		f^0_{\mathcal{W}(\tau_v,2n,\psi_{(2n)^k}),s}\in\mathrm{Ind}_{P_{2kn}(F_v)}^{\mathrm{Sp}_{4kn}(F_v)}(\mathcal{W}(\tau_v,2n,\psi_{(2n)^k})|\det\cdot|^s)
		\]
		be the unramified section normalized such that
		$f^0_{\mathcal{W}(\tau_v,2n,\psi_{(2n)^k}),s}(1_{4kn})=d_{\tau_v}^{\mathrm{Sp}_{4kn}}(s)
		$ (and hence $f^0_{\mathcal{W}(\tau_v,2n,\psi_{(2n)^k}),s}$ is not standard).
		Then
		\begin{equation}
			\label{3.2.1}
			\begin{aligned}
				\int_{G(F_v)}\int_{U_0^\prime(F_v)}\omega_{\pi^{\vee}_v}^0(g)f^{0}_{\mathcal{W}(\tau_v,2n,\psi_{(2n)^k}),s}(\delta u_0(1\times {^{\iota}g}))\psi_U(u_0)du_0dg
				=L(s+\frac{1}{2},\pi_v\times\tau_v).
			\end{aligned}
		\end{equation}
		Here the measure $dg$ on $G(F_v)$ is normalized such that $G(\mathcal{O}_v)$ has volume $1$ and the measure $du_0$ on $U_0^\prime(F_v)$ is taken to be a product measure, where each root subgroup is isomorphic to $F_v$ and we take the the self-dual measure on them with respect to $\psi$.
	\end{thm}

	\subsection{New integrals derived from the generalized doubling method}
	
	In this subsection, we extend an argument of Ginzburg and Soudry in \cite[Section 4.2]{GinzburgSoudry2020} and explain how to derive new integrals for $\mathrm{Sp}_{2n}\times\mathrm{GL}_k$. We assume that $n$ is even. For $g\in\mathrm{Sp}_{2n}(\mathbb{A})$, let
	\[
	\xi(\phi,f_{2n,k,s})(g)=\int_{G(F)\backslash G(\mathbb{A})}\int_{U(F)\backslash U(\mathbb{A})}\phi(h)E(u(g,h);f_{2n,k,s})\psi_U(u)dudh,
	\]
	and consider the following Fourier coefficient of $\xi(\phi,f_{2n,k,s})$:
	\begin{equation}
		\mathcal{L}(\phi,f_{2n,k,s})=\int_{\mathrm{Mat}_n^0(F)\backslash\mathrm{Mat}_n^0(\mathbb{A})}\xi(\phi,f_{2n,k,s})\left(\left[\begin{smallmatrix}
			1_n & z\\
			0 & 1_n
		\end{smallmatrix}\right]\right)\psi(\mathrm{tr}(Tz))dz.
		\label{eq-L-integral}
	\end{equation}
	Our goal is to prove the following theorem generalizing \cite[Theorem 4]{GinzburgSoudry2020} where $n=k=2$ is treated.

	\begin{thm}
		\label{thm 3.3}
		Let $n$ be a positive even integer. Given $\Phi\in\mathcal{S}(\mathrm{Mat}_{n}(\mathbb{A}))$, there are nontrivial choices of sections
		\[
		\begin{aligned}
			f_{2n,k,s}&\in\mathrm{Ind}_{P_{2kn}(\mathbb{A})}^{\mathrm{Sp}_{4kn}(\mathbb{A})}(\Delta(\tau,2n)|\det\cdot|^s),\\
			f_{n,k,s}&\in\mathrm{Ind}_{P_{kn}(\mathbb{A})}^{\mathrm{Sp}_{2kn}(\mathbb{A})}(\Delta(\tau\otimes\chi_T,n)|\det\cdot|^s),
		\end{aligned}
		\]
		such that the integral \eqref{eq-L-integral} is equal to
		\begin{equation}
			\label{3.3.11}
			\begin{aligned}
				\mathcal{L}(\phi,f_{2n,k,s})=&\int_{\mathrm{Sp}_{2n}(F)\backslash\mathrm{Sp}_{2n}(\mathbb{A})}\int_{N_{n^{k-1},kn}(F)\backslash N_{n^{k-1},kn}(\mathbb{A})}\psi_k(u)\phi(h)\\
				\times&\theta^{\Phi}_{\psi,n^2}(\alpha_T(u)i_T(1,h))E(ut(1,h);f_{n,k,s})dudh.
			\end{aligned}
		\end{equation}
		Here, $\theta^{\Phi}_{\psi,n^2}(\alpha_T(u)i_T(1,h))$ is the theta series defined in Section \ref{section-theta}, the character $\psi_k$ is given by
		\[
		\psi_k(u)=\psi\left(\sum_{i=1}^{k-2}\mathrm{tr}(u_{i,i+1})\right)
		\]
		for $u$ of the form in \eqref{2.1.1}, and $t(1,h)=t(1_n,h)=\mathrm{diag}(1_{(k-1)n}, h, 1_{(k-1)n})$.
	\end{thm}

	We shall always assume that $n$ is a positive even integer.	Clearly, $\mathcal{L}(\phi,f_{2n,k,s})$ equals
	\begin{equation}
		\begin{aligned}
			&\int_{G(F)\backslash G(\mathbb{A})}\int_{\mathrm{Mat}_n^0(F)\backslash\mathrm{Mat}_n^0(\mathbb{A})}\phi(h)\\
			\times&\int_{U(F)\backslash U(\mathbb{A})}E\left(u\left(\left[\begin{smallmatrix}
				1_n & z\\
				0 & 1_n
			\end{smallmatrix}\right],h\right);f_{2n,k,s}\right)\psi_U(u)\psi(\mathrm{tr}(Tz))dudzdh.
		\end{aligned}
		\label{eq-L-integral-1}
	\end{equation}
	
	We start by performing the root exchange process for the integral along $U(F)\backslash U(\mathbb{A})$ and prove the following lemma.
	
	\begin{lem}
		\label{rootexchange}
		There are nontrivial choices of smooth holomorphic sections
		\[
		\begin{aligned}
			f_{2n,k,s}&\in\mathrm{Ind}_{P_{2kn}(\mathbb{A})}^{\mathrm{Sp}_{4kn}(\mathbb{A})}(\Delta(\tau,2n)|\det\cdot|^s),\\
			f'_{2n,k,s}&\in\mathrm{Ind}_{P_{2kn}(\mathbb{A})}^{\mathrm{Sp}_{4kn}(\mathbb{A})}(\Delta(\tau,2n)|\det\cdot|^s),
		\end{aligned}
		\]
		such that 	
		\begin{equation}
			\label{3.3.1}
			\begin{aligned}
				\mathcal{L}(\phi,f_{2n,k,s})=&\int_{G(F)\backslash G(\mathbb{A})}\int_{N_{n^{k-1},kn}(F)\backslash N_{n^{k-1},kn}(\mathbb{A})}\psi_k(v)\phi(h)\\
				\times&\int_{N^0_{n^k,2kn}(F)\backslash N^0_{n^k,2kn}(\mathbb{A})}E(u\tilde{v}(1_{2n}\times h);f_{2n,k,s}')\psi_{N^0_{n^k,2kn},T}(u)dudvdh.
			\end{aligned}
		\end{equation}
		Here $N^0_{n^k,2kn}$ is the subgroup of $N_{n^k,2kn}$ containing elements of the form in \eqref{2.1.3} with $y=[\begin{array}{ccc}
			0_{n\times(k+1)n} & y_0 & y_0'
		\end{array}]$, $y_0\in\mathrm{Mat}_n, y_0'\in\mathrm{Mat}_{n,(k-2)n}$ and $\tilde{v}=\mathrm{diag}[1_{kn},v,1_{kn}]$. The character $\psi_{N^0_{n^k,2kn},T}$ is given by
		\begin{equation}
			\label{3.3.2}
			\begin{aligned}
				\psi_{N^0_{n^k,2kn},T}(v)&=\psi\left(\sum_{i=1}^{k-1}\mathrm{tr}(v_{i,i+1})-\mathrm{tr}(y_0)+\mathrm{tr}(Tz)\right),
			\end{aligned}
		\end{equation}
		for $v$ of the form in \eqref{2.1.3}. 
	\end{lem}

	\begin{proof}
		The lemma can be proved by performing the root exchange process for the integral along $U(F)\backslash U(\mathbb{A})$ and conjugating by certain Weyl elements as in the proof of \cite[Theorem 4]{GinzburgSoudry2020} for $n=k=2$. We explain how one extend the computations there to $k=3$ and any $n$. The general case is the same and we omit the lengthy computations.

 We view $12n\times 12n$ matrices as $12\times 12$ block matrices where each block is of size $n\times n$. Let $e_{i,j}$ be the elementary matrix which has one at the $(i,j)$ entry. Let 
		\[
		\begin{aligned}
			e_{i,j}'&=e_{i,j}-e_{13-j,13-i}, & 1\leq i,j\leq 6,\\
			e_{i,j}'&=e_{i,j}+e_{13-j,13-i}, & 1\leq i\leq 6,j>6.
		\end{aligned}
		\]
		Define
		\[
		\begin{aligned}
			&X_1=\{1+x_{1,2}e_{1,2}':x_{1,2}\in\mathrm{Mat}_n\}, &X_2=\{1+x_{3,4}e_{3,4}':x_{3,4}\in\mathrm{Mat}_n\},\\
			&Y_1=\{1+x_{2,3}e_{2,3}':y_{2,3}\in\mathrm{Mat}_n\}, &Y_2=\{1+x_{4,5}e_{4,5}':y_{4,5}\in\mathrm{Mat}_n\}.
		\end{aligned}
		\]
		Let $\widetilde{f}_{2n,k,s}\in\mathrm{Ind}_{P_{2kn}(\mathbb{A})}^{\mathrm{Sp}_{4kn}(\mathbb{A})}(\Delta(\tau,2n)|\det\cdot|^s)$ be a smooth, holomorphic section and $\xi_1\in C_c^{\infty}(X_1(\mathbb{A}))$, $\xi_2\in C_c^{\infty}(X_2(\mathbb{A}))$. We take
\begin{equation}
	\label{choosef1}
		f_{2n,k,s}(g)=\int_{X_2(\mathbb{A})}\int_{X_1(\mathbb{A})}\xi_2(x_2)\xi_1(x_1)\widetilde{f}_{2n,k,s}(gx_2x_1)dx_1dx_2,\qquad g\in\mathrm{Sp}_{4kn}(\mathbb{A}).
\end{equation}
		Take $B_1=U$ and $C_1$ the subgroup of $B_1$ generated by root subgroups in $U$ that do not lie in $Y_1$ so that $B_1=C_1Y_1$. Set $D_1=C_1X_1$ and perform the root exchange process for $(B_1,C_1,D_1,X_1,Y_1)$ by \cite[Lemma 1]{GinzburgSoudry2020}. Then let $B_2=D_1$ and similarly define $C_2$ so that $B_2=C_2Y_2$. Set $D_2=C_2X_2$ and perform the root exchange process for $(B_2,C_2,D_2,X_2,Y_2)$ again. Conjugating by the Weyl element 
		$w=\left[\begin{smallmatrix}
			w_0 & 0\\
			0 & \hat{w}_0
		\end{smallmatrix}\right]$ where
		\[
		w_0=\left[\begin{smallmatrix}
			1_n & 0 & 0 & 0 & 0 & 0\\
			0 & 0 & 1_n & 0 & 0 & 0\\
			0 & 1_n & 0 & 0 & 0 & 0\\
			0 & 0 & 0 & 0 & 1_n & 0\\
			0 & 0 & 0 & 1_n & 0 & 0\\
			0 & 0 & 0 & 0 & 0 & 1_n
		\end{smallmatrix}\right],
		\]
		the inner integral over $U(F)\backslash U(\mathbb{A})$ in \eqref{eq-L-integral-1} can be written as
		\[
		\begin{aligned}
			&\int_{U(F)\backslash U(\mathbb{A})}E\left(u\left(\left[\begin{smallmatrix}
				1_n & z\\
				0 & 1_n
			\end{smallmatrix}\right],h\right);f_{2n,k,s}\right)\psi_U(u)\\
			=&\int_{D_2'(F)\backslash D_2'(\mathbb{A})}E\left(uw\left(\left[\begin{smallmatrix}
			1_n & z\\
			0 & 1_n
		\end{smallmatrix}\right]\times h\right)w^{-1};f''_{2n,k,s}\right)\psi_{D_2'}(u)du,
		\end{aligned}
		\]
		where $D_2'$ consists of elements of the form
		\[
		\left[\begin{smallmatrix}
			1_n & u_{1,2} & \ast & \ast & \ast & \ast & \ast & \ast & \ast & \ast & \ast \\
			0 & 1_n & 0 & u_{2,4} & \ast & \ast & \ast & \ast & \ast & \ast & \ast \\
			0 & 0 & 1_n & \ast & u_{3,5}& \ast & \ast & \ast & \ast & \ast & \ast \\
			0 & 0 & 0 & 1_n & 0 & 0 & y_0 & 0 & \ast &\ast & \ast\\
			0 & 0 & 0 & 0 & 1_n & \ast & \ast & -y_0^{\ast} & \ast & \ast & \ast\\
			0 & 0 & 0 & 0 & 0 & 1_{2n} & \ast & 0 & \ast & \ast & \ast\\
			0 & 0 & 0 & 0 & 0 & 0 & 1_n & 0 & -u_{3,5}^{\ast} & \ast & \ast\\
			0 & 0 & 0 & 0 & 0 & 0 & 0 & 1_n & \ast & -u_{2,4}^{\ast} & \ast\\
			0 & 0 & 0 & 0 & 0 & 0 & 0 & 0 & 1_n & 0 & \ast\\
			0 & 0 & 0 & 0 & 0 & 0 & 0 & 0 & 0 & 1_n & -u_{1,2}^{\ast}\\
			0 & 0 & 0 & 0 & 0 & 0 & 0 & 0 & 0 & 0 & 1_n
		\end{smallmatrix}\right]
		\]
		and 
		\[
		\psi_{D_2'}(u)=\psi\left(\mathrm{tr}(u_{1,2}+u_{2,4}+u_{3,5}-y_0)\right).
		\]
	Here $f''_{2n,k,s}$ is the right $w$-translate of
	\[
\int_{Y_2(\mathbb{A})}\int_{Y_1(\mathbb{A})}\widehat{\xi}_2(y_2)\widehat{\xi}_1(y_1)\widetilde{f}_{2n,k,s}(gy_2y_1)dy_1dy_2,\qquad g\in\mathrm{Sp}_{4kn}(\mathbb{A}),
	\]
	where
	\[
	\widehat{\xi}_i(y_i)=\int_{X_i(\mathbb{A})}\xi_i(x_i)\psi_{D_i}([x_i,y_i])dy_i,\qquad i=1,2.
	\]

		We then take $B_3=D_2',X_3=Y_1,Y_3=X_2$ and define $C_3$ so that $B_3=C_3Y_3$. Set $D_3=C_3X_3$ and perform the root exchange process for $(B_3,C_3,D_3,X_3,Y_3)$. Then by conjugating by the Weyl element $w'=\left[\begin{smallmatrix}
			w_0' & 0\\
			0 & \hat{w}_0'
		\end{smallmatrix}\right]$ where
		\[
		w_0'=\left[\begin{smallmatrix}
			1_n & 0 & 0 & 0 & 0 & 0\\
			0 & 1_n & 0 & 0 & 0 & 0\\
			0 & 0 & 0 & 1_n & 0 & 0\\
			0 & 0 & 1_n & 0 & 0 & 0\\
			0 & 0 & 0 & 0 & 1_n & 0\\
			0 & 0 & 0 & 0 & 0 & 1_n
		\end{smallmatrix}\right],
		\]
		we see that the integral over $U(F)\backslash U(\mathbb{A})$ becomes the form
		\[
		\begin{aligned}
			&\int_{U(F)\backslash U(\mathbb{A})}E\left(u\left(\left[\begin{smallmatrix}
				1_n & z\\
				0 & 1_n
			\end{smallmatrix}\right],h\right);f_{2n,k,s}\right)\psi_U(u)\\
		=&\int_{D_3'(F)\backslash D_3'(\mathbb{A})}E\left(uw'w\left(\left[\begin{smallmatrix}
			1_n & z\\
			0 & 1_n
		\end{smallmatrix}\right]\times h\right)w^{-1}w'^{-1};f_{2n,k,s}'\right)\psi_{D_3'}(u)du,
		\end{aligned}
		\]
		where $D_3'$ consists of elements of the form
		\[
		\left[\begin{smallmatrix}
			1_n & u_{1,2} & \ast & \ast & \ast & \ast & \ast & \ast & \ast & \ast & \ast \\
			0 & 1_n & u_{2,3} & \ast & \ast & \ast & \ast & \ast & \ast & \ast & \ast \\
			0 & 0 & 1_n &0 & 0&0& y_0 & \ast & 0 & \ast & \ast \\
			0 & 0 & 0 & 1_n & u_{4,5} & \ast & \ast & \ast & \ast & \ast & \ast\\
			0 & 0 & 0 & 0 & 1_n & \ast & \ast & \ast & -y_0^{\ast} & \ast & \ast\\
			0 & 0 & 0 & 0 & 0 & 1_{2n} & \ast & \ast & 0 & \ast & \ast\\
			0 & 0 & 0 & 0 & 0 & 0 & 1_n & -u_{4,5}^{\ast} & 0 & \ast & \ast\\
			0 & 0 & 0 & 0 & 0 & 0 & 0 & 1_n & 0 & \ast & \ast\\
			0 & 0 & 0 & 0 & 0 & 0 & 0 & 0 & 1_n & -u_{2,3}^{\ast} & \ast\\
			0 & 0 & 0 & 0 & 0 & 0 & 0 & 0 & 0 & 1_n & -u_{1,2}^{\ast}\\
			0 & 0 & 0 & 0 & 0 & 0 & 0 & 0 & 0 & 0 & 1_n
		\end{smallmatrix}\right]
		\]
		and 
		\[
		\psi_{D_3'}(u)=\psi\left(\mathrm{tr}(u_{1,2}+u_{2,3}+u_{4,5}-y_0)\right).
		\]
		Here $f'_{2n,k,s}$ is the right $w'w$-translate of 
		\[
		\int_{X_2(\mathbb{A})}\int_{Y_2(\mathbb{A})}\widehat{\xi}_2(y_2)\widehat{\widehat{\xi_1}}(x_2)\widetilde{f}_{2n,k,s}(gy_2x_2)dy_2dx_2,
		\]
		where
		\[
		\widehat{\widehat{\xi_1}}(x_2)=\int_{Y_1(\mathbb{A})}\widehat{\xi}_1(y_1)\psi_{D_3}([y_1,x_2])dy_1.
		\]
		Computing $w'w\left(\left[\begin{smallmatrix}
			1_n & z\\
			0 & 1_n
		\end{smallmatrix}\right]\times h\right)w^{-1}w'^{-1}$ one easily obtain the integral \eqref{3.3.1}. This completes the proof of the lemma with $f_{2n,k,s}$ and $f'_{2n,k,s}$ taken as in the above computation.
	\end{proof}
	
Write $N^0_{n^k,2kn}=U_0Y_0$ such that $U_0$ contains elements with $y_0=0,y_0'=0$ in \eqref{2.1.3} and $Y_0$ contains elements such that all entries above diagonal are zero except $y_0,y_0'$. Also denote $\psi_{U_0,T}$ and $\psi_{Y_0}$ for the restriction of $\psi_{N^0_{n^k,2kn},T}$ to $U_0$ and $Y_0$. Then the integral in the second line of \eqref{3.3.1} can be factorized as
	\begin{equation}
		\label{3.3.3}
		\int_{Y_0(F)\backslash Y_0(\mathbb{A})}\int_{U_0(F)\backslash U_0(\mathbb{A})}E(uy\tilde{v}(1_{2n}\times h);f_{2n,k,s}')\psi_{U_0,T}(u)\psi_{Y_0}(y)dudy.
	\end{equation}
	Our next step is to apply a theorem of Ikeda \cite{Ikeda1994} for the inner integral along $U_0(F)\backslash U_0(\mathbb{A})$.
	
	\begin{lem}
		\label{ikeda}
	Given $\Phi_1,\Phi_2\in\mathcal{S}(\mathrm{Mat}_{n,kn}(\mathbb{A}))$, there are nontrivial choices of smooth holomorphic sections
	\[
	\begin{aligned}
		f_{2n,k,s}&\in\mathrm{Ind}_{P_{2kn}(\mathbb{A})}^{\mathrm{Sp}_{4kn}(\mathbb{A})}(\Delta(\tau,2n)|\det\cdot|^s),\\
		\check{f}_{2n,k,s}&\in\mathrm{Ind}_{P_{2kn}(\mathbb{A})}^{\mathrm{Sp}_{4kn}(\mathbb{A})}(\Delta(\tau,2n)|\det\cdot|^s),
	\end{aligned}
	\]
		such that
		\begin{equation}
				\label{3.3.4}
				\begin{aligned}
				\mathcal{L}(\phi,f_{2n,k,s})=&\int_{G(F)\backslash G(\mathbb{A})}\int_{N_{n^{k-1},kn}(F)\backslash N_{n^{k-1},kn}(\mathbb{A})}\psi_k(v)\phi(h)\\
				\times&\int_{Y_0(F)\backslash Y_0(\mathbb{A})}\theta^{\Phi_1}_{\psi,kn^2}(l_T^0(y)i_T^0(1,v\tilde{h}))\psi_{Y_0}(y)\\
				\times&\int_{N_{n^k,2kn}(F)\backslash N_{n^k,2kn}(\mathbb{A})}\overline{\theta^{\Phi_2}_{\psi,kn^2}(l_T^0(u)i_T^0(1,v\tilde{h}))}E(u\tilde{v}(1_{2n}\times h);\check{f}_{2n,k,s})\psi_{U_0}(u)dudydvdh.
			\end{aligned}
		\end{equation}
		Here the theta series and the splitting $i_T^0:\mathrm{SO}_{T_0}\times\mathrm{Sp}_{2kn}\to\widetilde{\mathrm{Sp}}_{2kn^2}$ are defined in Section \ref{section-theta}. We write $\tilde{h}=\mathrm{diag}[1_{(k-1)n},h,1_{(k-1)n}]$ for its embedding in $\mathrm{Sp}_{2kn}$.
	\end{lem}
	
	\begin{proof}
		Let $\check{f}_{2n,k,s}\in\mathrm{Ind}_{P_{2kn}(\mathbb{A})}^{\mathrm{Sp}_{4kn}(\mathbb{A})}(\Delta(\tau,2n)|\det\cdot|^s)$ be a smooth, holomorphic section. Applying a theorem of Ikeda \cite{Ikeda1994} for the Heisenberg group $\mathcal{H}_{2kn^2+1}$ and the dual pair $(\mathrm{SO}_{T_0},\mathrm{Sp}_{2kn})$, we can construct a smooth holomorphic section $\check{f}^{\Phi_1,\Phi_2}_{2n,k,s}\in\mathrm{Ind}_{P_{2kn}(\mathbb{A})}^{\mathrm{Sp}_{4kn}(\mathbb{A})}(\Delta(\tau,2n)|\det\cdot|^s)$ in the same way as \cite[Theorem 3]{GinzburgSoudry2020} such that
		\[
		\begin{aligned}
		&\int_{U_0(F)\backslash U_0(\mathbb{A})}E(uy\tilde{v}(1_{2n}\times h);\check{f}^{\Phi_1,\Phi_2}_{2n,k,s})\psi_{U_0,T}(u)du\\
		=&\theta^{\Phi_1}_{\psi,kn^2}(l_T^0(y)i_T^0(1,v\tilde{h}))\int_{N_{n^k,2kn}(F)\backslash N_{n^k,2kn}(\mathbb{A})}\overline{\theta^{\Phi_2}_{\psi,kn^2}(l_T^0(u)i_T^0(1,v\tilde{h}))}E(u\tilde{v}(1_{2n}\times h);\check{f}_{2n,k,s})\psi_{U_0}(u)du.
		\end{aligned}
		\]
		We explain the choice of section $f_{2n,k,s}$ in the lemma for $k=3$ and any $n$. The general case is the same and will be omitted for simplicity. 
		
		We retain the notations in the proof of Lemma \ref{rootexchange}. In \eqref{choosef1}, the section $f_{2n,k,s}$ is defined from a section $\widetilde{f}_{2n,k,s}$ which we are now going to determine. We pick Schwartz functions $\xi_{1,i}\in\mathcal{S}(X_2(\mathbb{A}))$, $\xi_{2,j}\in\mathcal{S}(Y_2(\mathbb{A}))$ for $1\leq i\leq r_1$, $1\leq j\leq r_2$ and smooth holomorphic section $f^{(i,j)}_{2n,k,s}$ such that
		\[
		\check{f}^{\Phi_1,\Phi_2}_{2n,k,s}(gw'w)=\sum_{i=1}^{r_1}\sum_{j=1}^{r_2}\int_{X_2(\mathbb{A})}\int_{X_1(\mathbb{A})}\widehat{\xi}_{2,j}(y_2)\widehat{\widehat{\xi_{1}}}_{,i}(x_2)f^{(i,j)}_{2n,k,s}(gy_2x_2)dy_2dx_2.
		\]
		We then take $\widetilde{f}_{2n,k,s}$ as a sum of $f^{(i,j)}_{2n,k,s}$ and define
		\[
		f_{2n,k,s}(g)=\sum_{i=1}^{r_1}\sum_{j=1}^{r_2}\int_{X_2(\mathbb{A})}\int_{X_1(\mathbb{A})}\xi_{2,j}(x_2)\xi_{1,i}(x_1)f^{(i,j)}_{2n,k,s}(gx_2x_1)dx_1dx_2.
		\]
		The computations in the proof of Lemma \ref{rootexchange} show that $f_{2n,k,s}'$ coincides with $\check{f}^{\Phi_1,\Phi_2}_{2n,k,s}$ which justify \eqref{3.3.4}.
	\end{proof}
	
	The final step is using an identity between Eisenstein series to write the third line of \eqref{3.3.4} into an Eisenstein series. 
	

	\begin{lem}
		\label{lem 3.2}
		Given $\Phi_2\in\mathcal{S}(\mathrm{Mat}_{n,kn}(\mathbb{A}))$ and a smooth holomorphic section 	\[\check{f}_{2n,k,s}\in\mathrm{Ind}_{P_{2kn}(\mathbb{A})}^{\mathrm{Sp}_{4kn}(\mathbb{A})}(\Delta(\tau,2n)|\det\cdot|^s),\] there exists a smooth meromorphic
		\[
		\lambda(\check{f}_{2n,k,s},\Phi_2)\in\mathrm{Ind}_{P_{kn}(\mathbb{A})}^{\mathrm{Sp}_{2kn}(\mathbb{A})}(\Delta(\tau\otimes\chi_T,n)|\det\cdot|^s)
		\]
		such that, as a function of $g\in\mathrm{Sp}_{2kn}$, we have
		\begin{equation}
			\label{3.3.5}
			\begin{aligned}
				\int_{N_{n^k,2kn}(F)\backslash N_{n^k,2kn}(\mathbb{A})}\theta^{\Phi_2}_{\psi,kn^2}(l_T^0(u)i_T^0(1,g))E(\tilde{g};\check{f}_{2n,k,s})\psi_{N_{n^k,2kn}}(u)du=E(g;\lambda(\check{f}_{2n,k,s},\Phi_2)).
			\end{aligned}
		\end{equation}
		 Here we denote $\tilde{g}=\mathrm{diag}[1_{kn},g,1_{kn}]$, and for $u$ of the form in \eqref{2.1.3}, the character $\psi_{N_{n^k,2kn}}$ is given by
		\begin{equation}
			\label{3.3.6}
			\psi_{N_{n^k,2kn}}(u)=\psi\left(\sum_{i=1}^{k-1}\mathrm{tr}(u_{i,i+1})\right).
		\end{equation}
	\end{lem}
	
	\begin{proof}
	When $k=n=2$ the above identity was proved in \cite[Lemma 2]{GinzburgSoudry2020}. The general case follows from the same idea and we provide some details for completeness. 	
	 We start by unfolding the Eisenstein series on the left-hand side of \eqref{3.3.5}, assuming $\mathrm{Re}(s)$ is large enough. That is, we need to compute
		\begin{equation}
			\label{3.3.7}
			\begin{aligned}
				&\sum_{\eta\in P_{2kn}(F)\backslash\mathrm{Sp}_{4kn}(F)/N_{n^k,2kn}(F)\widehat{\mathrm{Sp}}_{2kn}(F)}\int_{N_{n^k,2kn}(F)\backslash N_{n^k,2kn}(\mathbb{A})}\theta^{\Phi_2}_{\psi,kn^2}(l_T^0(u)i_T^0(1,g))\\
				\times&\sum_{\gamma\in S_{\eta}\backslash N_{n^k,2kn}(F)\widehat{\mathrm{Sp}}_{2kn}(F)} \check{f}_{2n,k,s}(\eta\gamma\hat{g})\psi_{N_{n^k,2kn}}(u)du,
			\end{aligned}
		\end{equation}
		where $\widehat{\mathrm{Sp}}_{2kn}$ is its image in $\mathrm{Sp}_{4kn}$ under the embedding $g\mapsto\tilde{g}$ and $S_{\eta}=\eta^{-1}P_{2kn}\eta\cap N_{n^k,2kn}(F)\widehat{\mathrm{Sp}}_{2kn}(F)$ is the stabilizer of $\eta$. One shows that only the orbit represented by
		\[
		w_0=\left[\begin{smallmatrix}
			0 & 1_{kn} & 0 & 0\\
			0 & 0 & 0 & 1_{kn}\\
			-1_{kn} & 0 & 0 & 0\\
			0 & 0 & 1_{kn} &0
		\end{smallmatrix}\right]
		\]
		is nonzero. In this case $S_{w_0}=V_0^\prime\widehat{P}_{kn}$ with $V_0^\prime$ the subgroup of $N_{n^k,2kn}$ consisting of elements of the form
		\[
		\left[\begin{smallmatrix}
			1_n & \ast & \ast & 0 & \ast & 0 & 0 & 0\\
			0 & \ddots & \ast & 0 & \ast & 0 & 0 & 0\\
			0 & 0 & 1_n & 0 & b & 0 & 0 & 0\\
			0 & 0 & 0 & 1_{kn} & 0 & b' & \ast & \ast\\
			0 & 0 & 0 & 0 & 1_{kn} & 0 & 0 & 0\\
			0 & 0 & 0 & 0 & 0 & 1_n & \ast & \ast\\
			0 & 0 & 0 & 0 & 0 & 0 & \ddots & \ast\\
			0 & 0 & 0 & 0 & 0 & 0 & 0 & 1_n
		\end{smallmatrix}\right]\in \Sp_{4kn}.
		\]
        We further factor $V_0^\prime=U_1U_2$ such that $U_1$ contains elements of the above form with $b=0$ (and thus $b'=0$) and $U_2=\{u_b=\diag\left[1_{(k-1)n}, \left[ \begin{smallmatrix} 1_n &0&b&0\\ 0&1_{kn}&0&b^\prime\\ 0&0&1_{kn}&0\\ 0&0&0&1_n\end{smallmatrix}\right],1_{(k-1)n} \right]:b\in\mathrm{Mat}_{n,kn}\}$ contains elements of the above form such that all $\ast$'s are zero. Denote $N^0_{n^k,2kn}=V_0^\prime\backslash N_{n^k,2kn}$. Then our sum \eqref{3.3.7} becomes
		\begin{equation}
			\label{3.3.8}
			\begin{aligned}
				&\sum_{\gamma\in P_{kn}(F)\backslash\mathrm{Sp}_{2kn}(F)}\int_{N^0_{n^k,2kn}(\mathbb{A})}\int_{U_2(F)\backslash U_2(\mathbb{A})}\theta_{\psi,kn^2}^{\Phi_2}(l_T^0(u_bu)i_T^0(1,\gamma g))\\
				\times&\int_{U_1(F)\backslash U_1(\mathbb{A})}\check{f}_{2n,k,s}(w_0u_1u_bu\widetilde{\gamma g})\psi_1(u_1)\psi_0(u)du_1du_bdu,
			\end{aligned}
		\end{equation}
		where $\psi_0,\psi_1$ are the restriction of $\psi_{N_{n^k,2kn}}$ to $N_{n^k,2kn}^0$ and $U_1$, respectively. 
		
		Let $U_{kn,n^{k}}$ be the unipotent subgroup of $\GL_{2kn}$ containing elements of the form
		\[
		\left[\begin{smallmatrix}
			1_{kn} & x & \ast & \ast & \ast\\
			0 & 1_n & u'_{2,3} & \ast & \ast \\
			0 & 0 & \ddots & \ddots & \ast\\
			0 & 0 & 0 & 1_n & u'_{k,k+1}\\
			0 & 0 & 0 & 0 & 1_n
		\end{smallmatrix}\right]\in\mathrm{GL}_{2kn}.
		\]
		Let $U_1'$ be the subgroup of $U_{kn,n^{k}}$ containing elements of the above form with $x=0$ and $\psi_1'$ the character of $U_1'$ given by
		\[
		\psi_1'(u_1')=\psi\left(\sum_{i=1}^{k-1}\mathrm{tr}(u'_{i+1,i+2})\right),\qquad u_1'\in U_1'(\mathbb{A}).
		\]
		Note that $w_0U_1w_0^{-1}=\left\{\diag[u_1^\prime,\widehat{h_1^\prime}]:u_1^\prime\in U_1^\prime  \right\}$.
		By changing the variable $u_1\mapsto w_0u_1w_0^{-1}$, the integral over $U_1(F)\backslash U_1(\mathbb{A})$ can be written as
		\[
		\int_{U_1'(F)\backslash U_1'(\mathbb{A})}\check{f}_{2n,k,s}(\mathrm{diag}[u_1',\widehat{h_1^\prime}]w_0u_b u\widetilde{\gamma g})\psi_1'(u'_1)du'_1.
		\]
		
		We use the following identity:
		\begin{equation}
			\label{4.24}
		\begin{aligned}
			&\int_{U_1'(F)\backslash U_1'(\mathbb{A})}\check{f}_{2n,k,s}(\mathrm{diag}[u_1',\widehat{h_1^\prime}]g)\psi_1''(u_1'')du''_1,\\
		=&\int_{U_{kn,n^{k}}(F)\backslash U_{kn,n^{k}}(\mathbb{A})}\check{f}_{2n,k,s}(\mathrm{diag}[u_1',\widehat{h_1^\prime}]g)\psi_1'(u_1')du'_1.
		\end{aligned}
		\end{equation}
		Here the integral in the second line has an extra integration. This is similar to \cite[(4.25)]{GinzburgSoudry2020} where $k=n=2$ is considered. The proof is exactly the same as there and the proof of \cite[Proposition 2.4]{GinzburgSoudry2021}. We denote by $\check{f}^{\psi_1'}_{2n,k,s}(g)$ for the integral in the second line of \eqref{4.24} and \eqref{3.3.8} becomes
		\begin{equation}
			\label{3.3.9}
			\begin{aligned}
				&\sum_{\gamma\in P_{kn}(F)\backslash\mathrm{Sp}_{2kn}(F)}\int_{N^0_{n^k,2kn}(\mathbb{A})}\psi_0(u)\\
				\times&\int_{U_2(F)\backslash U_2(\mathbb{A})}\theta_{\psi,kn^2}^{\Phi_2}(l_T^0(u_bu)i_T^0(1,\gamma g))\check{f}^{\psi_1'}_{2n,k,s}(w_0u_bu\widetilde{\gamma g})du_bdu.
			\end{aligned}
		\end{equation}
		Changing variables $u_b\mapsto w_0^{-1}u_bw_0$, the integral along $U_2(F)\backslash U_2(\mathbb{A})$ in the second line of \eqref{3.3.9}  equals
		\[
		\int_{U_2(F)\backslash U_2(\mathbb{A})}\theta_{\psi,kn^2}^{\Phi_2}(l_T^0(u_bu)i_T^0(1,\gamma g))\check{f}^{\psi_1'}_{2n,k,s}(w_0u\widetilde{\gamma g})du_b.
		\]
		Unfolding the theta series and using \eqref{weil3}, this becomes
		\[
		\omega_{\psi,kn^2}(l_T^0(u)i_T^0(1,\gamma g))\Phi(0)\check{f}^{\psi_1'}_{2n,k,s}(w_0u\widetilde{\gamma g}).
		\]
		Define
		\begin{equation}
			\label{3.3.10}
			\lambda(\check{f}_{2n,k,s},\Phi_2)(g)=\int_{N^0_{n^k,2kn}(\mathbb{A})}\omega_{\psi,kn^2}(l_T^0(u)i_T^0(1,g))\Phi_2(0)\check{f}^{\psi_1'}_{2n,k,s}(w_0u\widetilde{g})\psi_0(u)du,
		\end{equation}
		which is absolutely convergent for $\mathrm{Re}(s)\gg0$ and can be meromorphically continued to the whole complex plane. One checks that this is a section of the induced representation $\mathrm{Ind}^{\mathrm{Sp}_{2kn}(\mathbb{A})}_{P_{kn}(\mathbb{A})}(\Delta(\tau\otimes\chi_T,n)|\det\cdot|^s)$ and hence the expression \eqref{3.3.9} equals
		\[
		\sum_{\gamma\in P_{kn}(F)\backslash\mathrm{Sp}_{2kn}(F)}\lambda(\check{f}_{2n,k,s},\Phi_2)(\gamma g),
		\]
		which is an Eisenstein series as desired.
	\end{proof}

	We now complete the proof of Theorem \ref{thm 3.3}. By Lemma~\ref{lem 3.2}, the third line in \eqref{3.3.4} is an Eisenstein series and we have
	\begin{equation}
		\label{Lformula}
		\begin{aligned}
		\mathcal{L}(\phi,f_{2n,k,s})=&\int_{G(F)\backslash G(\mathbb{A})}\int_{N_{n^{k-1},kn}(F)\backslash N_{n^{k-1},kn}(\mathbb{A})}\psi_k(v)\phi(h)\\
		\times&\int_{Y_0(F)\backslash Y_0(\mathbb{A})}\theta^{\Phi_1}_{\psi,kn^2}(l_T^0(y)i_T^0(1,v\tilde{h}))E(\tilde{v}(1_{2n}\times h);f_{n,k,s})\psi_{Y_0}(y)dydvdh,
	\end{aligned}
	\end{equation}
	for
	\[
	f_{n,k,s}:=\lambda(\check{f}_{2n,k,s},\Phi_2).
	\]
	Unfolding the theta series, the second line of \eqref{Lformula} is 
	\begin{equation}
		\label{secondline}
	\int_{Y_0(F)\backslash Y_0(\mathbb{A})}\sum_{\xi\in\mathrm{Mat}_{n,kn}}\omega_{\psi,kn^2}^{\Phi_1}(l_T^0(y)i_T^0(1,v\tilde{h}))\Phi_1(\xi)E(\tilde{v}(1_{2n}\times h);f_{n,k,s})\psi_{Y_0}(y)dy.
	\end{equation}
	Recall that $y$ is of the form $\left[\begin{array}{cccc}
		0_{n,kn} & 0_n & y_0 & y_0'
	\end{array}\right]$ with $y_0\in\mathrm{Mat}_n, y_0'\in\mathrm{Mat}_{n,(k-2)n}$. Write $\xi=[\begin{array}{ccc}
		\xi_1 & \xi_2 & \xi_3
	\end{array}]$ with $\xi_1\in\mathrm{Mat}_{n,(k-2)n}, \xi_2\in\mathrm{Mat}_n, \xi_3\in\mathrm{Mat}_n$. Note that by \eqref{weil3},
	\[
	\omega_{\psi,kn^2}^{\Phi_1}(l_T^0(y)i_T^0(1,v\tilde{h}))\Phi_1(\xi)=	\psi(\mathrm{tr}(2T(\xi_1y_0'^{\ast}+\xi_2y_0^{\ast})))\omega_{\psi,kn^2}^{\Phi_1}(i_T^0(1,v\tilde{h}))\Phi_1(\xi).
	\]
	Thus the integral \eqref{secondline} is non-vanishing unless $\xi_1=0,\xi_2=(2T)^{-1}$ and one obtains a theta series $\theta_{\psi,n^2}^{\Phi}$ with $\Phi(\xi)=\Phi_1([\begin{array}{ccc}
		0 & (2T)^{-1} & \xi
	\end{array}])$. The completes the proof of the theorem.

\section{The global zeta integral and statement of theorems}
\label{section-global-zeta-integral}

The goal of the rest of the paper is to study the integral in Theorem~\ref{thm 3.3} derived from the generalized doubling constructions.

Let $F$ be a number field with the ring of ad\`{e}les $\mathbb{A}$. We keep assuming $n,k\geq 2$ are integers with $n$ even. We fix a nontrivial additive character $\psi:F\backslash\mathbb{A}\to\C^{\times}$ and let $T_0,T,\chi_T$ be as in Section \ref{section-theta}.  Let $(\pi,V_{\pi})$ be an irreducible cuspidal automorphic representation of $\mathrm{Sp}_{2n}(\mathbb{A})$ and $\phi\in V_{\pi}$ be a cusp form. Define the Fourier coefficient of $\phi$ by
\begin{equation}
\label{4.1}
\phi_{\psi,T}(h)=\int_{\mathrm{Mat}_n^0(F)\backslash\mathrm{Mat}_n^0(\mathbb{A})}\phi\left(\left[\begin{smallmatrix}
1_n & z\\
0 & 1_n
\end{smallmatrix}\right]h\right)\psi(\mathrm{tr}(Tz))dz.
\end{equation}
We always assume $T$ is chosen such that $\phi_{\psi,T}\neq 0$ which is possible by \cite{LiJian-Shu1992} (we remark that for any nonzero $\phi$, there is a $T$ of maximal rank such that $\phi_{\psi,T}\neq 0$). In general, the models on $\pi$ corresponding to \eqref{4.1} are not unique, and they are the same models considered in \cite{Piatetski-ShapiroRallis1988}. 

Let $\theta_{\psi}^{\Phi}:=\theta_{\psi,n^2}^{\Phi}$ be the theta series associated with the Weil representation $\omega_{\psi,n^2}$ and the Schwartz function $\Phi\in\mathcal{S}(\mathrm{Mat}_n(\mathbb{A}))$ defined in \eqref{2.2.1}. Let
\[
f_s:=f_{n,k,s}\in\mathrm{Ind}_{P_{kn}(\mathbb{A})}^{\mathrm{Sp}_{2kn}(\mathbb{A})}(\Delta(\tau\otimes\chi_T,n)|\det\cdot|^s)
\]
be a section obtained as in Lemma \ref{lem 3.2} and we form the Eisenstein series $E(g;f_s)$. For $u\in N_{n^{k-1},kn}$ of the form in \eqref{2.1.1} we define a character 
\begin{equation}
\label{4.2}
\begin{aligned}
\psi_k:N_{n^{k-1},kn}(F)\backslash N_{n^{k-1},kn}(\mathbb{A})&\to\C^{\times},\\
u&\mapsto\psi\left(\sum_{i=1}^{k-2}\mathrm{tr}(2Tu_{i,i+1})\right).
\end{aligned}
\end{equation} 

Recall that the global integral in \eqref{3.3.11} is
\begin{equation}
\label{4.3}
\begin{aligned}
\mathcal{Z}(\phi,\theta_{\psi,n^2}^{\Phi},f_s) :=&\int_{\mathrm{Sp}_{2n}(F)\backslash \mathrm{Sp}_{2n}(\mathbb{A})}\int_{N_{n^{k-1},kn}(F)\backslash N_{n^{k-1},kn}(\mathbb{A})}\phi(h)\\
\times&\theta_{\psi,n^2}(\alpha_T(u)i_T(1,h))E(ut(1,h);f_s)\psi_k(u)dudh.
\end{aligned}
\end{equation}

Now we state the basic properties of the integral $\mathcal{Z}(\phi,\theta_{\psi,n^2}^{\Phi},E(\cdot, f_s))$.
Let $N^0_{n^{k-1},kn}$ be the subgroup of $N_{n^{k-1},kn}$ containing elements of the form
\begin{equation}
\label{4.5}
\left[\begin{smallmatrix}
1_{(k-1)n} & \ast & 0 & \ast\\
0 & 1_n & 0 & 0\\
0 & 0 & 1_n & \ast\\
0 & 0 & 0 & 1_{(k-1)n}
\end{smallmatrix}\right],
\end{equation}
and let
\begin{equation}
\label{4.6}
\eta=\left[\begin{smallmatrix}
0 & 1_n & 0 & 0\\
0 & 0 & 0 & -1_{(k-1)n}\\
1_{(k-1)n} & 0 & 0 & 0\\
0 & 0 & 1_n & 0
\end{smallmatrix}\right].
\end{equation}
We have the following.

\begin{prop}
\label{prop 4.1}
The integral $\mathcal{Z}(\phi,\theta_{\psi}^{\Phi},E(\cdot, f_s))$ converges absolutely when $\mathrm{Re}(s)\gg 0$ and can be meromorphically continued to all $s\in\C$. For $\mathrm{Re}(s)\gg0$, it unfolds to
\begin{equation}
\label{4.7}
\begin{aligned}
\int_{N_n(\mathbb{A})\backslash\mathrm{Sp}_{2n}(\mathbb{A})}\int_{N_{n^{k-1},kn}^0(\mathbb{A})}\phi_{\psi,T}(h) \omega_{\psi}(\alpha_T^k(u)i_T(1,h))\Phi(1_n)f_{\mathcal{W}(\tau\otimes\chi_T,n,\psi_{2T}),s}(\eta ut(1,h))dudh.
\end{aligned}
\end{equation}
Here,
\begin{equation}
\label{4.8}
\begin{aligned}
f_{\mathcal{W}(\tau\otimes\chi_T,n,\psi_{2T}),s}(g)&=\int_{U_{n^k}(F)\backslash U_{n^k}(\mathbb{A})}f_s\left(\left[\begin{smallmatrix}
u & 0\\
0 & \hat{u}
\end{smallmatrix}\right]g\right)\psi_{2T}^{-1}(u)du,
\end{aligned}
\end{equation}
and
\[
\psi_{2T}(u)=\psi\left(\sum_{i=1}^{k-2}\mathrm{tr}(u_{i,i+1})+\mathrm{tr}(2Tu_{k-1,k})\right),
\]
with $u$ of the form in \eqref{2.3.1}.
\end{prop}

Proposition~\ref{prop 4.1} will be proved in  Section~\ref{section-unfolding}. Note that there is no analog of Proposition~\ref{prop 4.1} in \cite{GinzburgSoudry2020}.

We will take both $\Phi$ and $f_s$ to be decomposable and we can write (by Proposition \ref{eulerproduct})
\[
f_{\mathcal{W}(\tau\otimes\chi_T,n,\psi_{2T}),s}(g)=\prod_vf_{\mathcal{W}(\tau_v\otimes\chi_T,n,\psi_{2T}),s}(g_v)
\]
where $f_{\mathcal{W}(\tau_v\otimes\chi_T,n,\psi_{2T}),s}\in\mathrm{Ind}_{P_{kn}(F_v)}^{\mathrm{Sp}_{2kn}(F_v)}(\mathcal{W}(\tau_v\otimes\chi_T,n,\psi_{2T})|\det\cdot|^s)$. However, since the local models corresponding to \eqref{4.1} are not unique, we can not expect the integral $\mathcal{Z}(\phi,\theta_{\psi}^{\Phi},E(\cdot, f_s))$ to be Eulerian. Hence we will analyze this integral using the New Way method of Piatetski-Shapiro and Rallis, first appeared in \cite{Piatetski-ShapiroRallis1988}.

For a finite place $v$, we denote by $\mathcal{O}_{v}$ the ring of integers of $F_v$. We take $S$ to be a finite set of of places such that $v\notin S$ if and only if $v\nmid 2,3,\infty$; $\pi_v,\tau_v,\psi_v$ are unramified and $T_0\in\mathrm{Mat}_n(\mathcal{O}_{v}^{\times})$. 

For a place $v\notin S$, let $\Phi^0_v=\mathbf{1}_{\mathrm{Mat}_n(\mathcal{O}_{v})}$ be the characteristic function of $\mathrm{Mat}_n(\mathcal{O}_{v})$. Let 
\[
f^0_{\mathcal{W}(\tau_v\otimes\chi_T,n,\psi_{2T}),s}\in\mathrm{Ind}_{P_{kn}(F_v)}^{\mathrm{Sp}_{2kn}(F_v)}(\mathcal{W}(\tau_v\otimes\chi_T,n,\psi_{2T})|\det\cdot|^s)
\]
be the unramified section normalized such that 
\[
f^0_{\mathcal{W}(\tau_v\otimes\chi_T,n,\psi_{2T}),s}(1_{2kn})=d_{\tau_v}^{\mathrm{Sp}_{4kn}}(s).
\]
We also choose a global section $f_s^{\ast,S}$ such that the local components of $f_s^{\ast,S}$ at $v\not\in S$ are $f_s^0$ chosen above.

\begin{thm}
\label{thm-unramified-computation}
Let $n$ be even. For a place $v\notin S$, take $\Phi_v^0,f^{0}_{\mathcal{W}(\tau_v\otimes\chi_T,n,\psi_{2T}),s}$ as above and fix a non-zero unramified vector $v_0\in V_{\pi_v}$. Let $l_T:V_{\pi_v}\to\C$ be a linear functional on $V_{\pi_v}$ such that
\begin{equation}
\label{4.9}
l_T\left(\pi_v\left[\begin{smallmatrix}
1_n & z\\
0 & 1_n
\end{smallmatrix}\right]\xi\right)=\psi^{-1}(\mathrm{tr}(Tz))l_T(\xi)
\end{equation}
for all $\xi\in V_{\pi_v},z\in\mathrm{Mat}_n^0(F_v)$. Denote
\begin{equation}
\label{4.10}
\begin{aligned}
\mathcal{Z}_v^{\ast}(l_T,s)&=\int_{N_n(F_v)\backslash\mathrm{Sp}_{2n}(F_v)}\int_{N_{n^{k-1},kn}^0(F_v)}l_T(\pi_v(h)v_0)\\
\times&\omega_{\psi,v}(\alpha_T^k(u)i_T(1,h))\Phi_v^0(1_n)f^{0}_{\mathcal{W}(\tau\otimes\chi_T,n,\psi_{2T}),s}(\eta ut(1,h))dudh.
\end{aligned}
\end{equation}
Then for $\mathrm{Re}(s)\gg0$ we have
\begin{equation}
\label{4.11}
\mathcal{Z}_v^{\ast}(l_T,s)=L(s+\frac{1}{2},\pi_v\times\tau_v)\cdot l_T(v_0).
\end{equation}
\end{thm}

The unramified computations and the proof of the above theorem will be carried out in Section~\ref{section-unramified-computation}. 

By applying the strategy of \cite{Piatetski-ShapiroRallis1988} (see also the proofs of \cite[Corollary 3.4]{PollackShah2017}, \cite[Theorem 3.4]{Yan2021}) we obtain the following result.

\begin{thm}
\label{thm-main-restated}
Let $n$ be even. Fix an isomorphism $\pi\cong\otimes_v'\pi_v$ and identify $\phi\in V_{\pi}$ with $\otimes_v\xi_v$, for $\xi_v\in V_{\pi_v}$. For $\mathrm{Re}(s)\gg 0$, we have
\begin{equation}
\label{4.12}
\mathcal{Z}(\phi,\theta_{\psi}^{\Phi}, f^{\ast,S}_s)=L^S(s+\frac{1}{2},\pi\times\tau)\cdot\mathcal{Z}_S(\phi,\Phi_S,f_{S,s}),
\end{equation}
with
\begin{equation}
\label{4.13}
\begin{aligned}
\mathcal{Z}_S(\phi,\Phi_S, f_{S,s})&=\int_{N_n(\mathbb{A}_S)\backslash\mathrm{Sp}_{2n}(\mathbb{A}_S)}\int_{N_{n^{k-1},kn}^0(\mathbb{A}_S)}\phi_{\psi,T}(h)\\
\times&\omega_{\psi,S}(\alpha_T^k(u)i_T(1,h))\Phi_S(1_n)f_{S,s}(\eta ut(1,h))dudh,
\end{aligned}
\end{equation}
where $\omega_{\psi, S}=\otimes_{v\in S}\omega_{\psi, v}$, $\Phi_S\in \otimes_{v\in S}\mathcal{S}(\mathrm{Mat}_n(F_v))$, and $f_{S,s}\in \otimes_{v\in S}\mathrm{Ind}_{P_{kn}(F_v)}^{\mathrm{Sp}_{2kn}(F_v)}(\mathcal{W}(\tau_v\otimes\chi_T,n,\psi_{2T})|\det\cdot|^s)$.
\end{thm}

The local zeta integral at finite ramified places and archimedean places can be controlled by the following two propositions.

\begin{prop}
Let $v$ be a finite place and $K_0$ be an open compact subgroup of $\mathrm{Sp}_{2n}(F_v)$. There is a choice of $\Phi_0\in\mathcal{S}(\mathrm{Mat}_n(F_v))$ and $f_{0,s}\in\mathrm{Ind}_{P_{kn}(F_v)}^{\mathrm{Sp}_{2kn}(F_v)}(\mathcal{W}(\tau_v\otimes\chi_T,n,\psi_{2T})|\det\cdot|^s)$ such that for any irreducible admissible representation $\pi_v$ of $\mathrm{Sp}_{2n}(F_v)$, any vector $\xi_0\in V_{\pi_v}$ stable under $K_0$ and any linear functional $l_T:V_{\pi_v}\to\C$ satisfying \eqref{4.9} we have
\begin{equation}
\label{4.14}
\begin{aligned}
\int_{N_n(F_v)\backslash\mathrm{Sp}_{2n}(F_v)}\int_{N_{n^{k-1},kn}^0(F_v)}l_T(\pi_v(h)\xi_0)
 \omega_{\psi,v}(\alpha_T^k(u)i_T(1,h))\Phi_0(1_n)f_{0,s}(\eta ut(1,h))dudh=l_T(\xi_0).
\end{aligned}
\end{equation}
\label{prop-finite-non-vanishing}
\end{prop}

\begin{proof}
The proof is similar to other situations appeared elsewhere, see \cite[Proposition 6.6]{GinzburgRallisSoudry1998} for one example. We provide some details below following the same argument as in \emph{loc.cit}.

Let $\overline{N}_n(F_v)=\left\{ \overline{n}(w)=
\left[\begin{smallmatrix}
1_n & \\ w & 1_n
\end{smallmatrix}\right]: w\in \mathrm{Mat}_n^0(F_v)
  \right\}$. Then
$N_n(F_v)M_n(F_v)\overline{N}_n(F_v)$ is an open dense subset of $\mathrm{Sp}_{2n}(F_v)$ whose complement has Haar measure 0. Recall that an element $u\in N_{n^{k-1},kn}^0(F_v)$ has the form
\begin{equation}
\label{eq-local-ramified-unipotent}
u=\left[\begin{smallmatrix}
1_{(k-1)n} & A & 0& B\\
0 & 1_n & 0 & 0\\
0 & 0 & 1_n & A^\prime\\
0 & 0 & 0 & 1_{(k-1)n}
\end{smallmatrix}\right],
\end{equation}
with $A^\prime=-J_n {}^t A J_{(k-1)n}$. 
By applying \eqref{weil1} and \eqref{weil2}, it follows that the left-hand side of \eqref{4.14} is equal to
\begin{equation*}
\begin{split}
   &\int_{\mathrm{GL}_n(F_v)} \int_{\mathrm{Mat}_n^0(F_v)}    \int_{N_{n^{k-1},kn}^0(F_v)}  l_T\left(\pi_v\left( \left[\begin{smallmatrix} a & \\ & a^*\end{smallmatrix}\right]  \overline{n}(w)   \right)  \xi_0 \right)   \psi(\mathrm{tr}(TB_3)) \chi_T(\det(a)) \vert \det(a)\vert^{-\frac{3n}{2}-2} \\
    &   \omega_{\psi, v}\left( i_T(1,\overline{n}(w) \right) \Phi_0( a+A_2 a)
f_{0, s}
 \left(  
 \eta u t\left(1, \left[\begin{smallmatrix} a&\\&a^*\end{smallmatrix}\right]\left[\begin{smallmatrix} 1_n&\\w&1_n\end{smallmatrix}\right] \right) 
 \right)    du dw  da.
\end{split}	
\end{equation*}
Here, $u=u(A, B)$ is of the form \eqref{eq-local-ramified-unipotent} with  $A=\left[ \begin{smallmatrix} A_1 \\ A_2\end{smallmatrix}\right]$, $B=\left[ \begin{smallmatrix} B_1 & B_2 \\ B_3& B_1^*\end{smallmatrix}\right]$, $A_1\in \mathrm{Mat}_{(k-2)n\times n}(F_v)$, $A_2\in \mathrm{Mat}_{n}(F_v)$, $B_1\in \mathrm{Mat}_{(k-2)n\times n}(F_v)$, $B_2\in \mathrm{Mat}_{(k-2)n}^0(F_v)$, $B_3\in \mathrm{Mat}_{n}^0(F_v)$. 
Denote
\begin{equation}
\label{eq-local-ramified-unip-lower}
   \overline{N}_{n^{k-1},kn}^0(F_v)=\eta N_{n^{k-1},kn}^0(F_v) \eta^{-1}= \left\{ 
 \overline{u}(A, B)=\left[\begin{smallmatrix} 1_n &&&\\ & 1_{(k-1)n}&&\\ A & B &1_{(k-1)n} & \\ &A^\prime &&1_n\end{smallmatrix} \right]: A\in \mathrm{Mat}_{(k-1)n\times n}(F_v), B\in \mathrm{Mat}_{(k-1)n}^0(F_v)
  \right\}.
\end{equation}
Then the left-hand side of \eqref{4.14} is equal to
\begin{equation*}
\begin{split}
   \int_{\mathrm{GL}_n(F_v)} \int_{\mathrm{Mat}_n^0(F_v)}   & \int_{\overline{N}_{n^{k-1},kn}^0(F_v)}  l_T\left(\pi_v\left( \left[\begin{smallmatrix} a & \\ & a^*\end{smallmatrix}\right]  \overline{n}(w)   \right)  \xi_0 \right)   \psi^{-1}(\mathrm{tr}(TB_3)) \chi_T(\det(a)) \vert \det(a)\vert^{-\frac{3n}{2}-2} \\
      &   \omega_{\psi, v}\left( i_T(1,\overline{n}(w) \right) \Phi_0( a+A_2 a)
f_{0, s}
 \left(  
 \overline{u}  \left[ \begin{smallmatrix}a &&\\&1_{2(k-1)n} &\\ &&\hat{a} \end{smallmatrix}\right]  \left[\begin{smallmatrix} 1_n &&\\      &1_{2(k-1)n} & \\w    &&1_n\end{smallmatrix} \right]\eta    \right) 
   d\overline{u} dw  da.
\end{split}	
\end{equation*}
By conjugating the matrix $ \left[ \begin{smallmatrix}a &&\\&1_{2(k-1)n} &\\ &&\hat{a} \end{smallmatrix}\right]$ to the left of $\overline{u}$ and changing the variable $A\mapsto Aa^{-1}$, we obtain
\begin{equation}
\label{eq-finite-ramified-integral-1}
\begin{split}
   \int_{\mathrm{GL}_n(F_v)} &\int_{\mathrm{Mat}_n^0(F_v)}    \int_{\overline{N}_{n^{k-1},kn}^0(F_v)}  l_T\left(\pi_v\left( \left[\begin{smallmatrix} a & \\ & a^*\end{smallmatrix}\right]  \overline{n}(w)   \right)  \xi_0 \right)   \psi^{-1}(\mathrm{tr}(TB_3)) \chi_T(\det(a)) \vert \det(a)\vert^{-\frac{5n}{2}-2} \\
      &   \omega_{\psi, v}\left( i_T(1,\overline{n}(w) \right) \Phi_0( a+A_2)
f_{0, s}
 \left(   \left[ \begin{smallmatrix}a &&\\&1_{2(k-1)n} &\\ &&\hat{a} \end{smallmatrix}\right] 
\left[\begin{smallmatrix} 1_n &&&\\ & 1_{(k-1)n}&&\\ A & B &1_{(k-1)n} & \\ w&A^\prime &&1_n\end{smallmatrix} \right] \eta  \right) 
   d\overline{u} dw  da.
\end{split}	
\end{equation}
By writing the section $f_{0, s}$ in Jacquet's style notation (i.e., realizing it as a function $f_{0, s}:\mathrm{GL}_{kn}(F_v)\times\Sp_{2kn}(F_v)\to \mathbb{C}$), the left-hand side of \eqref{4.14} becomes
\begin{equation*}
\begin{split}
   \int_{\mathrm{GL}_n(F_v)} \int_{\mathrm{Mat}_n^0(F_v)}   & \int_{\overline{N}_{n^{k-1},kn}^0(F_v)}  l_T\left(\pi_v\left( \left[\begin{smallmatrix} a & \\ & a^*\end{smallmatrix}\right]  \overline{n}(w)   \right)  \xi_0 \right)   \psi^{-1}(\mathrm{tr}(TB_3)) \chi_T(\det(a)) \vert \det(a)\vert^{s+\frac{(k-5)n-3}{2}} \\
      &   \omega_{\psi, v}\left( i_T(1,\overline{n}(w) \right) \Phi_0( a+A_2)
f_{0, s}
 \left(   \left[ \begin{smallmatrix}a &\\&1_{(k-1)n} \end{smallmatrix}\right] ;
\left[\begin{smallmatrix} 1_n &&&\\ & 1_{(k-1)n}&&\\ A & B &1_{(k-1)n} & \\ w&A^\prime &&1_n\end{smallmatrix} \right] \eta  \right) 
   d\overline{u} dw  da.
\end{split}	
\end{equation*}
Now we choose the right $\eta$-translate of $f_{0,s}$ to have support in $P_{kn}(F_v)\cdot \mathcal{U}$, where $\mathcal{U}$ is a small neighborhood of identity, such that  $f_{0,s}(m; u \eta)=W_{0}(m)$ for $u\in \mathcal{U}$, $m\in \mathrm{GL}_{kn}(F_v)\cong M_{kn}(F_v)$. Here $W_{0}$ is a function in the $(k,2n)$ model $\mathcal{W}(\tau_v\otimes \chi_T, n, \psi_{2T})$, to be chosen later. We denote
\begin{equation*}
\begin{split}
\mathcal{V}_1  =\left\{ A:  \left[\begin{smallmatrix} 1_n &&&\\ & 1_{(k-1)n}&&\\ A &  &1_{(k-1)n} & \\ &A^\prime &&1_n\end{smallmatrix} \right]  \in \mathcal{U} \right\}, \mathcal{V}_2=\left\{ B:   \left[\begin{smallmatrix} 1_n &&&\\ & 1_{(k-1)n}&&\\  & B &1_{(k-1)n} & \\ &  &&1_n\end{smallmatrix} \right]  \in \mathcal{U} \right\},   \mathcal{V}_3  =\left\{ w:   \left[\begin{smallmatrix} 1_n &&\\ & 1_{2(k-1)n}&\\ \\ w& &1_n\end{smallmatrix} \right]  \in \mathcal{U} \right\}.
\end{split}
\end{equation*}
With this choice of $f_{0,s}$, 
$\mathcal{V}_1$ must be a small neighborhood of zero in $\mathrm{Mat}_{(k-1)n\times n}(F_v)$, $\mathcal{V}_2$ must be a small neighborhood of zero in $\mathrm{Mat}^0_{(k-1)n}(F_v)$, and $\mathcal{V}_3$ must be a small neighborhood of zero in $\mathrm{Mat}^0_{n}(F_v)$.
We choose $\mathcal{U}$ so small that $\pi(\overline{n}(w))\xi_0=\xi_0$ and $\omega_{\psi,v}\left( i_T(1,\overline{n}(w)\right) \Phi_0=\Phi_0$ for $w\in \mathcal{V}_3$, $\psi^{-1}(\mathrm{tr}(TB_3))=1$ for $B=\left[ \begin{smallmatrix} B_1 & B_2 \\ B_3& B_1^*\end{smallmatrix}\right]\in \mathcal{V}_2$, and $
    \Phi_0(a+A_2)=\omega_{\psi,v}(\alpha_T^k(u(A, 0, 0)))\Phi_0(a)=\Phi_0(a)
$ for $A=\left[ \begin{smallmatrix} A_1 \\ A_2\end{smallmatrix}\right]\in \mathcal{V}_1$. Then  the left-hand side of \eqref{4.14} is equal to
\begin{equation*}
\begin{split}
\mathrm{vol}(\mathcal{U}\cap \overline{N}_{kn}(F_v))  \cdot \int_{\mathrm{GL}_n(F_v)}   l_T\left(\pi_v \left(\left[\begin{smallmatrix} a & \\ & a^*\end{smallmatrix}\right]\right)    \xi_0 \right) \chi_T(\det(a)) \vert \det(a)\vert^{s+\frac{(k-5)n-3}{2}}  \Phi_0(a) W_0\left( \left[\begin{smallmatrix} a & \\ & 1_{(k-1)n}\end{smallmatrix}\right] \right) da.
\end{split}	
\end{equation*}
Notice that by assumption, $\xi_0$ is stabilized by a compact open subgroup $K_0$ of $\Sp_{2n}(F_v)$. Let $M_1$ be a small neighborhood of $1_{n}$ in $\mathrm{GL}_n(F_v)$ such that for any $a\in M_1$, $\left(\begin{smallmatrix} a & \\ & a^*\end{smallmatrix}\right)\in K_0$. Let $M_2$ be a small neighborhood of $1_n$ in $\mathrm{GL}_n(F)$ such that for any $a\in M_2$, $\left[\begin{smallmatrix} a & \\ & 1_{(k-1)n}\end{smallmatrix}\right]$ belongs to the stabilizer of $W_0$. Let $M_3$ be a small neighborhood of $1_n$ in $\mathrm{GL}_n(F)$ such that for any $a\in M_3$, $\det(a)$ belongs to the intersection of $\mathcal{O}_{v}^\times$ and the conductor of $\chi_T$. Let $M_0=M_1\cap M_2\cap M_3$, and let $\Phi_0$ be the characteristic function of $M_0$. Then the left-hand side of \eqref{4.14} is equal to
\begin{equation*}
\begin{split}
&\mathrm{vol}(\mathcal{U}\cap \overline{N}_{kn}(F_v))  \cdot \int_{M_0}   l_T\left(\pi_v \left[\begin{smallmatrix} a & \\ & a^*\end{smallmatrix}\right]    \xi_0 \right) \chi_T(\det(a)) \vert \det(a)\vert^{s+\frac{(k-5)n-3}{2}}  \Phi_0(a) W_0\left( \left[\begin{smallmatrix} a & \\ & 1_{(k-1)n}\end{smallmatrix}\right] \right) da\\
=& \mathrm{vol}(\mathcal{U}\cap \overline{N}_{kn}(F_v)) \cdot l_T(v_0)  \cdot \mathrm{vol}(M_0) \cdot W_0(1_{kn}).
 \end{split}
 \end{equation*}
Finally, we choose $W_0$ such that $W_0(1_{kn})=\mathrm{vol}(\mathcal{U}\cap \overline{N}_{kn}(F_v))^{-1} \cdot \mathrm{vol}(M_0)^{-1}$. With these choices, we obtain \eqref{4.14}.
\end{proof}

\begin{prop}
\label{prop-archimedean-non-vanishing}
For any complex number $s_0\in\C$, there is a choice of data $(\phi_j,\Phi_j,f_{j,s})$ such that the finite sum of archimedean integrals
\begin{equation}
\label{4.15}
\begin{aligned}
\sum_j\mathcal{Z}_{\infty}(\phi_j,\Phi_j,f_{j,s})
=&\sum_{j}\int_{N_n(\mathbb{A}_{\infty})\backslash\mathrm{Sp}_{2n}(\mathbb{A}_{\infty})}\int_{N_{n^{k-1},kn}^0(\mathbb{A}_{\infty})}\phi_{j,\psi,T}(h)\\
\times& \omega_{\psi,\infty}(\alpha_T^k(u)i_T(1,h))\Phi_{j}(1_n)f_{j,s}(\eta ut(1,h))dudh
\end{aligned}
\end{equation}
admits meromorphic continuation to the whole complex plane and its meromorphic continuation 
is holomorphic and nonzero in a neighborhood of $s_0$.
\end{prop}

\begin{proof} The proof is similar to \cite[Proposition 6.7]{GinzburgRallisSoudry1998}. 
Similar to \eqref{eq-finite-ramified-integral-1}, the archimedean integral  $\mathcal{Z}_{\infty}(\phi,\Phi_\infty,f_{\infty,s})$ is equal to  
\begin{equation*}
\begin{split}
   \int_{\mathrm{GL}_n(\A_\infty)} & \int_{\mathrm{Mat}_n^0(\A_\infty)}   \int_{\overline{N}_{n^{k-1},kn}^0(\A_\infty)}  \phi_{\psi, T}\left(  \left[\begin{smallmatrix} a & \\ & a^*\end{smallmatrix}\right]  \overline{n}(w)      \right)   \psi^{-1}(\mathrm{tr}(TB_3)) \chi_T(\det(a)) \vert \det(a)\vert^{-\frac{5n}{2}-2} \\
      &   \omega_{\psi, \infty}\left( i_T(1,\overline{n}(w) \right) \Phi_\infty( a+A_2)
f_{\infty, s}
 \left(   \left[ \begin{smallmatrix}a &&\\&1_{2(k-1)n} &\\ &&\hat{a} \end{smallmatrix}\right] 
\left[\begin{smallmatrix} 1_n &&&\\ & 1_{(k-1)n}&&\\ A & B &1_{(k-1)n} & \\ w&A^\prime &&1_n\end{smallmatrix} \right] \eta  \right) 
   d\overline{u} dw  da,
\end{split}	
\end{equation*}
where $A=\left[ \begin{smallmatrix} A_1 \\ A_2\end{smallmatrix}\right]$, $B=\left[ \begin{smallmatrix} B_1 & B_2 \\ B_3& B_1^*\end{smallmatrix}\right]$, $A_1\in \mathrm{Mat}_{(k-2)n\times n}(\A_\infty)$, $A_2\in \mathrm{Mat}_{n}(\A_\infty)$, $B_1\in \mathrm{Mat}_{(k-2)n\times n}(\A_\infty)$, $B_2\in \mathrm{Mat}_{(k-2)n}^0(\A_\infty)$, $B_3\in \mathrm{Mat}_{n}^0(\A_\infty)$. 
We choose $f_{\infty, s}$ such that the right $\eta$-translate of $f_{\infty, s}$ has support in $P_{kn}(\mathbb{A}_\infty)\cdot \overline{N}_{kn}(\mathbb{A}_\infty)$. We further assume that $\eta\cdot f_{\infty, s}\left(  \left(\begin{smallmatrix} 1_{kn} & * \\ & 1_{kn}\end{smallmatrix}\right) m(g)   \overline{u}    \right)  =\vert \det(g)\vert^{s+\frac{kn+1}{2}} \varphi(\overline{u}) W_\infty (m),$ where $\overline{u}\in \overline{N}_{kn}(\A_\infty)$, $g\in \GL_{kn}(\A_\infty)$, $W_\infty\in \mathcal{W}(\tau_\infty\otimes\chi_T,n,\psi_{2T})$, and $\varphi$ is a smooth function with compact support in $C_c^\infty(\overline{N}_{kn}(\mathbb{A}_\infty))$ given by 
$$
\varphi\left(    
\left[\begin{smallmatrix} 1_n &&&&&\\ &1_n&&&&\\&&1_{(k-2)n}&&&\\ A_1&B_1 & B_2 &1_{(k-2)n} &&\\A_2 &B_3&B_1^*&&1_n&\\ w&A_2^*&A_1^*&&&1_n\end{smallmatrix} \right]   \right)=\varphi_{11}(A_1)\varphi_{12}(A_2)\varphi_{21}(B_1)\varphi_{22}(B_2)\varphi_{23}(B_3)\varphi_3(w),
$$
where 
$\varphi_{11}\in \mathcal{C}_c^\infty(\mathrm{Mat}_{(k-2)n\times n}(\A_\infty))$, $\varphi_{12}\in \mathcal{C}_c^\infty(\mathrm{Mat}_{n}(\A_\infty))$, $\varphi_{21}\in \mathcal{C}_c^\infty(\mathrm{Mat}_{(k-2)n\times n}(\A_\infty))$, $\varphi_{22}\in \mathcal{C}_c^\infty(\mathrm{Mat}_{(k-2)n}^0(\A_\infty))$, $\varphi_{23}\in \mathcal{C}_c^\infty(\mathrm{Mat}_{n}^0(\A_\infty))$, $\varphi_3\in \mathcal{C}_c^\infty(\mathrm{Mat}_{n}^0(\A_\infty))$. Then $\mathcal{Z}_{\infty}(\phi,\Phi_\infty,f_{\infty,s})$ is equal to 
\begin{equation*}
\begin{split}
   \int_{\mathrm{GL}_n(\A_\infty)} & \int_{\mathrm{Mat}_n^0(\A_\infty)}   \int_{\overline{N}_{n^{k-1},kn}^0(\A_\infty)}   \pi\left[ \begin{smallmatrix} 1_n  & \\ w & 1_n\end{smallmatrix} \right] \phi_{\psi, T}\left(  \left[\begin{smallmatrix} a & \\ & a^*\end{smallmatrix}\right]      \right)   \psi^{-1}(\mathrm{tr}(TB_3)) \chi_T(\det(a)) \vert \det(a)\vert^{s+\frac{(k-5)n-3}{2}} \\
      &   \omega_{\psi, \infty}\left( i_T(1,\overline{n}(w) \right) \Phi_\infty( a+A_2)
W_\infty
 \left(   \left[ \begin{smallmatrix}a &\\&1_{(k-1)n}  \end{smallmatrix}\right] 
  \right) 
  \varphi_{11}(A_1)\varphi_{12}(A_2)\varphi_{21}(B_1)\varphi_{22}(B_2)\varphi_{23}(B_3)\varphi_3(w)
   d\overline{u} dw  da.
\end{split}	
\end{equation*}
The integral over $B_3$ gives a constant $\int_{\mathrm{Mat}_{n}^0(\A_\infty)} \psi^{-1}(\mathrm{tr}(TB_3)) \varphi_{23}(B_3)dB_3$ , we may choose $\varphi_{23}$ such that this constant is non-zero. Similarly, we may choose $\varphi_{11}$, $\varphi_{21}$, $\varphi_{22}$ such that the integrals $\int \varphi_{11}(A_1)dA_1$, $\int \varphi_{21}(B_1)dB_1$, $\int\varphi_{22}(B_2)dB_2$ are non-zero. Thus, up to a non-zero constant, the integral $\mathcal{Z}_{\infty}(\phi,\Phi_\infty,f_{\infty,s})$ is equal to 
\begin{equation}
\label{eq-archimedean-1}
\begin{split}
   \int_{\mathrm{GL}_n(\A_\infty)} \int_{\mathrm{Mat}_n(\A_\infty)}  & \int_{\mathrm{Mat}_n^0(\A_\infty)}     \pi\left[ \begin{smallmatrix} 1_n  & \\ w & 1_n\end{smallmatrix} \right] \phi_{\psi, T}\left(  \left[\begin{smallmatrix} a & \\ & a^*\end{smallmatrix}\right]      \right)     \chi_T(\det(a)) \vert \det(a)\vert^{s+\frac{(k-5)n-3}{2}} \\
      &   \omega_{\psi, \infty}\left( i_T(1,\overline{n}(w) \right) \Phi_\infty( a+A_2)
W_\infty
 \left(   \left[ \begin{smallmatrix}a &\\&1_{(k-1)n}  \end{smallmatrix}\right] 
  \right) 
   \varphi_{12}(A_2) \varphi_3(w)
   dw dA_2  da.
\end{split}	
\end{equation}
The inner $dw$-integral
\begin{equation}
\label{eq-archimedean-2}
\begin{split}
 \int_{\mathrm{Mat}_n^0(\A_\infty)}  \varphi_3(w)   \pi\left[ \begin{smallmatrix} 1_n  & \\ w & 1_n\end{smallmatrix} \right] \phi_{\psi, T}\otimes    \omega_{\psi, \infty}\left( i_T(1,\overline{n}(w) \right) \Phi_\infty  
   dw 
\end{split}	
\end{equation}
is a convolution of $\varphi_3$ against $\phi_{\psi, T}\otimes \Phi_\infty \in \mathcal{W}(\pi,\psi,T)\hat\otimes \omega_{\psi,\infty}$. Here, $\mathcal{W}(\pi,\psi,T)$ is the space spanned by the global Fourier coefficients of the form $\phi_{\psi, T}$, regarded as the representation space under the right translation action of $\overline{N}_{n}(\A_\infty)$, and the symbol $\hat\otimes$ represents the completed projective topological tensor product. It follows from the Dixmier-Malliavin Theorem that a linear combination of integrals of the form \eqref{eq-archimedean-2} gives a general element in the space $\mathcal{W}(\pi,\psi,T)\hat\otimes \omega_{\psi,\infty}$. Thus, up to a suitable linear combination of integrals of the form \eqref{eq-archimedean-1}, we get
\begin{equation}
\label{eq-archimedean-3}
\begin{split}
   \int_{\mathrm{GL}_n(\A_\infty)} \int_{\mathrm{Mat}_n(\A_\infty)}         \phi^\prime_{\psi, T}\left(  \left[\begin{smallmatrix} a & \\ & a^*\end{smallmatrix}\right]      \right)     \chi_T(\det(a)) \vert \det(a)\vert^{s+\frac{(k-5)n-3}{2}}    \Phi_\infty^\prime( a+A_2)
W_\infty
 \left(   \left[ \begin{smallmatrix}a &\\&1_{(k-1)n}  \end{smallmatrix}\right] 
  \right) 
   \varphi_{12}(A_2)  
    dA_2  da,
\end{split}	
\end{equation}
where $\phi^\prime_{\psi, T}\in \mathcal{W}(\pi,\psi, T)$, $\Phi_\infty^\prime\in \mathcal{S}(\mathrm{Mat}_n(\A_\infty))$. We apply the Dixmier-Malliavin Theorem again to the $dA_2$ integral, to get that, up to a linear combination, \eqref{eq-archimedean-3} becomes
\begin{equation}
\label{eq-archimedean-4}
\begin{split}
   \int_{\mathrm{GL}_n(\A_\infty)}         \phi^\prime_{\psi, T}\left(  \left[\begin{smallmatrix} a & \\ & a^*\end{smallmatrix}\right]      \right)     \chi_T(\det(a)) \vert \det(a)\vert^{s+\frac{(k-5)n-3}{2}}    \Phi_\infty^{\prime\prime}( a)
W_\infty
 \left(   \left[ \begin{smallmatrix}a &\\&1_{(k-1)n}  \end{smallmatrix}\right] 
  \right) 
     da
\end{split}	
\end{equation}
where $\Phi_\infty^{\prime\prime}\in \mathcal{S}(\mathrm{Mat}_n(\A_\infty))$. Now we choose $W_\infty\in \mathcal{W}(\tau_\infty\otimes\chi_T,n,\psi_{2T})$ such that $W_\infty \left(   \left[ \begin{smallmatrix}a &\\&1_{(k-1)n}\end{smallmatrix}\right]\right) =\varphi_4(a)$ where  $\varphi_4\in \mathcal{C}_c^\infty(\GL_n(\A_\infty))$ (c.f. \cite[Corollary 6.7]{Yan2021}). Setting $\Phi_\infty^{\prime\prime\prime}(a)=\Phi_\infty^{\prime\prime}(a)W_\infty \left(   \left[ \begin{smallmatrix}a &\\&1_{(k-1)n}\end{smallmatrix}\right]\right)$, we get
\begin{equation*}
\begin{split}
   \int_{\mathrm{GL}_n(\A_\infty)}         \phi^\prime_{\psi, T}\left(  \left[\begin{smallmatrix} a & \\ & a^*\end{smallmatrix}\right]      \right)     \chi_T(\det(a)) \vert \det(a)\vert^{s+\frac{(k-5)n-3}{2}}    \Phi_\infty^{\prime\prime\prime}( a)
     da .
\end{split}	
\end{equation*}
This is an archimedean integral considered in \cite{RallisSoudry1989}. It follows from \cite{RallisSoudry1989} that a linear combination of the archimedean integrals is holomorphic and nonzero in a neighborhood of $s_0$.
\end{proof}

\section{The basic identity}
\label{section-unfolding}
In this section, we will prove the basic identity in Proposition \ref{prop 4.1} using the standard strategy of unfolding Eisenstein series and theta series. For $\mathrm{Re}(s)\gg0$, we have
\[
\begin{aligned}
\mathcal{Z}(\phi,\theta_{\psi,n^2}^{\Phi}, f_s)=\sum_{\gamma\in P_{kn}(F)\backslash \mathrm{Sp}_{2kn}(F)/P_{n^{k-1},kn}(F)}I(\gamma),
\end{aligned}
\]
where
\[
\begin{aligned}
I(\gamma)&=\int_{\mathrm{Sp}_{2n}(F)\backslash \mathrm{Sp}_{2n}(\mathbb{A})}\int_{N_{n^{k-1},kn}(F)\backslash N_{n^{k-1},kn}(\mathbb{A})}\phi(h)\theta_{\psi}^{\Phi}(\alpha_T^k(u)i_T(1,h))\\
\times&\sum_{g\in H_{\gamma}(F)\backslash P_{n^{k-1},kn}(F)}f_s(\gamma gut(1,h))\psi_k(u)dudh,
\end{aligned}
\]
with $H_{\gamma}=\gamma^{-1}P_{kn}\gamma\cap P_{n^{k-1},kn}$ the stabilizer of the orbit represented by $\gamma$. We will then proceed by showing that $I(\gamma)=0$ for all but one terms using the following two methods:
\begin{itemize}
    \item Show that $I(\gamma)$ contains an inner integral of the form $\int_{N(F)\backslash N(\mathbb{A})}\phi(nh)dn$ for some unipotent radical $N$ of a parabolic subgroup of $\mathrm{Sp}_{2n}$. The integral then vanishes by the cuspidality of $\phi$.
    \item Show that $I(\gamma)$ contains an inner integral of the form $\int_{U(F)\backslash U(\mathbb{A})}\psi_U(u)du$ for some unipotent group $U$. The integral vanishes for a non-trivial additive character $\psi_U$ on $U(\mathbb{A})$ which is trivial on $U(F)$. 
\end{itemize}

For integers $r_i$ with $0\le r_i\le n$ for $1\le i \le k-1$, we set $\underline{r}:=(r_1,...,r_{k-1})$. For $1\leq i\leq k-1$, we set $\underline{r}^i:=(0,...,0,r_i,...,r_{k-1})$. We understand that $\underline{r}^1=\underline{r}$. 
We first describe the orbits of
\[
P_{kn}(F)\backslash \mathrm{Sp}_{2kn}(F)/P_{n^{k-1},kn}(F)
\]
and their stabilizers in the following lemma. 

\begin{lem}
\label{lemma-coset-representative}
The representatives of $P_{kn}(F)\backslash \mathrm{Sp}_{2kn}(F)/P_{n^{k-1},kn}(F)$ are given by
\[
\gamma_{\underline{r}}:=\gamma_{r_1,{\cdots},r_{k-1}}:=\left[\begin{smallmatrix}
\mu_{k-1}' & 0 & 0 & 0 & 0 & 0 & \epsilon_{k-1}'\\
0 & \ddots & 0 & 0 & 0 & \iddots & 0\\
0 & 0 & \mu_{1}' & 0 & \epsilon_{1}' & 0\\
0 & 0 & 0 & 1_{2n} & 0 & 0 & 0\\
0 & 0 & \epsilon_1 & 0 &  \mu_1 & 0 & 0\\
0 & \iddots & 0 & 0 & 0 & \ddots & 0\\
\epsilon_{k-1} & 0 & 0 & 0 & 0 & 0 & \mu_{l-1}
\end{smallmatrix}\right],
\]
where $\mu_i,\epsilon_i,\mu_i',\epsilon_i'$ are $n\times n$ matrices 
\[
\mu_i=\left[\begin{smallmatrix}
1_{r_i} & 0\\
0 & 0
\end{smallmatrix}\right],\epsilon_i=\left[\begin{smallmatrix}
0 & 0\\
1_{n-r_i} & 0
\end{smallmatrix}\right],\mu_i'=\left[\begin{smallmatrix}
0 & 0\\
0 & 1_{r_i}
\end{smallmatrix}\right],\epsilon'_i=\left[\begin{smallmatrix}
0 & -1_{n-r_i}\\
0 & 0
\end{smallmatrix}\right],
\]
with $0\leq r_i\leq n$. Denote $H_{\underline{r}}$ for the stabilizer of $\gamma_{\underline{r}}$. Then $H_{\underline{r}}=M_{\underline{r}}\ltimes N_{\underline{r}}$, where $M_{\underline{r}}$ consists of elements
\[
m(h_0,g_1,{\cdots},g_{k-1})=\mathrm{diag}[g_{k-1},{\cdots},g_1,h_0,\hat{g}_1,{\cdots},\hat{g}_{k-1}]
\]
where $h_0\in P_n$ and for $1\leq i\leq k-1$,
\[
g_i\in B^-_{r_i,n}=\left\{\left[\begin{smallmatrix}
\ast & 0_{(n-r_i)\times r_i}\\
\ast & \ast
\end{smallmatrix}\right]\in\mathrm{GL}_n\right\}.
\]
\end{lem}

\begin{proof}
Let $(V,\langle\cdot,\cdot\rangle)$ be the underlying skew-hermitian space of the group $\mathrm{Sp}_{2kn}$ with Witt decomposition $V=I\oplus I'$ into two maximal isotropic subspace so that $P_{kn}$ is the parabolic subgroup of $\mathrm{Sp}_{2kn}$ fixing $I$. The parabolic subgroup $P_{n^{k-1},kn}$ is the one fixing some flag of isotropic subspaces
\[
0\subset I_1\subset I_2\subset{\cdots}\subset I_{k-1}\subset V
\]
with $I_i\subset I$ of rank $ni$. Then the double coset $P_{kn}\backslash \mathrm{Sp}_{2kn}/P_{n^{k-1},kn}$ is parameterized by tuple $(\kappa_1,{\cdots},\kappa_{k-1})$ where $\kappa_i=\mathrm{dim}(I\gamma\cap I_i)$. One can easily pick the representatives as in the lemma and their stabilizers can be obtained by straightforward matrix computations.
\end{proof}

Denote $I_{\underline{r}}=I(\gamma_{\underline{r}})$. Then by Lemma~\ref{lemma-coset-representative},
\[
\begin{aligned}
I_{\underline{r}}=&\int_{\mathrm{Sp}_{2n}(F)\backslash \mathrm{Sp}_{2n}(\mathbb{A})}\sum_{\substack{1\leq i\leq k-1\\g_i\in B^-_{r_i,n}(F)\backslash\mathrm{GL}_n(F)\\h_0\in P_n(F)\backslash \mathrm{Sp}_{2n}(F)}}\int_{N_{n^{k-1},kn}(F)\backslash N_{n^{k-1},kn}(\mathbb{A})}\phi(h)\theta_{\psi}^{\Phi}(\alpha^k_T(u)i_T(1,h))\\
\times&\sum_{u_0\in N_{\underline{r}}(F)\backslash N_{n^{k-1},kn}(F)}f_s(\gamma_{\underline{r}}u_0m(h_0,g_1,{\cdots},g_{k-1})ut(1,h))\psi_k(u)dudh.
\end{aligned}
\]
Now we change variables $u\mapsto m(h_0,1,{\cdots},1)^{-1}um(h_0,1,{\cdots},1)$. Clearly, $\psi_k$ is preserved under this change, and note that
\[
\begin{aligned}
m(h_0,1,{\cdots},1)^{-1}u(x,y,z)m(h_0,1,{\cdots},1)&=u([x,y]h_0,z),\\
\theta^{\Phi}_{\psi}(\alpha_T^k(u([x,y]h_0,z))i_T(1,h))&=\theta^{\Phi}_{\psi}(\alpha_T^k(u(x,y,z))i_T(1,h_0h)).
\end{aligned}
\]
Then we obtain
\[
\begin{aligned}
I_{\underline{r}}=&\int_{P_n(F)\backslash\mathrm{Sp}_{2n}(\mathbb{A})}\sum_{\substack{1\leq i\leq k-1\\g_i\in B_{r_i,n}^-(F)\backslash\mathrm{GL}_n(F)}}\int_{N_{n^{k-1},kn}(F)\backslash N_{n^{k-1},kn}(\mathbb{A})}\phi(h)\theta_{\psi}^{\Phi}(\alpha^k_T(u)i_T(1,h))\\
\times&\sum_{u_0\in N_{\underline{r}}(F)\backslash N_{n^{k-1},kn}(F)}f_s(\gamma_{\underline{r}}u_0m(1,g_1,{\cdots},g_{k-1})ut(1,h))\psi_k(u)dudh.
\end{aligned}
\]

\begin{lem}
$I_{\underline{r}}=0$ unless $r_1=0$.
\end{lem}

\begin{proof}
Let $N_{n^{k-1},kn}^c$ be a normal subgroup of $N_{n^{k-1},kn}$ containing elements of the form 
\[
u(z)=\left[\begin{smallmatrix}
1_{(k-2)n} & 0 & 0 & 0 & 0 & 0\\
0 & 1_n & 0 & 0 & z & 0\\
0 & 0& 1_n & 0 & 0 & 0\\
0 & 0 & 0 & 1_n & 0 & 0\\
0 & 0 & 0 & 0 & 1_n & 0\\
0 & 0 & 0 & 0 & 0 & 1_{(k-2)n}
\end{smallmatrix}\right].
\]
Using the formula \eqref{weil1} of the Weil representation, we have
\[
\theta_{\psi}^{\Phi}(\alpha^k_T(u(z))\alpha^k_T(u')i_T(1,h))=\psi(\mathrm{tr}(Tz))\theta_{\psi}^{\Phi}(\alpha^k_T(u')i_T(1,h)),
\]
for $u(z)\in N_{n^{k-1},kn}^c(\mathbb{A}),u'\in N_{n^{k-1},kn}(\mathbb{A}),h\in \mathrm{Sp}_{2n}(\mathbb{A})$. Therefore,
\[
\begin{aligned}
I_{\underline{r}}=&\int_{P_n(F)\backslash\mathrm{Sp}_{2n}(\mathbb{A})}\sum_{\substack{1\leq i\leq k-1\\g_i\in B_{r_i,n}^-(F)\backslash \mathrm{GL}_n(F)}}\int_{N_{n^{k-1},kn}^c(\mathbb{A})N_{n^{k-1},kn}(F)\backslash N_{n^{k-1},kn}(\mathbb{A})}\\
\times&\int_{\mathrm{Mat}_n^0(F)\backslash\mathrm{Mat}_n^0(\mathbb{A})}\phi(h)\psi(\mathrm{tr}(Tz))\theta_{\psi}^{\Phi}(\alpha_T(u)i_T(1,h))\psi_k(u)\\
\times&\sum_{u_0\in N_{\underline{r}}(F)\backslash N_{n^{k-1},kn}(F)}f_s(\gamma_{\underline{r}}u_0m(1,g_1,{\cdots},g_{k-1})u(z)ut(1,h))dzdudh.
\end{aligned}
\]
Note that
\[
m(1,g_1,{\cdots},g_{k-1})^{-1}u(z)m(1,g_1,{\cdots},g_{k-1})=u(g_1^{-1}z\hat{g}).
\]
Changing variables
$z\mapsto g_1^{-1}z\hat{g}_1$ we obtain
\[
\begin{aligned}
I_{\underline{r}}=&\int_{P_n(F)\backslash \mathrm{Sp}_{2n}(\mathbb{A})}\sum_{\substack{1\leq i\leq k-1 \\g_i\in B_{r_i,n}^-(F)\backslash \mathrm{GL}_n(F)}}\int_{N_{n^{k-1},kn}^c(\mathbb{A})N_{n^{k-1},kn}(F)\backslash N_{n^{k-1},kn}(\mathbb{A})}\\
\times&\int_{\mathrm{Mat}_n^0(F)\backslash\mathrm{Mat}_n^0(\mathbb{A})}\phi(h)\psi(\mathrm{tr}(Tg_1^{-1}z\hat{g}_1))\theta_{\psi}^{\Phi}(\alpha_T(u)i_T(1,h))\psi_k(u)\\
\times&\sum_{u_0\in N_{\underline{r}}(F)\backslash N_{n^{k-1},kn}(F)}f_s(\gamma_{\underline{r}}u_0u(z)m(1,g_1,{\cdots},g_{k-1})ut(1,h))dzdudh.
\end{aligned}
\]
Write $z=\left[\begin{smallmatrix}
z_1 & z_2\\
z_3 & z_4
\end{smallmatrix}\right]$ with $z_1\in\mathrm{Mat}_{n-r_1,r_1},z_2\in\mathrm{Mat}_{n-r_1},z_3\in\mathrm{Mat}_{r_1},z_4\in\mathrm{Mat}_{r_1,n-r_1}$ and note that $\gamma_{\underline{r}}$ commutes with $u\left(\left[\begin{smallmatrix}
0 & 0\\
z_3 & 0
\end{smallmatrix}\right]\right)$. Then 
\[
\begin{aligned}
I_{\underline{r}}=&\int_{P_n(F)\backslash \mathrm{Sp}_{2n}(\mathbb{A})}\sum_{\substack{1\leq i\leq k-1\\g_i\in B_{r_i,n}^-(F)\backslash \mathrm{GL}_n(F)}}\int_{N_{n^{k-1},kn}^c(\mathbb{A})N_{n^{k-1},kn}(F)\backslash N_{n^{k-1},kn}(\mathbb{A})}\\
\times&\phi(h)\theta_{\psi}^{\Phi}(\alpha_T(u)i_T(1,h))\int_{z_1,z_2,z_4}\psi\left(\mathrm{tr}\left(Tg_1^{-1}\left[\begin{smallmatrix}
z_1 & z_2\\
0 & z_4
\end{smallmatrix}\right]\hat{g}_1\right)\right)\psi_k(u)\\
\times&\sum_{u_0\in N_{\underline{r}}(F)\backslash N_{n^{k-1},kn}(F)}f_s\left(\gamma_{\underline{r}}
u\left(\left[\begin{smallmatrix}
z_1 & z_2\\
0 & z_4
\end{smallmatrix}\right]\right)u_0m(1,g_1,{\cdots},g_{k-1})ut(1,h)\right)\\
\times&\int_{\mathrm{Mat}_{r_1}(F)\backslash\mathrm{Mat}_{r_1}(\mathbb{A})}\psi\left(\mathrm{tr}\left(Tg_1^{-1}\left[\begin{smallmatrix}
0 & 0\\
z_3 & 0
\end{smallmatrix}\right]\hat{g}_1\right)\right)dz_3dz_1dz_2dz_4dudh.
\end{aligned}
\]
Note that when $r_1>0$, the character $z_3 \mapsto \psi\left(\mathrm{tr}\left(Tg_1^{-1}\left[\begin{smallmatrix}
0 & 0\\
z_3 & 0
\end{smallmatrix}\right]\hat{g}_1\right)\right)$ is nontrivial and thus the inner integral over $z_3$ in the last line vanishes, which completes the proof of the lemma.
\end{proof}

By the above lemma, we now need to consider the integral
\[
\begin{aligned}
I_{\underline{r}^2}=&\int_{P_n(F)\backslash\mathrm{Sp}_{2n}(\mathbb{A})}\sum_{\substack{2\leq i\leq k-1\\g_i\in B_{r_i,n}^-(F)\backslash\mathrm{GL}_n(F)}}\int_{N_{n^{k-1},kn}(F)\backslash N_{n^{k-1},kn}(\mathbb{A})}\phi(h)\theta_{\psi}^{\Phi}(\alpha^k_T(u)i_T(1,h))\\
\times&\sum_{u_0\in N_{\underline{r}^2}(F)\backslash N_{n^{k-1},kn}(F)}f_s(\gamma_{\underline{r}^2}u_0m(1,1,g_2,{\cdots},g_{k-1})ut(1,h))\psi_k(u)dudh.
\end{aligned}
\]
We will show that $I_{\underline{r}^2}=0$ unless $r_2=0$ and prove the following lemma by induction. 

\begin{lem}
$I_{\underline{r}}=0$ unless $r_1=r_2=...=r_{k-1}=0$.
\end{lem}

\begin{proof}
Changing variables $u\mapsto u'=m(1,1,g_2,{\cdots},g_{k-1})^{-1}um(1,1,g_2,{\cdots},g_{k-1})$ we obtain
\[
\begin{aligned}
I_{\underline{r}^2}=&\int_{P_n(F)\backslash\mathrm{Sp}_{2n}(\mathbb{A})}\sum_{\substack{2\leq i\leq k-1\\g_i\in B_{r_i,n}^-(F)\backslash\mathrm{GL}_n(F)}}\int_{N_{\underline{r}^2}(F)\backslash N_{n^{k-1},kn}(\mathbb{A})}\phi(h)\theta_{\psi}^{\Phi}(\alpha^k_T(u)i_T(1,h))\\
\times&f_s(\gamma_{\underline{r}^2}um(1,1,g_2,{\cdots},g_{k-1})t(1,h))\psi_k(u')dudh.
\end{aligned}
\]
Note that $N_{\underline{r}^2}$ contains a subgroup $N_2$ consisting of elements of the form
\[
\left[\begin{smallmatrix}
1_{(k-3)n} & 0 & 0 & 0 & 0 & 0 & 0 & 0 & 0 & 0 & 0\\
0 & 1_{n-r_2} & 0 & 0 & 0 & 0 & 0 & 0 & 0 & 0 & 0\\ 
0 & 0 & 1_{r_2} & 0 & x & 0 & 0 & 0 & 0 & 0 &0 \\
0 & 0 & 0 & 1_{n-r_2} & 0 & 0 & 0 & 0 & 0 & 0 & 0\\
0 & 0 & 0 & 0 & 1_{r_2} & 0 & 0 & 0 & 0 & 0 &0 \\
0 &0 & 0& 0 & 0 & 1_{2n} & 0 & 0 & 0 & 0 & 0 \\
0 &0 & 0& 0 & 0 & 0 & 1_{r_2} & 0 & -x^{\ast} & 0 & 0 \\
0 & 0 &0 & 0 & 0 & 0 & 0&  1_{n-r_2} & 0 & 0 & 0\\
0 & 0 &0 & 0 & 0 & 0 & 0 & 0&  1_{r_2} & 0 & 0\\
0 & 0 &0 & 0 & 0 & 0 & 0 & 0 & 0&  1_{n-r_2} & 0\\
0 & 0 &0 & 0 & 0 & 0 & 0 & 0 & 0&  0 & 1_{(k-3)n}
\end{smallmatrix}\right].
\]
Thus
\[
\begin{aligned}
I_{\underline{r}^2}=&\int_{P_n(F)\backslash\mathrm{Sp}_{2n}(\mathbb{A})}\sum_{\substack{2\leq i\leq k-1\\g_i\in B_{r_i,n}^-(F)\backslash\mathrm{GL}_n(F)}}\int_{N_2(\mathbb{A})N_{\underline{r}^2}(F)\backslash N_{n^{k-1},kn}(\mathbb{A})}\phi(h)\psi_k(u')\\
\times&\theta_{\psi}^{\Phi}(\alpha^k_T(u)i_T(1,h))f_s(\gamma_{\underline{r}^2}um(1,1,g_2,{\cdots},g_{k-1})t(1,h))
 \int_{N_2(F)\backslash N_2(\mathbb{A})}\psi_k(u_2)du_2dudh.
\end{aligned}
\]
Since $\psi_k(u_2)$ is a nontrivial character on $N_2(F)\backslash N_2(\mathbb{A})$ the integral in the last line is zero unless $r_2=0$. The lemma then follows by induction on $\underline{r}^3,...,\underline{r}^{k-1}$ using the same argument.
\end{proof}

We then assume $\underline{r}=(0,...,0)$ and omit it from our notation (so $\gamma=\gamma_{0,{\cdots},0},N=N_{0,{\cdots},0}$). Our integral becomes
\[
\begin{aligned}
I=\int_{P_n(F)\backslash\mathrm{Sp}_{2n}(\mathbb{A})}\int_{N(F)\backslash N_{n^{k-1},kn}(\mathbb{A})}\phi(h)\theta_{\psi}^{\Phi}(\alpha^k_T(u)i_T(1,h))f_s(\gamma ut(1,h))\psi_k(u)dudh.
\end{aligned}
\]
We next unfold the theta series. The general linear group $\mathrm{GL}_n(F)$ acts on $\mathrm{Mat}_n(F)$ by right multiplication. For $\xi\in\mathrm{Mat}_n(F)$ denote $G_{\xi}=\{g\in\mathrm{GL}_n(F):\xi g=\xi\}$ for its stabilizer. Then
\[
I=\sum_{\xi\in\mathrm{Mat}_n(F)/\mathrm{GL}_n(F)}I_{\xi},
\]
where
\[
\begin{aligned}
I_{\xi}=&\int_{P_n(F)\backslash\mathrm{Sp}_{2n}(\mathbb{A})}\int_{N(F)\backslash N_{n^{k-1},kn}(\mathbb{A})}\phi(h)\\
\times&\sum_{a\in G_{\xi}\backslash\mathrm{GL}_n(F)}\omega_{\psi}(\alpha^k_T(u)i_T(1,h))\Phi(\xi a)f_s(\gamma ut(1,h))\psi_k(u)dudh.
\end{aligned}
\]

\begin{lem}
$I_{\xi}=0$ unless $\xi=1_n$.
\end{lem}

\begin{proof}
Assume $\xi\neq 1_n$ and let $\xi$ be of the form $\xi=\left[\begin{smallmatrix}
0 & x\\
0 & 1_r
\end{smallmatrix}\right]$ with $0\leq r<n$ so that $G_{\xi}=\left\{\left[\begin{smallmatrix}
\ast & \ast\\
0 & 1_r
\end{smallmatrix}\right]\in\mathrm{GL}_n(F)\right\}$. Recall that $\mathrm{GL}_n$ is embedded in $\mathrm{Sp}_{2n}$ via $a\mapsto m(a)=\mathrm{diag}[a,\hat{a}]$. Denote $\widetilde{G}_{\xi}$ for its image in $\mathrm{Sp}_{2n}$. Clearly the representatives of $G_{\xi}\backslash\mathrm{GL}_n(F)$ can be taken in $\mathrm{SL}_n(F)$ so that by \eqref{weil2}
\[
\omega_{\psi}(\alpha_T^k(u)i_T(1,h))\Phi(\xi a)=\omega_{\psi}(\alpha_T^k(u)i_T(1,m(a)h))\Phi(\xi)
\]
for $a\in G_{\xi}\backslash\mathrm{GL}_n(F)$. Changing variables $u\mapsto t(1,m(a))^{-1}ut(1,m(a))$ we obtain
\[
\begin{aligned}
I_{\xi}=&\int_{\widetilde{G}_{\xi}N_n(F)\backslash\mathrm{Sp}_{2n}(\mathbb{A})}\int_{N(F)\backslash N_{n^{k-1},kn}(\mathbb{A})}\phi(h)\omega_{\psi}(\alpha^k_T(u)i_T(1,h))\Phi(\xi )f_s(\gamma ut(1,h))\psi_k(u)dudh.
\end{aligned}
\] 
Let $N_n^r$ be a normal subgroup of $\widetilde{G}_{\xi}N_n$ consisting of elements of the form
\[
\left[\begin{smallmatrix}
1_{n-r} & x & y & z\\
0 & 1_r & 0 & y^{\ast}\\
0 & 0 & 1_r & -x^{\ast}\\
0 & 0 & 0 & 1_{n-r}
\end{smallmatrix}\right]
\]
and write
\[
\begin{aligned}
I_{\xi}=&\int_{\widetilde{G}_{\xi}N_n(F)N_n^r(\mathbb{A})\backslash\mathrm{Sp}_{2n}(\mathbb{A})}\int_{N(F)\backslash N_{n^{k-1},kn}(\mathbb{A})}\int_{N_n^r(F)\backslash N_n^r(\mathbb{A})}\phi(nh)\\
\times&\omega_{\psi}(\alpha^k_T(u)i_T(1,nh))\Phi(\xi )f_s(\gamma ut(1,nh))\psi_k(u)dndudh.
\end{aligned}
\]
Changing variables $u\mapsto t(1,n)ut(1,n)^{-1}$ and noting that $\gamma t(1,n)\gamma^{-1}\in N_n(F)$, we have
\[
\begin{aligned}
I_{\xi}=&\int_{\widetilde{G}_{\xi}N_n(F)N_n^r(\mathbb{A})\backslash\mathrm{Sp}_{2n}(\mathbb{A})}\int_{N(F)\backslash N_{n^{k-1},kn}(\mathbb{A})}\int_{N_n^r(F)\backslash N_n^r(\mathbb{A})}\phi(nh)\\
\times&\omega_{\psi}(i_T(1,n)\alpha^k_T(u)i_T(1,h))\Phi(\xi )f_s(\gamma ut(1,h))\psi_k(u)dndudh.
\end{aligned}
\]
Using formulas of the Weil representation, we have
\[
\begin{aligned}
&\omega_{\psi}(i_T(1,n)\alpha^k_T(u)i_T(1,h))\Phi(\xi )\\
=&\omega_{\psi}\left(\alpha_k^T\left(u^0\left(\left[\begin{smallmatrix}
0 & x\\
0 & 0
\end{smallmatrix}\right],0,0\right)\right)i_T(1,n)\alpha_T^k(u)i_T(1,h)\right)\Phi\left(\left[\begin{smallmatrix}
0 & 0\\
0 & 1_r
\end{smallmatrix}\right]\right)\\
=&\omega_{\psi}\left(i_T(1,n)\alpha^k_T\left(u^0\left(\left[\begin{smallmatrix}
0 & x\\
0 & 0
\end{smallmatrix}\right],\left[\begin{smallmatrix}
0 & x'\\
0 & 0
\end{smallmatrix}\right],0\right)\right)\alpha_T^k(u)i_T(1,h)\right)\Phi\left(\left[\begin{smallmatrix}
0 & 0\\
0 & 1_r
\end{smallmatrix}\right]\right)\\
=&\omega_{\psi}\left(\alpha^k_T\left(u^0\left(0,\left[\begin{smallmatrix}
0 & x'\\
0 & 0
\end{smallmatrix}\right],0\right)\right)\alpha_T^k(u)i_T(1,h)\right)\Phi(\xi).
\end{aligned}
\]
Here the first equality uses \eqref{weil1}, the second equality is just conjugating $i_T(1,n)$ to the front and the third equality uses \eqref{weil1} and \eqref{weil2}.
Changing variables $u\mapsto u^0\left(0,\left[\begin{smallmatrix}
0 & x'\\
0 & 0
\end{smallmatrix}\right],0\right)^{-1}u$ and since $u^0\left(0,\left[\begin{smallmatrix}
0 & x'\\
0 & 0
\end{smallmatrix}\right],0\right)\in N(\mathbb{A})$ the integral becomes
\[
\begin{aligned}
I_{\xi}=&\int_{\widetilde{G}_{\xi}N_n(F)N_n^r(\mathbb{A})\backslash\mathrm{Sp}_{2n}(\mathbb{A})}\int_{N(F)\backslash N_{n^{k-1},kn}(\mathbb{A})}\omega_{\psi}(\alpha^k_T(v)i_T(1,h))\Phi(\xi )f_s(\gamma ut(1,h))\psi_k(u)\\
\times&\int_{N_n^r(F)\backslash N_n^r(\mathbb{A})}\phi(nh)dndudh.
\end{aligned}
\]
The lemma follows since the integral in the last line vanishes by the cuspidality of $\phi$.
\end{proof}

We then assume $\xi=1_n$ and omit it from our notation. It remains to consider the integral
\[
\begin{aligned}
I=&\int_{N_n(F)\backslash\mathrm{Sp}_{2n}(\mathbb{A})}\int_{N(F)\backslash N_{n^{k-1},kn}(\mathbb{A})}\phi(h)\\
\times&\omega_{\psi}(\alpha_T^k(u)i_T(1,h))\Phi(1_n)f_s(\gamma ut(1,h))\psi_k(u)dudh\\
=&\int_{N_n(\mathbb{A})\backslash\mathrm{Sp}_{2n}(\mathbb{A})}\int_{N(F)\backslash N_{n^{k-1},kn}(\mathbb{A})}\int_{N_n(F)\backslash N_n(\mathbb{A})}\phi(nh)\\
\times&\omega_{\psi}(\alpha^k_T(u)i_T(1,nh))\Phi(1_n)f_s(\gamma ut(1,nh))dn\psi_k(u)dudh.
\end{aligned}
\]
Changing variables $u\mapsto t(1,n)ut(1,n)^{-1}$ and noting that by \eqref{weil3}
\[
\omega_{\psi}(i_T(1,n(z))\alpha_T^k(u)i_T(1,h))=\psi(\mathrm{tr}(Tz))\omega_{\psi}(\alpha_T^k(u)i_T(1,h)),
\]
we have
\[
\begin{aligned}
I=&\int_{N_n(\mathbb{A})\backslash\mathrm{Sp}_{2n}(\mathbb{A})}\int_{N(F)\backslash N_{n^{k-1},kn}(\mathbb{A})}\int_{N_n(F)\backslash N_n(\mathbb{A})}\phi_{\psi,T}(h)\\
\times&\omega_{\psi}(\alpha^k_T(u)i_T(1,h))\Phi(1_n)f_s(\gamma ut(1,h))\psi_k(u)dudh\\
=&\int_{N_n(\mathbb{A})\backslash\mathrm{Sp}_{2n}(\mathbb{A})}\int_{N(\mathbb{A})\backslash N_{n^{k-1},kn}(\mathbb{A})}\phi_{\psi,T}(h)\\
\times&\int_{N(F)\backslash N(\mathbb{A})}\omega_{\psi}(\alpha^k_T(u_0)\alpha^k_T(u)i_T(1,h))\Phi(1_n)f_s(\gamma u_0ut(1,h))\psi_k(u_0)du_0dudh.
\end{aligned}
\]

By straightforward computation, $N=N_{0,{\cdots},0}$ consists of elements of the form
\begin{equation}
\left[\begin{smallmatrix}
1_n &  v_{1,2} &\ast & \ast & 0 & \ast & 0 & 0& 0 & 0\\
0 & \ddots & \ddots & \ast & 0 & \ast & 0 & 0 & 0 & 0\\
0 & 0 & 1_n & v_{k-2,k-1} & 0 & \ast & 0 & 0 & 0  & 0\\
0 & 0 & 0 & 1_n & 0 & y & 0 & 0 &  0 & 0\\
0 & 0 & 0 & 0 & 1_n & 0 & y^{\ast} & \ast & \ast & \ast\\
0 & 0 & 0 & 0 & 0 & 1_n & 0 & 0 & 0 & 0\\
0 & 0 & 0 & 0 & 0 & 0 & 1_n & -v_{k-2,k-1}^{\ast} & \ast & \ast\\
0 & 0 & 0 & 0 & 0 & 0 & 0 & 1_n & \ddots & \ast\\
0 & 0 & 0 & 0 & 0 & 0 & 0 & 0 & \ddots & -v_{1,2}^{\ast}\\
0 & 0 & 0 & 0 & 0 & 0 & 0 & 0 & 0 & 1_n
\end{smallmatrix}\right].
\label{eq-unipotent-N}
\end{equation}
Denote elements of the above form by $u_0(y)$. The integral of $u_0$ over $N(F)\backslash N(\mathbb{A})$ can be written as
\[
\begin{aligned}
&\int_{N(F)\backslash N(\mathbb{A})}\omega_{\psi}(\alpha_T^k(u_0(y)u)i_T(1,h))\Phi(1_n)f_s(\gamma u_0(y)vt(1,h))\psi_k(u_0(y))du_0(y)\\
=&\omega_{\psi}(\alpha_T^k(v)i_T(1,h))\Phi(1_n)\int_{N(F)\backslash N(\mathbb{A})}f_s(\gamma u_0(y)ut(1,h))\psi_k(u_0(y))\psi(\mathrm{tr}(2Ty))du_0(y).
\end{aligned}
\]
Let
\[
\eta_0=\left[\begin{smallmatrix}
0 & 0 & 1_n & 0 & 0 & 0\\
0 & \iddots & 0 & 0 & 0 & 0\\
1_n & 0 & 0 & 0 & 0 & 0\\
0 & 0 & 0 & 0 & 0 & 1_n\\
0 & 0 & 0 & 0 & \iddots & 0\\
0 & 0 & 0 & 1_n & 0 & 0
\end{smallmatrix}\right],\quad \eta=\left[\begin{smallmatrix}
0 & 1_n & 0 & 0\\
0 & 0 & 0 & -1_{(k-1)n}\\
1_{(k-1)n} & 0 & 0 & 0\\
0 & 0 & 1_n & 0
\end{smallmatrix}\right].
\]
For $u_0\in N$ as in \eqref{eq-unipotent-N}, note that
\[
\eta_0\gamma u_0=\left[\begin{smallmatrix}
1_n & -y^{\ast} & \ast & \ast & \ast & 0 & 0 & 0 & 0 & 0\\
0 & 1_n & -v_{k-2,k-1}^{\ast} & \ast & \ast & 0 & 0 & 0 & 0 & 0\\
0 & 0 & \ddots & \ddots & \ast &  0 & 0 & 0 & 0 & 0\\
0 & 0 & 0 & 1_n & -v_{1,2}^{\ast} & 0 & 0 & 0 & 0 & 0\\
0 & 0 & 0 & 0 & 1_n & 0 & 0 & 0 & 0 & 0\\
0 & 0 & 0 & 0 & 0 & 1_n & v_{1,2} & \ast & \ast & \ast\\
0 & 0 & 0 & 0 & 0 & 0 & \ddots & \ddots  & \ast & \ast\\
0 & 0 & 0 & 0 & 0 & 0 & 0 & 1_n & v_{k-2,k-1} & \ast\\
0 & 0 & 0 & 0 & 0 & 0 & 0 & 0 & 1_n & y\\
0 & 0 & 0 & 0 & 0 & 0 & 0 & 0 & 0 & 1_n
\end{smallmatrix}\right]\eta=:u_0'\eta.
\]
Therefore, the integral over $N(F)\backslash N(\mathbb{A})$ becomes
\[
\begin{aligned}
&\int_{N(F)\backslash N(\mathbb{A})}f_s(u_0'(y)\eta ut(1,h))\psi_k(u_0(y))\psi(\mathrm{tr}(2Ty))du_0(y)\\
=&\int_{U_{n^k}(F)\backslash U_{n^k}(\mathbb{A})}f_s\left(\left[\begin{smallmatrix}
v & 0\\
0 & \hat{v}
\end{smallmatrix}\right]\eta ut(1,h)\right)\psi_{2T}^{-1}(v)dv.
\end{aligned}
\]
Here $v$ is of the form
\[
\left[\begin{smallmatrix}
1_n & v_{1,2} & \ast & \ast & \ast\\
0 & \ddots & \ddots & \ast & \ast\\
0 & 0 & 1_n & v_{k-2,k-1} & \ast\\
0 & 0 & 0 & 1_n & y\\
0 & 0 & 0 & 0 & 1_n
\end{smallmatrix}\right]
\]
and
\[
\psi_{2T}^{-1}(v)=\psi\left(\mathrm{tr}(2Ty)+\sum_{i=1}^{k-2}\mathrm{tr}(2Tv_{i,i+1})\right).
\]
This completes the proof of Proposition~\ref{prop 4.1}.

	\section{The unramified computation}
	\label{section-unramified-computation}
	
	In this section, we study the local zeta integral corresponding to $\mathcal{Z}(\phi,\theta_{\psi,n^2}^{\Phi},E(\cdot, f_s))$ at $v\notin S$. To ease the notation, we shall omit the symbol $v$ for simplicity. Therefore, $F$ is a non-archimedean local field with ring of integers $\mathcal{O}$. We assume $|2|=|3|=1$ in $F$. Let $T_0=\mathrm{diag}[t_1,...,t_n]\in\GL_n(F)\cap\mathrm{M at}_n(\mathcal{O}^{\times})$ be a diagonal matrix and set $T=J_nT_0$. We define the quadratic character $\chi_T:F^{\times}\to\C^{\times}$ by $\chi_T(x)=(x,\det(T))$ where $(\cdot,\cdot)$ is the local Hilbert symbol. Fix a nontrivial additive unramified character $\psi$ of $F$. 
	
	Let $(\pi,V_{\pi})$ be an irreducible admissible unramified representation of $\mathrm{Sp}_{2n}(F)$ with a fixed non-zero unramified vector $v_0\in V_{\pi}$. Let $(\tau,V_{\tau})$ be an irreducible admissible unitary generic unramified representation of $\GL_k(F)$. Take $\Phi^0=\mathbf{1}_{\mathrm{Mat}_n(\mathcal{O})}$ to be the characteristic function of $\mathrm{Mat}_n(\mathcal{O})$ and
	\[
	f^0_{\mathcal{W}(\tau\otimes\chi_T,n,\psi_{2T}),s}\in\mathrm{Ind}_{P_{kn}(F)}^{\mathrm{Sp}_{2kn}(F)}(\mathcal{W}(\tau\otimes\chi_T,n,\psi_{2T})|\det\cdot|^s),
	\]
	the unramified section normalized such that 
	\[
	f^0_{\mathcal{W}(\tau\otimes\chi_T,n,\psi_{2T}),s}(1_{2kn})=d_{\tau}^{\mathrm{Sp}_{4kn}}(s).
	\]
	
	The aim of this section is to prove Theorem \ref{thm-unramified-computation} restated as follows.
	
		\begin{thm}
			\label{mainthm}
		Let $l_T:V_{\pi}\to\C$ be a linear functional on $V_{\pi}$ satisfying \eqref{4.9}. Then for $\mathrm{Re}(s)\gg0$, we have
		\begin{equation}
			\label{maineq}
			\begin{aligned}
				&\mathcal{Z}^{\ast}(l_T,s)\\
				:=&\int_{N_n(F)\backslash\mathrm{Sp}_{2n}(F)}\int_{N_{n^{k-1},kn}^0(F)}l_T(\pi(h)v_0)\omega_{\psi}(\alpha_T^k(u)i_T(1,h))\Phi^0(1_n)f^{0}_{\mathcal{W}(\tau\otimes\chi_T,n,\psi_{2T}),s}(\eta ut(1,h))dudh\\
				=&L(s+\frac{1}{2},\pi\times\tau)\cdot l_T(v_0).
			\end{aligned}
		\end{equation}
	\end{thm}
	
	The strategy for proving this theorem is to compare this integral with the unramified integral from the generalized doubling method \eqref{3.2.1}. We will reformulate the integral in \eqref{3.2.1} and the integral $\mathcal{Z}^{\ast}(l_T,s)$ in Section \ref{sec6.2} and \ref{sec6.3} respectively (see Proposition \ref{prop6.4} and \ref{prop6.9}). By showing that they coincide in Section \ref{sec6.4}, we then obtain \eqref{maineq} by Theorem \ref{thm29CFGK}.

	\subsection{Relation between unramified sections}
	
	As preparation, we show the relation between unramified sections $f^0_{\mathcal{W}(\tau\otimes\chi_T,n,\psi),s}$ when the additive character $\psi$ varies. Recall that for any character $\psi_{n^k}:U_{n^k}(F)\to\C^{\times}$, the $(k,n)$-model $\mathcal{W}(\tau\otimes\chi_T,n,\psi_{n^k})$ consists of functions $W_{\xi}:\mathrm{GL}_{kn}(F)\to\C$ of the form
	\[
	W_{\xi}(g)=\Lambda(\Delta(\tau\otimes\chi_T,n)(g)\xi)
	\]
	where $\xi$ is in the space of $\Delta(\tau\otimes\chi_T,n)$ and $\Lambda$ can be realized as
	\[
	\xi\mapsto\int_{U_{n^k}(F)}\xi(w_{k,n}u)\psi^{-1}_{n^k}(u)du, \quad w_{k,n}=\left[\begin{smallmatrix}
		0 & 0 & 0 & 1_n\\
		0 & 0 & 1_n & 0\\
		0 & \iddots & 0 & 0\\
		1_n & 0 & 0 & 0
	\end{smallmatrix}\right]\in \GL_{kn}.
	\]
	By abusing the notation we denote the extension of our fixed character $\psi:F\to\C^{\times}$ to $U_{n^k}(F)\to \C^\times$ (see \eqref{2.3.2}) also as $\psi$. Recall that we have also defined a character $\psi_{2T}$ on $U_{n^k}(F)$. Given $g_1,{\cdots},g_k\in\mathrm{GL}_n(F)$, we define another character $\psi_{g_1,{\cdots},g_k}$ on $U_{n^k}(F)$ by
	\[
	\psi_{g_1,{\cdots},g_{k}}(u)=\psi\left(\sum_{i=1}^{k-1}\mathrm{tr}(2g_iTg_{i+1}^{-1}u_{i,i+1})\right),
	\]
	with $u$ of the form in \eqref{2.3.1}. Let $W_{\xi}^{g_1,{\cdots},g_k}$ be the function in $\mathcal{W}(\tau\otimes\chi_T,n,\psi_{g_1,{\cdots},g_k})$ corresponding to $\xi\in\Delta(\tau\otimes\chi_T,n)$, the linear functional $\Lambda$, and the character $\psi_{g_1,{\cdots},g_{k}}$ as above. Then $\psi_{2T}=\psi_{1,{\cdots},1}$ and we simply denote $W_{\xi}=W_{\xi}^{1,{\cdots},1}$. Note that the Levi subgroup $M_{n^k}$ of $P_{n^k}$ acts transitively on the set of generic characters of $U_{n^k}$. By left-translation of an element in the Levi subgroup we can change the character appearing in the model as in the following lemma. Recall that for $g\in \GL_{nk}(F)$, we denote $m(g)=\diag[g, \hat{g}]\in M_{nk}(F)$.
		
	\begin{lem}
		\label{lemma-relation-of-sections}
		If $W_{\xi}$ is unramified then so is $W_{\xi}^{g_1,{\cdots},g_k}$. More precisely,
		\[
		W_{\xi}\left(\mathrm{diag}[g_1,{\cdots},g_k]g\right)=\left|\frac{\det g_1}{\det g_k}\right|^n\prod_{i=1}^k\chi_i\chi_T(\det g_{k+1-i})|\det g_{k+1-i}|^{\frac{(k-2i+1)n}{2}}W_{\xi}^{g_1,{\cdots},g_k}(g).
		\]
	\end{lem}
	
	\begin{proof}
		By the definition of $W_{\xi}$, we have
		\[
		\begin{aligned}
			W_{\xi}(\mathrm{diag}[g_1,{\cdots},g_k]g)
			=\int_{U_{n^k}(F)}\Delta(\tau\otimes\chi_T,n)(g)\xi\left(w_{k,n}u\mathrm{diag}[g_1,{\cdots},g_k]\right)\psi_{2T}^{-1}(u)du.
		\end{aligned}
		\]
		Changing variables $u\mapsto\mathrm{diag}[g_1,..,g_k]u\mathrm{diag}[g_1,{\cdots},g_k]^{-1}$, we note that $u_{i,i+1}$ is changed to $g_iu_{i,i+1}g_{i+1}^{-1}$ and the above integral equals
		\[
		\begin{aligned}
			&\left|\frac{\det g_1}{\det g_k}\right|^n\int_{U_{n^k}(F)}\Delta(\tau\otimes\chi_T,n)(g)\xi\left(\mathrm{diag}[g_k,{\cdots},g_1]w_{k,n}u\right)\psi_{g_1,{\cdots},g_k}^{-1}(u)du\\
			=&\left|\frac{\det g_1}{\det g_k}\right|^n\prod_{i=1}^k\chi_i\chi_T(\det g_{k+1-i})|\det g_{k+1-i}|^{\frac{(k-2i+1)n}{2}}\int_{U_{n^k}(F)}\Delta(\tau\otimes\chi_T,n)(g)\xi\left(w_{k,n}u\right)\psi_{g_1,{\cdots},g_k}^{-1}(u)du
		\end{aligned}
		\]
		as desired.
	\end{proof}
	
	Using Lemma~\ref{lemma-relation-of-sections}, we have the following relation between unramified sections.
	
	\begin{cor}
		For $g_1,{\cdots},g_k\in\mathrm{GL}_n(F)$, there exists an unramified section 
		\[
		f^0_{\mathcal{W}(\tau\otimes\chi_T,n,\psi_{g_1,{\cdots},g_k}),s}\in\mathrm{Ind}_{P_{kn}(F)}^{\mathrm{Sp}_{2kn}(F)}(\mathcal{W}(\tau\otimes\chi_T,n,\psi_{g_1,{\cdots},g_k})|\det\cdot|^s)
		\]
		determined by $f^0_{\mathcal{W}(\tau\otimes\chi_T,n,\psi_{2T}),s}$ such that
		\[
		\begin{aligned}
			&f^0_{\mathcal{W}(\tau\otimes\chi_T,n,\psi_{2T}),s}(m(\mathrm{diag}[g_1,{\cdots},g_k])g)\\
			=&\left|\frac{\det g_1}{\det g_k}\right|^n\prod_{i=1}^k\chi_i\chi_T(\det g_{k+1-i})|\det g_{k+1-i}|^{s+(k-i)n+\frac{n+1}{2}}f^0_{\mathcal{W}(\tau\otimes\chi_T,n,\psi_{g_1,{\cdots},g_k}),s}(g).
		\end{aligned}
		\]
		In particular, if we take $g_i=(2T)^{i-1}$ then $\psi_{g_1,{\cdots},g_k}=\psi$ and we have
		\[
		f^0_{\mathcal{W}(\tau\otimes\chi_T,n,\psi_{2T}),s}(m(\mathrm{diag}[g_1,{\cdots},g_k])g)=f^0_{\mathcal{W}(\tau\otimes\chi_T,n,\psi),s}(g).
		\]
		\label{cor-relation-sections}
	\end{cor}

	\subsection{Reformulating the unramified integral from the generalized doubling method}
	\label{sec6.2}
	Recall the following unramified integral \eqref{3.2.1} from the generalized doubling method
	\begin{equation}
		\label{doublingintegral}
		\begin{aligned}
		\int_{\Sp_{2n}(F)}\int_{N^0_{(2n)^{k-1},2kn}(F)}\omega_{\pi}^0(h)f^{0}_{\mathcal{W}(\tau,2n,\psi),s}(\delta u(1\times {^{\iota}h}))\psi_{N^0_{(2n)^{k-1},2kn}}(u)dudh
		=L(s+\frac{1}{2},\pi\times\tau),
	\end{aligned}
	\end{equation}
	where $N^0_{(2n)^{k-1},2kn}:=N_{(2n)^{k-1},2kn}\cap N_{2kn}$ contains elements of the form
	\begin{equation}
		\label{n0}
		\left[\begin{smallmatrix}
			1_{2n(k-2)} & 0 & 0 & 0 & 0 & a_1 & a_2 & b_1 & b_2 & c\\
			0 & 1_n & 0 & 0 & 0 & x_1 & x_2 & z_1 & z_2 & b_2^{\ast}\\
			0 & 0 & 1_n & 0 & 0 & x_3 & x_4 & z_3 & z_1^{\ast} & b_1^{\ast}\\
			0 & 0 & 0 & 1_n &0 & 0 & 0 & x_4^{\ast} & x_2^{\ast} & a_2^{\ast}\\
			0 & 0 & 0 & 0 & 1_n & 0 & 0 & x_3^{\ast} & x_1^{\ast} & a_1^{\ast}\\
			0 & 0 & 0 & 0 & 0 & 1_n & 0 & 0 & 0 & 0\\
			0 & 0 & 0 & 0 & 0 & 0 & 1_n & 0 & 0 & 0\\
			0 & 0 & 0 & 0 & 0 & 0 & 0 & 1_n & 0 & 0\\
			0 & 0 & 0 & 0 & 0 & 0 & 0 & 0 & 1_n & 0\\
			0 & 0 & 0 & 0 & 0 & 0 & 0 & 0 & 0 & 1_{2n(k-2)}
		\end{smallmatrix}\right]
	\end{equation}
	and the character $\psi_{N^0_{(2n)^{k-1},2kn}}$ is defined by
	\[
	\psi_{N^0_{(2n)^{k-1},2kn}}(u)=\psi(\mathrm{tr}(x_4))
	\]
	for $u\in N^0_{(2n)^{k-1},2kn}(F)$ of the above form \eqref{n0}.

	\begin{lem}\label{lemma6.4}
		Let $v_0\in V_{\pi}$ be an unramified vector and $l_T$ any linear functional on $V_{\pi}$ satisfying \eqref{4.9}. Then for $\mathrm{Re}(s)\gg0$,
		\begin{equation}
		\begin{split}
			&\int_{\mathrm{Sp}_{2n}(F)}\int_{N^0_{(2n)^{k-1},2kn}(F)}l_T(\pi(h)v_0)f^{0}_{\mathcal{W}(\tau,2n,\psi),s}(\delta u_0(1\times {^{\iota}h}))\psi_{N^0_{(2n)^{k-1},2kn}}(u)dudh\\
			=&L(s+\frac{1}{2},\pi\times\tau)\cdot l_T(v_0).
		\end{split}
		\label{eq-unramified-eq1}
	\end{equation}
	\end{lem}
	
	\begin{proof} 
		Since $v_0\in V_{\pi}$ is unramified, the function
		\begin{equation}
			\label{44}
		h\mapsto\int_{\mathrm{Sp}_{2n}(\mathcal{O})}l_T(\pi(kh)v_0)dk,\qquad h\in\Sp_{2n}(F)
		\end{equation}
		is bi-invariant under $\Sp_{2n}(\mathcal{O})$, and as such is a constant multiple of $\omega_{\pi}^0$. Note that the value of the function \eqref{44} at identity is $l_T(v_0)$. It follows from the normalization of $\omega_{\pi}^0$ that
		\[
		\int_{\Sp_{2n}(\mathcal{O})}l_T(\pi(kh)v_0)dk=l_T(v_0)\omega_{\pi}^0(h).
		\]
		Note that the function
		\[
		h\mapsto \int_{N^0_{(2n)^{k-1},2kn}(F)}l_T(\pi(h)v_0)f^{0}_{\mathcal{W}(\tau,2n,\psi),s}(\delta u(1\times {^{\iota}h}))\psi_{N^0_{(2n)^{k-1},2kn}}(u)du
		\]
		is bi-invariant under $\Sp_{2n}(\mathcal{O})$, thus we obtain
		\[
		\begin{aligned}
			&\int_{\mathrm{Sp}_{2n}(F)}\int_{N^0_{(2n)^{k-1},2kn}(F)}l_T(\pi(h)v_0)f^{0}_{\mathcal{W}(\tau,2n,\psi),s}(\delta u(1\times {^{\iota}h}))\psi_{N^0_{(2n)^{k-1},2kn}}(u)dudh\\
			=&\int_{\Sp_{2n}(\mathcal{O})\backslash\Sp_{2n}(F)}\int_{\Sp_{2n}(\mathcal{O})}l_T(\pi(kh)v_0)dk\int_{N^0_{(2n)^{k-1},2kn}(F)}f^{0}_{\mathcal{W}(\tau,2n,\psi),s}(\delta u(1\times {^{\iota}h}))\psi_{N^0_{(2n)^{k-1},2kn}}(u)dudh\\
			=&l_T(v_0)	\int_{\Sp_{2n}(F)}\int_{N^0_{(2n)^{k-1},2kn}(F)}\omega_{\pi}^0(g)f^{0}_{\mathcal{W}(\tau,2n,\psi),s}(\delta u(1\times {^{\iota}h}))\psi_{N^0_{(2n)^{k-1},2kn}}(u)dudh\\
			=&L(s+\frac{1}{2},\pi\times\tau)\cdot l_T(v_0),
		\end{aligned}
		\]
		by \eqref{doublingintegral}.
	\end{proof}
	
	In this subsection, we reformulate the integral \eqref{eq-unramified-eq1} and prove the following.
	
	\begin{prop} \label{prop6.4}
		For $\mathrm{Re}(s)\gg0$, we have
		\begin{equation}
			\label{propL}
			\begin{aligned}
			&\int_{\mathrm{Sp}_{2n}(F)}\int_{N^0_{(2n)^{k-1},2kn}(F)}l_T(\pi(h)v_0)f^{0}_{\mathcal{W}(\tau,2n,\psi),s}(\delta u(1\times {^{\iota}h}))\psi_{N^0_{(2n)^{k-1},2kn}}(u)dudh\\
			=&\int_{\mathrm{GL}_n(F)\cap\mathrm{Mat}_n(\mathcal{O})}l_T(\pi(m(a))v_0)|\det a|^{-(2nk-n+1)}f^{0}_{\mathcal{W}(\tau,2n,\psi),s}\left(\left[\begin{smallmatrix}
				a & 0 & 0 & 0 & 0\\
				0 & 1_n & 0 & 0 & 0\\
				0 & 0 & 1_{4n(k-1)} & 0 & 0\\
				0 & 0 & 0 & 1_n & 0\\
				0 & 0 & 0 & 0 & \hat{a}
			\end{smallmatrix}\right]\right)da.
			\end{aligned}
		\end{equation}
	\end{prop}

	We start with the Iwasawa decomposition $\Sp_{2n}(F)=P_n(F)\Sp_{2n}(\mathcal{O})$ with respect to the Siegel parabolic subgroup. We write $h=n(z)m(a)k$ for $m(a)\in M_n(F)$, $n(z)\in N_n(F)$ and $k\in\Sp_{2n}(\mathcal{O})$. The left hand side of \eqref{propL} becomes
	\[
		\begin{aligned}
			&\int_{\mathrm{GL}_n(F)}\int_{\mathrm{Mat}_n^0(F)}l_T(\pi(m(a))v_0)\psi^{-1}(\mathrm{tr}(Tz))|\det a|^{-n-1}\\
			\times&\int_{N^0_{(2n)^{k-1},2kn}(F)}f^{0}_{\mathcal{W}(\tau,2n,\psi),s}(\delta u_0(1\times {^{\iota}n}(z)\cdot{^{\iota}m}(a)\cdot{^{\iota }k}))\psi_{N^0_{(2n)^{k-1},2kn}}(u)dudzda\\
			=&\int_{\mathrm{GL}_n(F)}\int_{\mathrm{Mat}_n^0(F)}l_T(\pi(m(a))v_0)\psi^{-1}(\mathrm{tr}(Tz))|\det a|^{-n-1}\\
			\times&	\int_{N^0_{(2n)^{k-1},2kn}(F)}f^{0}_{\mathcal{W}(\tau,2n,\psi),s}(\delta u_0\delta^{-1}\cdot\delta(n(-z)\times {^{\iota}m}(a))\delta^{-1})\psi_{N^0_{(2n)^{k-1},2kn}}(u)du.
		\end{aligned}
	\]
	 We denote $\overline{N}^0_{(2n)^{k-1},2kn}=\delta N^0_{(2n)^{k-1},2kn}\delta^{-1}$ which consists elements of the form
	\begin{equation}
		\label{6.2.2}
		\left[\begin{smallmatrix}
			1_{n} & 0 & 0 & 0 & 0 & 0 & 0 & 0 & 0 & 0\\
			0 & 1_{n} & 0 & 0 & 0 & 0 & 0 & 0 & 0 & 0\\
			0 & 0 & 1_{n} & 0 & 0 & 0 & 0 & 0 & 0 & 0\\
			0 & 0 & 0 & 1_{n} & 0 & 0 & 0 & 0 & 0 & 0\\
			0 & 0 & 0 & 0 & 1_{2n(k-2)} & 0 & 0 & 0 & 0 & 0\\
			d_1 & d_2 & c_1 & c_2 & b & 1_{2n(k-2)} & 0 & 0 & 0 & 0\\
			x_1 & x_2 & z_1 & z_2 & c_2^{\ast} & 0 & 1_{n} & 0 & 0 & 0\\
			x_3 & x_4 & z_3 & z_1^{\ast} & c_1^{\ast} & 0 & 0 & 1_{n} & 0 & 0\\
			0 & 0 & x_4^{\ast} & x_2^{\ast} & d_2^{\ast} & 0 & 0 & 0 & 1_{n} & 0\\
			0 & 0 & x_3^{\ast} & x_1^{\ast} & d_1^{\ast} & 0 & 0 & 0 & 0 & 1_{n}
		\end{smallmatrix}\right].
	\end{equation}
	For $u$ as in \eqref{6.2.2} we define
	\[
	\psi_{\overline{N}^0_{(2n)^{k-1},2kn}}(u)=\psi(\mathrm{tr}(x_4)).
	\]
	Then above integral can be rewritten as
	\begin{equation}
		\begin{aligned}
			&\int_{\mathrm{GL}_n(F)}\int_{\mathrm{Mat}_n^0(F)}l_T(\pi(m(a))v_0)\psi^{-1}(\mathrm{tr}(Tz))|\det a|^{-n-1}\int_{\overline{N}^0_{(2n)^{k-1},2kn}(F)}\psi_{\overline{N}^0_{(2n)^{k-1},2kn}}(u)\\
			\times& f^{0}_{\mathcal{W}(\tau,2n,\psi),s}\left(u\left[\begin{smallmatrix}
				a & 0 & 0 & 0 & 0\\
				0 & 1_n & 0 & 0 & 0\\
				0 & 0 & 1_{4n(k-1)} & 0 & 0\\
				0 & 0 & 0 & 1_n & 0\\
				0 & 0 & 0 & 0 & \hat{a}
			\end{smallmatrix}\right]\left[\begin{smallmatrix}
				1_n & 0 & 0 & 0 & 0 \\
				0 & 1_n & 0 & 0 & 0 \\
				0 & 0 & \widetilde{z} & 0 & 0\\
				1-a & z & 0 & 1_n & 0\\
				0 & 1-a^{\ast} & 0 & 0 & 1_n
			\end{smallmatrix}\right]\right)dudzda.
		\end{aligned}
		\label{eq-unramified-eq3}
	\end{equation}
	Here we denote $\widetilde{z}=\mathrm{diag}[^\iota n(z),{\cdots},^\iota n(z),{^{\iota}n}(-z),{\cdots},{^{\iota}n}(-z)]$ with $^\iota n(z)$ and ${^{\iota}n}(-z)$ appearing $k-1$ times respectively.

	To prove \eqref{propL}, we will show that the inner integral along $\overline{N}^0_{(2n)^{k-1},2kn}(F)$ in \eqref{eq-unramified-eq3} vanishes unless $a\in\GL_n(F)\cap\mathrm{Mat}_n(\mathcal{O})$ (Lemma \ref{lem6.5}), $z\in\mathrm{Mat}_n^0(F)$ (Lemma \ref{lem6.6}) and then compute the inner integral in Lemma \ref{lem6.8} and Lemma \ref{lem6.9}. Our main tools are the following invariance properties of $f^{0}_{\mathcal{W}(\tau,2n,\psi),s}$.

	\begin{itemize}
		\item Since $f^{0}_{\mathcal{W}(\tau,2n,\psi),s}$ is unramified, we have
		\begin{equation}
			\label{invarianta}
		f^{0}_{\mathcal{W}(\tau,2n,\psi),s}(gk)=f^{0}_{\mathcal{W}(\tau,2n,\psi),s}(g),\qquad g\in\Sp_{4kn}(F),
		\end{equation}
		 for all $k\in\Sp_{4kn}(\mathcal{O})$.
		\item By definition of the parabolic induction, we have
		\begin{equation}
			\label{invariantb}
		f^{0}_{\mathcal{W}(\tau,2n,\psi),s}(ng)=f^{0}_{\mathcal{W}(\tau,2n,\psi),s}(g),\qquad g\in\Sp_{4kn}(F),
		\end{equation}
		for all $n\in N_{2kn}(F)$ in the unipotent radical of the Siegel parabolic subgroup of $\Sp_{4kn}(F)$.
		\item By the invariance property (Proposition \ref{prop-invariance-(k,c)-representation}, see also \cite[page. 1030]{CaiFriedbergGinzburgKaplan2019}), we have
		\begin{equation}
			\label{invariantc}
		f^{0}_{\mathcal{W}(\tau,2n,\psi),s}(\mathrm{diag}[g_0,...,g_0,\hat{g}_0,...,\hat{g}_0]g)=f^{0}_{\mathcal{W}(\tau,2n,\psi),s}(g),\qquad g\in\Sp_{4kn}(F),
		\end{equation}
		for all $g_0\in\Sp_{2n}(F)$.
	\end{itemize}

	\begin{lem}
		\label{lem6.5}
		The inner integral along $\overline{N}^0_{(2n)^{k-1},2kn}(F)$ in \eqref{eq-unramified-eq3} vanishes unless $a\in\mathrm{GL}_n(F)\cap\mathrm{Mat}_n(\mathcal{O})$.
	\end{lem}
	
	\begin{proof}
		Denote the integral along $\overline{N}^0_{(2n)^{k-1},2kn}(F)$ in \eqref{eq-unramified-eq3} by
		\[
		\begin{aligned}
		\mathcal{L}_1:=\int_{\overline{N}^0_{(2n)^{k-1},2kn}(F)}\psi_{\overline{N}^0_{(2n)^{k-1},2kn}}(u)f^{0}_{\mathcal{W}(\tau,2n,\psi),s}\left(u\left[\begin{smallmatrix}
			a & 0 & 0 & 0 & 0\\
			0 & 1_n & 0 & 0 & 0\\
			0 & 0 & 1_{4n(k-1)} & 0 & 0\\
			0 & 0 & 0 & 1_n & 0\\
			0 & 0 & 0 & 0 & \hat{a}
		\end{smallmatrix}\right]\left[\begin{smallmatrix}
			1_n & 0 & 0 & 0 & 0 \\
			0 & 1_n & 0 & 0 & 0 \\
			0 & 0 & \widetilde{z} & 0 & 0\\
			1-a & z & 0 & 1_n & 0\\
			0 & 1-a^{\ast} & 0 & 0 & 1_n
		\end{smallmatrix}\right]\right)du.
		\end{aligned}
		\]
		We translate $f^{0}_{\mathcal{W}(\tau,2n,\psi),s}$ on the right by
		\[
		R_1:=\left[
		\begin{smallmatrix}
			1_n & 0 & r & 0 & 0 & 0 & 0 & 0 & 0\\
			0 & 1_n & 0 & 0 & 0 & 0 & 0 & 0 & 0\\
			0 & 0 & 1_n & 0 & 0 & 0 & 0 & 0 & 0\\
			0 & 0 & 0 & 1_n & 0 & 0 & 0 & 0 & 0\\
			0 & 0 & 0 & 0 & 1_{4n(k-2)} & 0 & 0 & 0 & 0\\
			0 & 0 & 0 & 0 & 0 & 1_n & 0 & 0 & 0\\
			0 & 0 & 0 & 0 & 0 & 0 & 1_n & 0 & -r^{\ast}\\
			0 & 0 & 0 & 0 & 0 & 0 & 0 & 1_n & 0\\
			0 & 0 & 0 & 0 & 0 & 0 & 0 & 0 & 1_n
		\end{smallmatrix}
		\right]
		\]
		for $r\in\mathrm{Mat}_n(\mathcal{O})$. This stabilizes $f^{0}_{\mathcal{W}(\tau,2n,\psi),s}$ by \eqref{invarianta}. Conjugate this matrix to the left we have that 
		\[
			\begin{aligned}
			\mathcal{L}_1=\int_{\overline{N}^0_{(2n)^{k-1},2kn}(F)}\psi_{\overline{N}^0_{(2n)^{k-1},2kn}}(u) f^{0}_{\mathcal{W}(\tau,2n,\psi),s}\left(R_1'u\left[\begin{smallmatrix}
				a & 0 & 0 & 0 & 0\\
				0 & 1_n & 0 & 0 & 0\\
				0 & 0 & 1_{4n(k-1)} & 0 & 0\\
				0 & 0 & 0 & 1_n & 0\\
				0 & 0 & 0 & 0 & \hat{a}
			\end{smallmatrix}\right]\left[\begin{smallmatrix}
			1_n & 0 & 0 & 0 & 0 \\
			0 & 1_n & 0 & 0 & 0 \\
			0 & 0 & \widetilde{z} & 0 & 0\\
			1-a & z & 0 & 1_n & 0\\
			0 & 1-a^{\ast} & 0 & 0 & 1_n
			\end{smallmatrix}\right]\right)du.
		\end{aligned}
		\]
		with
		\[
		R_1'=\left[\begin{smallmatrix}
		1_n & 0 & ar & 0 & 0 & 0 & 0 & 0 & 0\\
		0 & 1_n & 0 & 0 & 0 & 0 & 0 & 0 & 0 \\
		0 & 0 & 1_n & 0 & 0 & 0 & 0 & 0 & 0 \\
		0 & 0 & 0 & 1_n & 0 & 0 & 0 & 0 & 0\\
		0 & 0 & 0 & 0 & 1_{4n(k-2)} & 0 & 0 & 0 &0\\
		0 & 0 & 0 & 0 & 0 & 1_n & 0 & 0 & 0\\
		0 & r^{\ast}(1-a^{\ast}) & 0 & 0 & 0 & 0 & 1_n & 0 & -r^{\ast}a^{\ast}\\
		0 & 0 & (1-a)r &0  & 0 & 0 & 0 & 1_n & 0\\
		0 & 0 & 0 & 0 & 0 & 0 & 0 & 0 & 1_n
		\end{smallmatrix}\right].
		\]
		Changing variables of $u_0$ by $x_4\mapsto x_4-r^{\ast}(1-a^{\ast})$ we see that
		\[
	\mathcal{L}_1=\psi(\mathrm{tr}(ar))\psi(\mathrm{tr}(ar-r))\mathcal{L}_1.
		\]
		Here $\psi(\mathrm{tr}(ar))$ comes from $f^{0}_{\mathcal{W}(\tau,2n,\psi),s}$ and $\psi(\mathrm{tr}(ar-r))$ comes from $\psi_{\overline{N}^0_{(2n)^{k-1},2kn}}$. Hence $\mathcal{L}_1=0$ unless $\psi(\mathrm{tr}(2ar))=1$ for all $r\in\mathrm{Mat}_n(\mathcal{O})$, which implies $a\in\mathrm{Mat}_n(\mathcal{O})$.
	\end{proof}
	
	The integral \eqref{eq-unramified-eq3} then equals
	\begin{equation}
		\label{6.2.3}
		\begin{aligned}
			&\int_{\mathrm{GL}_n(F)\cap\mathrm{Mat}_n(\mathcal{O})}\int_{\mathrm{Mat}_n^0(F)}l_T(\pi(m(a))v_0)\psi^{-1}(\mathrm{tr}(Tz))|\det a|^{-n-1}\\
			\times&\int_{\overline{N}^0_{(2n)^{k-1},2kn}(F)}\psi_{\overline{N}^0_{(2n)^{k-1},2kn}}(u)f^{0}_{\mathcal{W}(\tau,2n,\psi),s}\left(u\left[\begin{smallmatrix}
				a & 0 & 0 & 0 & 0\\
				0 & 1_n & 0 & 0 & 0\\
				0 & 0 & 1_{4n(k-1)} & 0 & 0\\
				0 & 0 & 0 & 1_n & 0\\
				0 & 0 & 0 & 0 & \hat{a}
			\end{smallmatrix}\right]\left[\begin{smallmatrix}
				1_n & 0 & 0 & 0 & 0 \\
				0 & 1_n & 0 & 0 & 0 \\
				0 & 0 & \widetilde{z} & 0 & 0\\
				0 & z & 0 & 1_n & 0\\
				0 & 0 & 0 & 0 & 1_n
			\end{smallmatrix}\right]\right)dudzda.
		\end{aligned}
	\end{equation}

	\begin{lem}
		\label{lem6.6}
		The inner integral along $\overline{N}^0_{(2n)^{k-1},2kn}(F)$ in \eqref{6.2.3} vanishes unless $z\in\mathrm{Mat}_n^0(\mathcal{O})$.
	\end{lem}
	
	\begin{proof}
	The proof is the same as Lemma \ref{lem6.5}. We denote the integral along $\overline{N}^0_{(2n)^{k-1},2kn}(F)$ in \eqref{6.2.3} by
		\[
		\begin{aligned}
			\mathcal{L}_{2}:=&\int_{\overline{N}^0_{(2n)^{k-1},2kn}(F)}\psi_{\overline{N}^0_{(2n)^{k-1},2kn}}(u)f^{0}_{\mathcal{W}(\tau,2n,\psi),s}\left(u_0\left[\begin{smallmatrix}
				a & 0 & 0 & 0 & 0\\
				0 & 1_n & 0 & 0 & 0\\
				0 & 0 & 1_{4n(k-1)} & 0 & 0\\
				0 & 0 & 0 & 1_n & 0\\
				0 & 0 & 0 & 0 & \hat{a}
			\end{smallmatrix}\right]\left[\begin{smallmatrix}
				1_n & 0 & 0 & 0 & 0 \\
				0 & 1_n & 0 & 0 & 0 \\
				0 & 0 & \widetilde{z} & 0 & 0\\
				0 & z & 0 & 1_n & 0\\
				0 & 0 & 0 & 0 & 1_n
			\end{smallmatrix}\right]\right)du.
		\end{aligned}
		\]
		We translate $f^{0}_{\mathcal{W}(\tau,2n,\psi),s}$ on the right by
		\[
		R_{2}:=\left[
		\begin{smallmatrix}
			1_n & 0 & 0 & r & 0 & 0 & 0 & 0 & 0\\
			0 & 1_n & 0 & 0 & 0 & 0 & 0 & 0 & 0\\
			0 & 0 & 1_n & 0 & 0 & 0 & 0 & 0 & 0\\
			0 & 0 & 0 & 1_n & 0 & 0 & 0 & 0 & 0\\
			0 & 0 & 0 & 0 & 1_{4n(k-2)} & 0 & 0 & 0 & 0\\
			0 & 0 & 0 & 0 & 0 & 1_n & 0 & 0 & -r^{\ast}\\
			0 & 0 & 0 & 0 & 0 & 0 & 1_n & 0 & 0\\
			0 & 0 & 0 & 0 & 0 & 0 & 0 & 1_n & 0\\
			0 & 0 & 0 & 0 & 0 & 0 & 0 & 0 & 1_n
		\end{smallmatrix}
		\right]
		\]
		for $r\in\mathrm{Mat}_n(\mathcal{O})$. This stabilizes $f^{0}_{\mathcal{W}(\tau,2n,\psi),s}$ by \eqref{invarianta}. Conjugating this matrix to the left we have that 
		\[
		\begin{aligned}
		\mathcal{L}_{2}=\int_{\overline{N}^0_{(2n)^{k-1},2kn}(F)}\psi_{\overline{N}^0_{(2n)^{k-1},2kn}}(u)f^{0}_{\mathcal{W}(\tau,2n,\psi),s}\left(R_{2}'u\left[\begin{smallmatrix}
			a & 0 & 0 & 0 & 0\\
			0 & 1_n & 0 & 0 & 0\\
			0 & 0 & 1_{4n(k-1)} & 0 & 0\\
			0 & 0 & 0 & 1_n & 0\\
			0 & 0 & 0 & 0 & \hat{a}
		\end{smallmatrix}\right]\left[\begin{smallmatrix}
			1_n & 0 & 0 & 0 & 0 \\
			0 & 1_n & 0 & 0 & 0 \\
			0 & 0 & \widetilde{z} & 0 & 0\\
			0 & z & 0 & 1_n & 0\\
			0 & 0 & 0 & 0 & 1_n
		\end{smallmatrix}\right]\right)du
		\end{aligned}
		\]
		with
		\[
		R_{2}'=\left[\begin{smallmatrix}
		1_n & 0 & -arz & ar & 0 & 0 & 0 & 0 & 0\\
		0 & 1_n & 0 & 0 & 0 & 0 & 0 & 0 & 0 \\
		0 & 0 & 1_n & 0 & 0 & 0 & 0 & 0 & 0\\
		0 & 0 & 0 & 1_n & 0 & 0 & 0 & 0 & 0\\
		0 & 0 & 0 & 0 & 1_n & 0 & 0 & 0 & 0\\
		0 & 0 & 0 & 0 & 0 & 1_n & 0 & 0 & -r^{\ast}a^{\ast}\\
		0 & 0 & 0 & 0 & 0 & 0 & 1_n & 0 & z^{\ast}r^{\ast}a^{\ast}\\
		0 & 0 & 0 & 0 & 0 & 0 & 0 & 1_n & 0\\
		0 & 0 & 0 & 0 & 0 & 0 & 0 & 0 & 1_n
		\end{smallmatrix}\right].
		\]
		We see that
		\[
		\mathcal{L}_{2}=\psi(\mathrm{tr}(arz))\mathcal{L}_{2}.
		\]
		Hence $\mathcal{L}_{2}=0$ unless $\psi(\mathrm{tr}(arz))=1$ for all $r\in\mathrm{Mat}_n(\mathcal{O})$ which implies $z\in\mathrm{Mat}_n(\mathcal{O})$.
	\end{proof}
	
	Then \eqref{6.2.3} becomes
	\begin{equation}
		\label{6.2.4}
		\begin{aligned}
			&\int_{\mathrm{GL}_n(F)\cap\mathrm{Mat}_n(\mathcal{O})}l_T(\pi(m(a))v_0)|\det a|^{-(2nk-n+1)}\\
			\times&\int_{\overline{N}^0_{(2n)^{k-1},2kn}(F)}f^{0}_{\mathcal{W}(\tau,2n,\psi),s}\left(\left[\begin{smallmatrix}
				a & 0 & 0 & 0 & 0\\
				0 & 1_n & 0 & 0 & 0\\
				0 & 0 & 1_{4n(k-1)} & 0 & 0\\
				0 & 0 & 0 & 1_n & 0\\
				0 & 0 & 0 & 0 & \hat{a}
			\end{smallmatrix}\right]u\right)\psi_{\overline{N}^0_{(2n)^{k-1},2kn}}(u)duda.
		\end{aligned}
	\end{equation}

We now calculate the inner integral along $\overline{N}^0_{(2n)^{k-1},2kn}(F)$ in \eqref{6.2.4}. Write $u$ as in \eqref{6.2.2}, we view the integrand as a function in variables $b$, $c=(c_1,c_2)$, $d=(d_1,d_2)$, $x=(x_1,x_2,x_3,x_4)$, $z=(z_1,z_2,z_3)$ and re-denote this inner integral as
\[
\mathcal{L}(b,c,d,x,z)=\int_{b,c,d,x,z}f^{0}_{\mathcal{W}(\tau,2n,\psi),s}\left(\left[\begin{smallmatrix}
	a & 0 & 0 & 0 & 0\\
	0 & 1_n & 0 & 0 & 0\\
	0 & 0 & 1_{4n(k-1)} & 0 & 0\\
	0 & 0 & 0 & 1_n & 0\\
	0 & 0 & 0 & 0 & \hat{a}
\end{smallmatrix}\right]u(b,c,d,x,z)\right)\psi(\mathrm{tr}(x_4))du.
\]

	\begin{lem}
		\label{lem6.8}
		We have
		\[
		\mathcal{L}(b,c,d,x,z)=\int_{b,c,d,z}f^{0}_{\mathcal{W}(\tau,2n,\psi),s}\left(\left[\begin{smallmatrix}
			a & 0 & 0 & 0 & 0\\
			0 & 1_n & 0 & 0 & 0\\
			0 & 0 & 1_{4n(k-1)} & 0 & 0\\
			0 & 0 & 0 & 1_n & 0\\
			0 & 0 & 0 & 0 & \hat{a}
		\end{smallmatrix}\right]u(b,c,d,z)\right)du=:\mathcal{L}(b,c,d,z).
		\]
	\end{lem}
	
	\begin{proof}
		We need to show that the function
		\[
		\mathcal{L}'(x_1,x_2,x_3,x_4):=\int_{b,c,d,z}f^{0}_{\mathcal{W}(\tau,2n,\psi),s}\left(\left[\begin{smallmatrix}
			a & 0 & 0 & 0 & 0\\
			0 & 1_n & 0 & 0 & 0\\
			0 & 0 & 1_{4n(k-1)} & 0 & 0\\
			0 & 0 & 0 & 1_n & 0\\
			0 & 0 & 0 & 0 & \hat{a}
		\end{smallmatrix}\right]u(b,c,d,x,z)\right)\psi(\mathrm{tr}(x_4))du
		\]
		is supported on $x_1,x_2,x_3,x_4\in\mathrm{Mat}_n(\mathcal{O})$.
		
		We first treat $x_2,x_3$. The idea is again same as Lemma \ref{lem6.5} and Lemma \ref{lem6.6} but we translate $f^{0}_{\mathcal{W}(\tau,2n,\psi),s}$ on the left by
		\[
		R_{3}:=\mathrm{diag}\left[n(ar),{\cdots},n(ar),\widehat{n(ar)},{\cdots},\widehat{n(ar)}\right]
		\]
		for $r\in\mathrm{Mat}_{n}(\mathcal{O})$ with $n(ar)=\left[\begin{smallmatrix}
			1_{n} & ar\\
		0 & 1_{n}
		\end{smallmatrix}\right]$ appearing $k$ times. This stabilizes $f^{0}_{\mathcal{W}(\tau,2n,\psi),s}$ by \eqref{invariantc}. Conjugating $R_{3}$ to the right and changing variables in $u$ (especially $x_4\mapsto x_4+x_3r$), we obtain
		\[
		\begin{aligned}
		\mathcal{L}'(x_1,x_2,x_3,x_4)=&\int_{b,c,d,z}f^{0}_{\mathcal{W}(\tau,2n,\psi),s}\left(R_{3}\left[\begin{smallmatrix}
			a & 0 & 0 & 0 & 0\\
			0 & 1_n & 0 & 0 & 0\\
			0 & 0 & 1_{4n(k-1)} & 0 & 0\\
			0 & 0 & 0 & 1_n & 0\\
			0 & 0 & 0 & 0 & \hat{a}
		\end{smallmatrix}\right]u(b,c,d,x,z)\right)\psi(\mathrm{tr}(x_4))du\\
		=&\int_{b,c,d,z}f^{0}_{\mathcal{W}(\tau,2n,\psi),s}\left(\left[\begin{smallmatrix}
			a & 0 & 0 & 0 & 0\\
			0 & 1_n & 0 & 0 & 0\\
			0 & 0 & 1_{4n(k-1)} & 0 & 0\\
			0 & 0 & 0 & 1_n & 0\\
			0 & 0 & 0 & 0 & \hat{a}
		\end{smallmatrix}\right]u(b,c,d,x,z)R'_{3}\right)\psi(\mathrm{tr}(x_4+x_3r))du\\
		=&\psi(\mathrm{tr}(x_3r))\mathcal{L}'(x_1,x_2,x_3,x_4)
		\end{aligned}
		\]
		with
		\[
		R_3':=\mathrm{diag}\left[n(r),n(ar),...,n(ar),\widehat{n(ar)},...,\widehat{n(ar)},\widehat{n}(r)\right]\in\Sp_{4kn}(\mathcal{O}).
		\]
		Hence $\mathcal{L}'(x_1,x_2,x_3,x_4)=0$ unless $\psi(\mathrm{tr}(x_3r))=1$ for all $r\in\mathrm{Mat}_n(\mathcal{O})$ which implies $x_3a\in\mathrm{Mat}_n(\mathcal{O})$. Similarly, we translate $f^{0}_{\mathcal{W}(\tau,2n,\psi),s}$ on the left by
		\[
		R_{4}:=\mathrm{diag}\left[n'(r),{\cdots},n'(r),\widehat{n'(r)},{\cdots},\widehat{n'(r)}\right]
		\]
		for $r\in\mathrm{Mat}_{n}(\mathcal{O})$ with $n'(r)=\left[\begin{smallmatrix}
			1_{n} & 0\\
			r & 1_{n}
		\end{smallmatrix}\right]$ appearing $k$ times. The same process as above shows that $x_2\in\mathrm{Mat}_n(\mathcal{O})$. This shows that $\mathcal{L}'(x_1,x_2,x_3,x_4)=\mathcal{L}'(x_1,0,0,x_4)=:\mathcal{L}'(x_1,x_4)$.
		
		Now consider the variable $x_4$. Let 
		\[
		g_0=\left[\begin{smallmatrix}
			1_i & 0 & 0 & 0\\
			r  & 1_{n-i} & 0 & 0\\
			0 & 0 & 1_{n-i} & 0\\
			0 & 0 & -r^{\ast} & 1_i
		\end{smallmatrix}\right],\qquad 1\leq i\leq n-1,
		\]
		where $r$ has entries in $\mathcal{O}$. Write
		\[
		x_4=\left[\begin{smallmatrix}
			x_{41} & x_{42}\\
			x_{43} & x_{44}
		\end{smallmatrix}\right],x_{41}\in\mathrm{Mat}_i,x_{44}\in\mathrm{Mat}_{n-i}.
		\]
		We translate $f^{0}_{\mathcal{W}(\tau,2n,\psi),s}$ on both left and right by $\mathrm{diag}[g_0,{\cdots},g_0,\hat{g}_0,{\cdots},\hat{g}_0]$ with $k$ copies of $g_0$. Note that 
		\[
		\widehat{g}_0\left[\begin{smallmatrix}
			x_2 & & \\
			&x_{41} & x_{42} \\
			& x_{43} & x_{44}
		\end{smallmatrix}\right]\widehat{g}_0=\left[\begin{smallmatrix}
			\ast & & \\
			& x_{41}-x_{42}r & x_{42}\\ 
			& -rx_{41}+x_{43}+rx_{42}r-x_{44}r & -rx_{42}+x_{44}
		\end{smallmatrix}\right].
		\] 
		Changing variables in $u$ (especially $x_{41}\mapsto x_{41}+x_{42}r$ and $x_{44}\mapsto x_{44}+rx_{42}$) provides
		\[
		\begin{aligned}
		\mathcal{L}'(x_1,x_4)=\psi(\mathrm{tr}(2x_{42}r))\mathcal{L}'(x_1,x_4)
		\end{aligned}
		\]
		and hence $\mathcal{L}'(x_1,x_4)\neq 0$ implies $x_{42}\in\mathrm{Mat}_n(\mathcal{O})$. Similarly, we consider 
		\[
		g'_0=\left[\begin{smallmatrix}
			1_i & r & 0 & 0\\
			0  & 1_{n-i} & 0 & 0\\
			0 & 0 & 1_{n-i} & -r^{\ast}\\
			0 & 0 & 0 & 1_i
		\end{smallmatrix}\right],\qquad 1\leq i\leq n-1
		\]
		with $r$ has entries in $\mathcal{O}$. The same process above shows that $\mathcal{L}'(x_1,x_4)\neq 0$ implies $x_{43}\in\mathrm{Mat}_n(\mathcal{O})$ which forces $x_4\in\mathrm{Mat}_n(\mathcal{O})$. We then have $\mathcal{L}'(x_1,x_4)=\mathcal{L}'(x_1,0)=:\mathcal{L}'(x_1)$.
		
	We finally translate $f^{0}_{\mathcal{W}(\tau,2n,\psi),s}$ on the right by
		\[
		\left[\begin{smallmatrix}
			1_n & 0 & 0 & r & 0\\
			0 & 1_n & 0 & 0 & r^{\ast}\\
			0 & 0 & 1_{4n(k-1)} & 0 & 0\\
			0 & 0 & 0 & 1_n & 0\\
			0 & 0 & 0 & 0 & 1_n
		\end{smallmatrix}\right]
		\]
		for $r\in\mathrm{Mat}_n(\mathcal{O})$ and conjugate it to the left. This gives
		\[
	\mathcal{L}'(x_1)=\psi(\mathrm{tr}(x_1r))\mathcal{L}'(x_1)
		\]
	which shows that $x_1\in\mathrm{Mat}_n(\mathcal{O})$ and $\mathcal{L}'(x_1)=\mathcal{L}'(0)$ and thus the claim in the lemma.
	\end{proof}
	
	\begin{lem}
			\label{lem6.9}
	We have
	\[
	\mathcal{L}(b,c,d,z)=f^{0}_{\mathcal{W}(\tau,2n,\psi),s}\left(\left[\begin{smallmatrix}
		a & 0 & 0 & 0 & 0\\
		0 & 1_n & 0 & 0 & 0\\
		0 & 0 & 1_{4n(k-1)} & 0 & 0\\
		0 & 0 & 0 & 1_n & 0\\
		0 & 0 & 0 & 0 & \hat{a}
	\end{smallmatrix}\right]\right).
	\]
	\end{lem}
	
	\begin{proof}
		Denote
		\[
		d=\left[\begin{smallmatrix}
			d_{k-2}\\
			\vdots\\
			d_1
		\end{smallmatrix}\right],\qquad c=\left[\begin{smallmatrix}
		c_{k-2}\\
		\vdots\\
		c_1
		\end{smallmatrix}\right],\qquad\text{ with }c_i,d_i\in\mathrm{Mat}_{2n}(F).
		\]
		The computations are straightforward and exactly the same as previous lemmas. We will be brief here and only explain how we translate $f^{0}_{\mathcal{W}(\tau,2n,\psi),s}$.
		
		 For each $i=0,1,{\cdots},k-2$, we denote a matrix
		\[
		g_i=\left[\begin{smallmatrix}
			1_{2n} & 0 & 0 & 0 & 0 & r & 0 & 0\\
			0 & 1_{2ni} & 0 & 0 & 0 & 0 & 0 & 0\\
			0 & 0 & 1_{2n} & 0 & 0 & 0 & 0 & r^{\ast}\\
			0 & 0 & 0 & 1_{2n(k-2-i)} & 0 & 0 & 0 & 0\\
			0 & 0 & 0 & 0 & 1_{2n(k-2-i)} & 0 & 0 & 0\\
			0 & 0 & 0 & 0 & 0 & 1_{2n} & 0 & 0\\
			0 & 0 & 0 & 0 & 0 & 0 & 1_{2ni} & 0\\
			0 & 0 & 0 & 0 & 0 & 0 & 0 & 1_{2n}
		\end{smallmatrix}\right]
		\]
		with $r\in\mathrm{Mat}_{2n}(\mathcal{O})$. Translating $f^{0}_{\mathcal{W}(\tau,2n,\psi),s}$ on the right by $g_0$ and conjugating it to the left show that $d_1$ and $z$ are supported in $\mathrm{Mat}_{2n}(\mathcal{O})$. Then translating $f^{0}_{\mathcal{W}(\tau,2n,\psi),s}$ on the right by $g_i$ for $i=1,{\cdots},k-3$ (in this order) and conjugating it to the left show that $d_{i+1}$ and $c_i$ are supported in $\mathrm{Mat}_{2n}(\mathcal{O})$. Finally translating $f^{0}_{\mathcal{W}(\tau,2n,\psi),s}$ on the right by $g_{k-2}$ shows that $c_{k-2}$ is supported in $\mathrm{Mat}_{2n}(\mathcal{O})$.
		
		 For each $j=1,2,{\cdots},k-2$ and $i=0,1,{\cdots},k-2-j$, we use the matrix
		\[
		g_{ji}=\left[\begin{smallmatrix}
			1_{2nj} & 0 & 0 & 0 & 0 & 0 & 0 & 0 & 0 & 0\\
			0 & 1_{2n} & 0 & 0 & 0 & 0 & r & 0 & 0 & 0\\
			0 & 0 & 1_{2ni} & 0 & 0 & 0 & 0 & 0 & 0 & 0\\
			0 & 0 & 0 & 1_{2n} & 0 & 0 & 0 & 0 & r^{\ast} & 0\\
			0 & 0 & 0 & 0 & 1_{2n(k-2-i-j)} & 0 & 0 & 0 & 0 & 0\\
			0 & 0 & 0 & 0 & 0 & 1_{2n(k-2-i-j)} & 0 & 0 & 0 & 0\\
			0 & 0 & 0 & 0 & 0 & 0 & 1_{2n} & 0 & 0 & 0\\
			0 & 0 & 0 & 0 & 0 & 0 & 0 & 1_{2ni} & 0 & 0\\
			0 & 0 & 0 & 0 & 0 & 0 & 0 & 0 & 1_{2n} & 0\\
			0 & 0 & 0 & 0 & 0 & 0 & 0 & 0 & 0 & 1_{2nj}
		\end{smallmatrix}\right]
		\]
		with $r\in\mathrm{Mat}_{2n}(\mathcal{O})$. For each fixed $j=1,2,{\cdots},k-2$ (in this order) we translate $f^0_{\mathcal{W}(\tau,2n,\psi),s}$ on the right by $g_{ji}$ for $i=0,{\cdots},k-2-j$ (in this order) and 
		we conjugate it to the left. This shows that the entries of the $j$-th column (viewed in $2n\times 2n$ blocks) of $b$ are supported in $\mathcal{O}$ and thus $b$ is supported in $\mathrm{Mat}_{2n(k-2)}(\mathcal{O})$. This proves the lemma.
	\end{proof}

We conclude that \eqref{6.2.4} equals
\[
\begin{aligned}
\int_{\mathrm{GL}_n(F)\cap\mathrm{Mat}_n(\mathcal{O})}l_T(\pi(m(a))v_0)|\det a|^{-(2nk-n+1)}f^{0}_{\mathcal{W}(\tau,2n,\psi),s}\left(\left[\begin{smallmatrix}
	a & 0 & 0 & 0 & 0\\
	0 & 1_n & 0 & 0 & 0\\
	0 & 0 & 1_{4n(k-1)} & 0 & 0\\
	0 & 0 & 0 & 1_n & 0\\
	0 & 0 & 0 & 0 & \hat{a}
\end{smallmatrix}\right]\right)da,
\end{aligned}
\]
which completes the proof of Proposition \ref{prop6.4}.

	\subsection{Reformulating the unramified New Way integral}\label{sec6.3}
	
	We reformulate our integral
\[\mathcal{Z}^{\ast}(l_T,s)=\int_{N_n(F)\backslash\mathrm{Sp}_{2n}(F)}\int_{N_{n^{k-1},kn}^0(F)}l_T(\pi(h)v_0)\\
\omega_{\psi}(\alpha_T^k(u)i_T(1,h))\Phi^0(1_n)f^{0}_{\mathcal{W}(\tau\otimes\chi_T,n,\psi_{2T}),s}(\eta ut(1,h))dudh.\]
and prove the following.
	
	\begin{prop}\label{prop6.9}
	For $\mathrm{Re}(s)\gg0$, we have
		\begin{equation}
			\begin{aligned}
				\mathcal{Z}^{\ast}(l_T,s)=\int_{\mathrm{GL}_n(F)\cap\mathrm{Mat}_n(\mathcal{O})}l_T(\pi(m(a))v_0)|\det a|^{-(kn-\frac{n}{2}+1)}f^{0}_{\mathcal{W}(\tau,n,\psi),s}\left(\left[\begin{smallmatrix}
					a & 0 & 0\\
					0 & 1_{2n(k-1)} & 0\\
					0 & 0 & \hat{a}
				\end{smallmatrix}\right])\right)da.
			\end{aligned}
		\end{equation}
	\end{prop}
	
	The proof is similar to the computations in Section \ref{sec6.2}. We start by the Iwasawa decomposition of $\Sp_{2n}(F)$ with respect to the Siegel parabolic subgroup. We have
	\[
	\begin{aligned}
		\mathcal{Z}^{\ast}(l_T,s)&=\int_{\mathrm{GL}_n(F)}\int_{N_{n^{k-1},kn}^0(F)}l_T(\pi(m(a))v_0)\\
		\times&\omega_{\psi}(\alpha_T^k(u)i_T(1,m(a)))\Phi^0(1_n)f^{0}_{\mathcal{W}(\tau\otimes\chi_T,n,\psi_{2T}),s}(\eta ut(1,m(a)))|\det a|^{-n-1}duda.
	\end{aligned}
	\]
	Changing variables $u\mapsto t(1,m(a))ut(1,m(a))$ and using the formulas of the Weil representation \eqref{weil2}, we obtain that $\mathcal{Z}^{\ast}(l_T,s)$ is equal to
	\begin{equation}
		\label{6.3.1}
		\begin{aligned}
			&\int_{\mathrm{GL}_n(F)}\int_{N_{n^{k-1},kn}^0(F)}l_T(\pi(m(a))v_0)\chi_T(\det a)|\det a|^{-(kn-\frac{n}{2}+1)}\\
			\times&\omega_{\psi}(\alpha_T^k(u))\Phi^0(a)f^{0}_{\mathcal{W}(\tau\otimes\chi_T,n,\psi_{2T}),s}(\eta t(1,m(a))u)duda.
		\end{aligned}
	\end{equation}
	We write $u=u(x,0,z)$ as in \eqref{2.1.1} and \eqref{2.1.2}. Then the inner integral over $N_{n^{k-1},kn}^0(F)$ in \eqref{6.3.1} is
	\[
	\begin{aligned}
		&\int_{N_{n^{k-1},kn}^0(F)}\psi(\mathrm{tr}(Tz))\Phi^0(a+x)f^{0}_{\mathcal{W}(\tau\otimes\chi_T,n,\psi_{2T}),s}(\eta t(1,m(a))u(x,0,z))du\\
		=&\int_{N_{n^{k-1},kn}^0(F)}\psi(\mathrm{tr}(Tz))\Phi^0(x)f^{0}_{\mathcal{W}(\tau\otimes\chi_T,n,\psi_{2T}),s}(\eta t(1,m(a))u(x-a,0,z))du\\
		=&\int_{N_{n^{k-1},kn}^0(F)}\psi(\mathrm{tr}(Tz))f^{0}_{\mathcal{W}(\tau\otimes\chi_T,n,\psi_{2T}),s}(\eta t(1,m(a))u(-a,0,z))du.
	\end{aligned}
	\]
	Thus integral \eqref{6.3.1} becomes
	\begin{equation}
		\label{6.3.2}
		\begin{aligned}
			&\int_{\mathrm{GL}_n(F)}\int_{N_{n^{k-1},kn}^{0,a}(F)}l_T(\pi(m(a))v_0)\chi_T(\det a)|\det a|^{-(kn-\frac{n}{2}+1)}\\
			\times&f^{0}_{\mathcal{W}(\tau\otimes\chi_T,n,\psi_{2T}),s}(\eta t(1,m(a))u(-a,0,z))\psi(\mathrm{tr}(Tz))duda.
		\end{aligned}
	\end{equation}
	Here, $N_{n^{k-1},kn}^{00}(F)$ is the subgroup of $N_{n^{k-1},kn}^{0}(F)$ containing elements of the form
	\[
	\left[\begin{smallmatrix}
		1_{(k-2)n} &0 & -b & 0 & -c & d\\
		0 & 1_n & -a & 0 & z & c^{\ast}\\
		0 & 0 & 1_n & 0 & 0 & 0\\
		0 & 0 & 0 & 1_n & a^{\ast} & b^{\ast}\\
		0 & 0 & 0 & 0 & 1_n & 0\\
		0 & 0 & 0 & 0 & 0 & 1_n
	\end{smallmatrix}\right].
	\]
	With $u\in N_{n^{k-1},kn}^{00}(F)$ of the above form, we write \eqref{6.3.2} as
	\[
	\begin{aligned}
		&\int_{\mathrm{GL}_n(F)}\int_{N_{n^{k-1},kn}^{00}(F)}l_T(\pi(m(a))v_0)\chi_T(\det a)|\det a|^{-(kn-\frac{n}{2}+1)}\psi(\mathrm{tr}(Tz))\\
		\times&f^{0}_{\mathcal{W}(\tau\otimes\chi_T,n,\psi_{2T}),s}\left(\left[\begin{smallmatrix}
			a & 0 & 0\\
			0 & 1_{2n(k-1)} & 0\\
			0 & 0 & \hat{a}
		\end{smallmatrix}\right]\left[\begin{smallmatrix}
			1_n & 0 & 0 &0 & 0 & 0\\
			0 & 1_n & 0 & 0 & 0 & 0\\
			0 & 0 & 1_{(k-2)n} & 0 & 0 & 0\\
			b & c & d & 1_{(k-2)n} & 0 & 0\\
			-a & z & c^{\ast} & 0 & 1_n & 0\\
			0 & -a^{\ast} & b^{\ast} & 0 & 0 & 1_n
		\end{smallmatrix}\right]\right)duda.
	\end{aligned}
	\]
	Applying Corollary~\ref{cor-relation-sections} we see that the above integral is equal to
	\begin{equation}
		\label{formulalem6.10}
	\begin{aligned}
		&\int_{\mathrm{GL}_n(F)}\int_{N_{n^{k-1},kn}^{00}(F)}l_T(\pi(m(a))v_0)\chi_T(\det a)|\det a|^{-(kn-\frac{n}{2}+1)}\psi(\mathrm{tr}(4T^2z))\\
		\times&f^{0}_{\mathcal{W}(\tau\otimes\chi_T,n,\psi),s}\left(\left[\begin{smallmatrix}
			a & 0 & 0\\
			0 & 1_{2n(k-1)} & 0\\
			0 & 0 & \hat{a}
		\end{smallmatrix}\right]\left[\begin{smallmatrix}
			1_n & 0 & 0 &0 & 0 & 0\\
			0 & 1_n & 0 & 0 & 0 & 0\\
			0 & 0 & 1_{(k-2)n} & 0 & 0 & 0\\
			b & c & d & 1_{(k-2)n} & 0 & 0\\
			-a & z & c^{\ast} & 0 & 1_n & 0\\
			0 & -a^{\ast} & b^{\ast} & 0 & 0 & 1_n
		\end{smallmatrix}\right]\right)duda.
	\end{aligned}
	\end{equation}

	\begin{lem}
		The inner integral along $N_{n^{k-1},kn}^{00}(F)$ in \eqref{formulalem6.10} vanishes unless $a\in\mathrm{Mat}_n(\mathcal{O})$.
	\end{lem}
	
	\begin{proof}
		The proof is same as Lemma \ref{lem6.5} and \ref{lem6.6}. Denote the integral along $N_{n^{k-1},kn}^{00}(F)$ in \eqref{formulalem6.10} by $\mathcal{L}$. We translate $f^{0}_{\mathcal{W}(\tau\otimes\chi_T,n,\psi),s}$ by $u^0(r,0,0)$ for $r\in\mathrm{Mat}_n(\mathcal{O})$ on the right and conjugate it to the left which shows that
		\[
		\mathcal{L}=\psi(\mathrm{tr}(ar))\mathcal{L}.
		\]
		Hence $\mathcal{L}=0$ unless $\psi(\mathrm{tr}(ar))=1$ for all $r\in\mathrm{Mat}_n(\mathcal{O})$ which implies $a\in\mathrm{Mat}_n(\mathcal{O})$.
	\end{proof}

	Denote $\overline{N}^0_{n^{k-1},kn}$ for the unipotent group consisting elements of the form
	\[
u(b,c,d,z):=	\left[\begin{smallmatrix}
		1_n & 0 & 0 & 0 & 0 & 0\\
		0 & 1_n & 0 & 0 & 0 & 0\\
		0 & 0 & 1_{(k-2)n} & 0 & 0 & 0\\
		b & c & d & 1_{(k-2)n} & 0 & 0\\
		0 & z & c^{\ast} & 0 & 1_n & 0\\
		0 & 0 & b^{\ast} & 0 & 0 & 1_n
	\end{smallmatrix}\right]
	\]
	and define 
	\[
	\psi_{\overline{N}^0_{n^{k-1},kn}}(u):=\psi(\mathrm{tr}(4T^2z))
	\]
	for $u\in\overline{N}^0_{n^{k-1},kn}(F)$ of the above form.
	The integral \eqref{formulalem6.10} then becomes
	\begin{equation}
		\label{672}
	\begin{aligned}
		&\int_{\mathrm{GL}_n(F)\cap\mathrm{Mat}_n(\mathcal{O})}\int_{\overline{N}^0_{n^{k-1},kn}(F)}l_T(\pi(m(a))v_0)\chi_T(\det a)|\det a|^{-(kn-\frac{n}{2}+1)}\psi(\mathrm{tr}(4T^2z))\\
		\times&f^{0}_{\mathcal{W}(\tau\otimes\chi_T,n,\psi),s}\left(\left[\begin{smallmatrix}
			a & 0 & 0\\
			0 & 1_{2n(k-1)} & 0\\
			0 & 0 & \hat{a}
		\end{smallmatrix}\right]u(b,c,d,z)\right)duda.
	\end{aligned}
	\end{equation}

	\begin{lem}
		\label{lem6.10}
		We have
		\[
		\begin{aligned}
			\mathcal{L}(b,c,d,z):=&\int_{b,c,d,z}\psi(\mathrm{tr}(4T^2z))f^{0}_{\mathcal{W}(\tau\otimes\chi_T,n,\psi),s}\left(\left[\begin{smallmatrix}
				a & 0 & 0\\
				0 & 1_{2n(k-1)} & 0\\
				0 & 0 & \hat{a}
			\end{smallmatrix}\right]u(b,c,d,z)\right)du\\
			=&f^{0}_{\mathcal{W}(\tau\otimes\chi_T,n,\psi),s}\left(\left[\begin{smallmatrix}
				a & 0 & 0\\
				0 & 1_{2n(k-1)} & 0\\
				0 & 0 & \hat{a}
			\end{smallmatrix}\right]\right).
		\end{aligned}
		\]
		\end{lem}
		
	\begin{proof}
		This is the same as Lemma \ref{lem6.8} and \ref{lem6.9}. Indeed, for each $j=0,1,2,{\cdots},k-2$ and $i=0,1,{\cdots},k-2-j$, consider the matrix
		\[
		g_{ji}=\left[\begin{smallmatrix}
			1_{nj} & 0 & 0 & 0 & 0 & 0 & 0 & 0 & 0 & 0\\
			0 & 1_{n} & 0 & 0 & 0 & 0 & r & 0 & 0 & 0\\
			0 & 0 & 1_{ni} & 0 & 0 & 0 & 0 & 0 & 0 & 0\\
			0 & 0 & 0 & 1_{n} & 0 & 0 & 0 & 0 & r^{\ast} & 0\\
			0 & 0 & 0 & 0 & 1_{n(k-2-i-j)} & 0 & 0 & 0 & 0 & 0\\
			0 & 0 & 0 & 0 & 0 & 1_{n(k-2-i-j)} & 0 & 0 & 0 & 0\\
			0 & 0 & 0 & 0 & 0 & 0 & 1_{n} & 0 & 0 & 0\\
			0 & 0 & 0 & 0 & 0 & 0 & 0 & 1_{ni} & 0 & 0\\
			0 & 0 & 0 & 0 & 0 & 0 & 0 & 0 & 1_{n} & 0\\
			0 & 0 & 0 & 0 & 0 & 0 & 0 & 0 & 0 & 1_{nj}
		\end{smallmatrix}\right]
		\]
		with $r\in\mathrm{Mat}_n(\mathcal{O})$. For each fixed $j=0,1,2,{\cdots},k-2$ (in this order) we translate $f^0_{\mathcal{W}(\tau\otimes\chi_T,n,\psi),s}$ on the right by $g_{ji}$ for $i=0,{\cdots},k-2-j$ (in this order) and conjugate it to the left. This shows that $\mathcal{L}(b,c,d,z)\neq 0$ unless $b,c,d,z$ has entries in $\mathcal{O}$ which implies the lemma.
	\end{proof}
	
	We conclude that \eqref{672} equals
	\[
		\begin{aligned}
		&\int_{\mathrm{GL}_n(F)\cap\mathrm{Mat}_n(\mathcal{O})}l_T(\pi(m(a))v_0)\chi_T(\det a)|\det a|^{-(kn-\frac{n}{2}+1)}f^{0}_{\mathcal{W}(\tau\otimes\chi_T,n,\psi),s}\left(\left[\begin{smallmatrix}
			a & 0 & 0\\
			0 & 1_{2n(k-1)} & 0\\
			0 & 0 & \hat{a}
		\end{smallmatrix}\right]\right)da.
	\end{aligned}
	\]
	Note that since $\chi_T$ is quadratic, we have 
	\[
\chi_T(\det a)f^{0}_{\mathcal{W}(\tau\otimes\chi_T,n,\psi),s}\left(\left[\begin{smallmatrix}
	a & 0 & 0\\
	0 & 1_{2n(k-1)} & 0\\
	0 & 0 & \hat{a}
\end{smallmatrix}\right]\right)=f^{0}_{\mathcal{W}(\tau,n,\psi),s}\left(\left[\begin{smallmatrix}
a & 0 & 0\\
0 & 1_{2n(k-1)} & 0\\
0 & 0 & \hat{a}
\end{smallmatrix}\right]\right)
	\]
	which completes the proof of Proposition \ref{prop6.9}.

	\subsection{Proof of Theorem~\ref{thm-unramified-computation}}\label{sec6.4}
	
We are now going to prove Theorem~\ref{mainthm} (=Theorem~\ref{thm-unramified-computation}). By Proposition \ref{prop6.4} and Proposition \ref{prop6.9}, it amounts to compare
	\[
	f^{0}_{\mathcal{W}(\tau,2n,\psi),s}\left(\left[\begin{smallmatrix}
		a & 0 & 0 & 0 & 0\\
		0 & 1_n & 0 & 0 & 0\\
		0 & 0 & 1_{4n(k-1)} & 0 & 0\\
		0 & 0 & 0 & 1_n & 0\\
		0 & 0 & 0 & 0 & \hat{a}
	\end{smallmatrix}\right]\right)\qquad\text{ and }\qquad f^{0}_{\mathcal{W}(\tau,n,\psi),s}\left(\left[\begin{smallmatrix}
	a & 0 & 0\\
	0 & 1_{2n(k-1)} & 0\\
	0 & 0 & \hat{a}
	\end{smallmatrix}\right]\right).
	\]
The main tool is the decomposition of $(k,2n)$-functionals (\cite[Lemma 22]{CaiFriedbergGinzburgKaplan2019}, see also \cite{CaiFriedbergGourevitchKaplan2021}).


Recall that the section $\widetilde{f}^{0}_{\mathcal{W}(\tau,n,\psi),s}\in\mathrm{Ind}_{P_{kn}(F)}^{\Sp_{2kn}(F)}(\mathcal{W}(\tau,n,\psi)|\det\cdot|^s)$ is a smooth function
\[
\widetilde{f}^{0}_{\mathcal{W}(\tau,n,\psi),s}:\Sp_{2kn}(F)\to\mathcal{W}(\tau,n,\psi)
\]
satisfying
\[
\widetilde{f}^{0}_{\mathcal{W}(\tau,n,\psi),s}(n(z)m(a)g)=|\det a|^{s+\frac{kn+1}{2}}\Delta(\tau,n)(a)\widetilde{f}^0(g),
\]
where $\mathcal{W}(\tau,n,\psi)$ is the unique $(k,n)$-model of $\Delta(\tau,n)$ with respect to $\psi$. Due to the spherical nature of $\widetilde{f}^{0}_{\mathcal{W}(\tau,n,\psi),s}$, it is easy to see that $\widetilde{f}^{0}_{\mathcal{W}(\tau,n,\psi),s}(1_{2kn})$ is an unramified function in $\mathcal{W}(\tau,n,\psi)$. Let $W^0_{\tau,n,\psi}$ be the unique unramified function in $\mathcal{W}(\tau,n,\psi)$ normalized such that $W^0_{\tau,n,\psi}(1_{kn})=1$. By our definition $f^{0}_{\mathcal{W}(\tau,n,\psi),s}(g)=\widetilde{f}^{0}_{\mathcal{W}(\tau,n,\psi),s}(g)(1_{kn})$, we shall write $\widetilde{f}^{0}_{\mathcal{W}(\tau,n,\psi),s}(1_{2kn})=d_{\tau}^{\Sp_{4kn}}(s)W^0_{\tau,n,\psi}$. Clearly, we have
\begin{equation}
\label{AAA}
f^{0}_{\mathcal{W}(\tau,n,\psi),s}\left(\left[\begin{smallmatrix}
	a & 0 & 0\\
	0 & 1_{2n(k-1)} & 0\\
	0 & 0 & \hat{a}
	\end{smallmatrix}\right]\right)=d_{\tau}^{\Sp_{4kn}}(s)|\det a|^{s+\frac{kn+1}{2}}\left(\Delta(\tau,n)\left(\left[\begin{smallmatrix}a & 0\\
0 & 1_{n(k-1)}\end{smallmatrix}\right]\right)W^0_{\tau,n,\psi}\right)(1_{kn}).
\end{equation}

Consider the section $\widetilde{f}^{0}_{\mathcal{W}(\tau,2n,\psi),s}\in\mathrm{Ind}_{P_{2n}(F)}^{\Sp_{4kn}(F)}(\mathcal{W}(\tau,2n,\psi)|\det\cdot|^s)$ which is a smooth function $\Sp_{4kn}(F)\to\mathcal{W}(\tau,2n,\psi)$. To compare with the section $\widetilde{f}^{0}_{\mathcal{W}(\tau,2n,\psi),s}$, we use the decomposition of $(k,2n)$-functionals. For $u\in U_{(2n)^k}$ of the form in \eqref{2.3.1}, we write $u=[u_{i,j}]_{1\leq i,j\leq k}$ with $u_{i,j}\in\mathrm{Mat}_{2n}$. We further write the block $u_{i,j}$ in the form
\[
u_{i,j}=\left[\begin{array}{cc}
u_{i,j}^1 & u_{i,j}^2\\
u_{i,j}^3 & u_{i,j}^4
\end{array}\right],\qquad u_{i,j}^1,u_{i,j}^2,u_{i,j}^3,u_{i,j}^4\in\mathrm{Mat}_n.
\]
Let $U_{(2n)^k}^3$ be the subgroup of $U_{(2n)^k}$ consisting of matrices $u\in U_{(2n)^k}$ such that in each block $u_{i,j}$ with $i<j$, the coordinates of $u_{i,j}^1,u_{i,j}^2,u_{i,j}^4$ are zero. Set
\[
l_{n,n}=\left[\begin{smallmatrix}
1_n\\
0 & 0 & 1_n\\
0 & 0 & 0 & 0 & 1_n & \ddots\\
 & & & & & & 1_n & 0\\
0 & 1_n\\
0 & 0 & 0 & 1_n & & \ddots\\
& & & & & & 0 & 1_n
\end{smallmatrix}\right]\in\GL_{2kn}.
\]
We have the following lemma which is essentially due to \cite[Lemma 22]{CaiFriedbergGinzburgKaplan2019} (see also the computations in \cite[page. 1030]{CaiFriedbergGinzburgKaplan2019}). 

\begin{lem}
Let $P_{(kn,kn)}$ be the standard parabolic subgroup of $\Sp_{4kn}$ whose Levi component is $\GL_{kn}\times\GL_{kn}$. There is an unramified section
\[
\widetilde{f}^0_{\mathcal{W}(\tau,n,\psi)\otimes\mathcal{W}(\tau,n,\psi),s}\in\mathrm{Ind}_{P_{(kn,kn)}(F)}^{\Sp_{4kn}(F)}\left(\mathcal{W}(\tau,n,\psi)|\det\cdot|^{s-\frac{n}{2}}\otimes\mathcal{W}(\tau,n,\psi)|\det\cdot|^{s+\frac{n}{2}}\right)
\]
such that
\begin{equation}
\label{lem22A}
f^0_{\mathcal{W}(\tau,2n,\psi),s}(h)=\int_{U_{(2n)^k}^3(F)}\widetilde{f}^0_{\mathcal{W}(\tau,n,\psi)\otimes\mathcal{W}(\tau,n,\psi),s}(l_{n,n}uh)du.
\end{equation}
\end{lem}

By \eqref{lem22A}, we have that
\[
f^0_{\mathcal{W}(\tau,2n,\psi),s}(1_{4kn})=\int_{U_{(2n)^k}^3(F)}\widetilde{f}^0_{\mathcal{W}(\tau,n,\psi)\otimes\mathcal{W}(\tau,n,\psi),s}(l_{n,n}u)du=\widetilde{f}^0_{\mathcal{W}(\tau,n,\psi)\otimes\mathcal{W}(\tau,n,\psi),s}(1_{4kn})
\]
where the second equality follows by the proof of \cite[Corollary 23]{CaiFriedbergGinzburgKaplan2019}. It is clear that $\widetilde{f}^0_{\mathcal{W}(\tau,n,\psi)\otimes\mathcal{W}(\tau,n,\psi),s}(1_{4kn})$ is an unramified function in $\mathcal{W}(\tau,n,\psi)\otimes\mathcal{W}(\tau,n,\psi)$ which is a multiple of $W_{\tau,n,\psi}^0\otimes W_{\tau,n,\psi}^0$ and thus
\[
\widetilde{f}^0_{\mathcal{W}(\tau,n,\psi)\otimes\mathcal{W}(\tau,n,\psi),s}(1_{4kn})=d_{\tau}^{\Sp_{4kn}}(s)W_{\tau,n,\psi}^0\otimes W_{\tau,n,\psi}^0.
\]
Note that the matrix
\[
\left[\begin{smallmatrix}
		a & 0 & 0 & 0 & 0\\
		0 & 1_n & 0 & 0 & 0\\
		0 & 0 & 1_{4n(k-1)} & 0 & 0\\
		0 & 0 & 0 & 1_n & 0\\
		0 & 0 & 0 & 0 & \hat{a}
	\end{smallmatrix}\right]
\]
commutes with $l_{n,n}u$ and we have
\begin{equation}
\label{BBB}
\begin{aligned}
f^0_{\mathcal{W}(\tau,2n,\psi),s}\left(\left[\begin{smallmatrix}
		a & 0 & 0 & 0 & 0\\
		0 & 1_n & 0 & 0 & 0\\
		0 & 0 & 1_{4n(k-1)} & 0 & 0\\
		0 & 0 & 0 & 1_n & 0\\
		0 & 0 & 0 & 0 & \hat{a}
	\end{smallmatrix}\right]\right)&=\int_{U_{(2n)^k}^3(F)}\widetilde{f}^0_{\mathcal{W}(\tau,n,\psi)\otimes\mathcal{W}(\tau,n,\psi),s}\left(\left[\begin{smallmatrix}
		a & 0 & 0 & 0 & 0\\
		0 & 1_n & 0 & 0 & 0\\
		0 & 0 & 1_{4n(k-1)} & 0 & 0\\
		0 & 0 & 0 & 1_n & 0\\
		0 & 0 & 0 & 0 & \hat{a}
	\end{smallmatrix}\right]l_{n,n}u\right)du\\
&=d_{\tau}^{\Sp_{4kn}}(s)|\det a|^{s+\frac{3kn+1}{2}-\frac{n}{2}}\left(\Delta(\tau,n)\left(\left[\begin{smallmatrix}
a & 0\\
0 & 1_{n(k-1)}
\end{smallmatrix}\right]\right)W_{\tau,n,\psi}^0\right)(1_{2kn}).
\end{aligned}
\end{equation}
Comparing \eqref{AAA} with \eqref{BBB}, we obtain
\begin{equation}
\label{eq-unramified-identity-of-sections}
f^0_{\mathcal{W}(\tau,2n,\psi),s}\left(\left[\begin{smallmatrix}
		a & 0 & 0 & 0 & 0\\
		0 & 1_n & 0 & 0 & 0\\
		0 & 0 & 1_{4n(k-1)} & 0 & 0\\
		0 & 0 & 0 & 1_n & 0\\
		0 & 0 & 0 & 0 & \hat{a}
	\end{smallmatrix}\right]\right)=|\det a|^{kn-\frac{n}{2}}f^{0}_{\mathcal{W}(\tau,n,\psi),s}\left(\left[\begin{smallmatrix}
	a & 0 & 0\\
	0 & 1_{2n(k-1)} & 0\\
	0 & 0 & \hat{a}
	\end{smallmatrix}\right]\right).
\end{equation}
We conclude that, for $\mathrm{Re}(s)\gg0$, we have
\begin{equation*}
\begin{split}
&\mathcal{Z}^{\ast}(l_T,s) \\
=&\int_{\mathrm{GL}_n(F)\cap\mathrm{Mat}_n(\mathcal{O})}l_T(\pi(m(a))v_0)|\det a|^{-(kn-\frac{n}{2}+1)}f^{0}_{\mathcal{W}(\tau,n,\psi),s}\left(\left[\begin{smallmatrix}
					a & 0 & 0\\
					0 & 1_{2n(k-1)} & 0\\
					0 & 0 & \hat{a}
				\end{smallmatrix}\right])\right)da \quad \text{(by Proposition \ref{prop6.9})} \\
=& \int_{\mathrm{GL}_n(F)\cap\mathrm{Mat}_n(\mathcal{O})}l_T(\pi(m(a))v_0)|\det a|^{-(2nk-n+1)}f^0_{\mathcal{W}(\tau,2n,\psi),s}\left(\left[\begin{smallmatrix}
		a & 0 & 0 & 0 & 0\\
		0 & 1_n & 0 & 0 & 0\\
		0 & 0 & 1_{4n(k-1)} & 0 & 0\\
		0 & 0 & 0 & 1_n & 0\\
		0 & 0 & 0 & 0 & \hat{a}
	\end{smallmatrix}\right]\right) da \quad \text{(by \eqref{eq-unramified-identity-of-sections})} \\
=& \int_{\mathrm{Sp}_{2n}(F)}\int_{N^0_{(2n)^{k-1},2kn}(F)}l_T(\pi(h)v_0)f^{0}_{\mathcal{W}(\tau,2n,\psi),s}(\delta u(1\times {^{\iota}h}))\psi_{N^0_{(2n)^{k-1},2kn}}(u)dudh  \quad \text{(by Proposition \ref{prop6.4})} \\
=&  L(s+\frac{1}{2},\pi\times\tau)\cdot l_T(v_0)  \quad \text{(by Lemma \ref{lemma6.4})}.
\end{split}
\end{equation*}
This completes the proof of Theorem \ref{mainthm}.

\section*{Acknowledgements} 

We are grateful to Thanasis Bouganis, Jim Cogdell and Hang Xue for their support and encouragement. We greatly appreciate the anonymous referees for their thorough review of the paper and very helpful suggestions that have significantly improved the paper. 
In particular, we would like to thank one referee for suggesting us to consider the realizations of $(k,c)$-functionals with respect to composition of $c$ in \cite{CaiFriedbergGourevitchKaplan2021}, which enabled us to complete the unramified computation for all even $n$ and all $k$, thereby removing the condition that $k$ divides $n$ in a previous version. PY is partially supported by an AMS-Simons Travel Grant.

\bibliographystyle{alpha}
\bibliography{References}

\begin{thebibliography}{CFGK23}

\bibitem[BFG95]{BumpFurusawaGinzburg1995}
Daniel Bump, Masaaki Furusawa, and David Ginzburg.
\newblock Non-unique models in the {R}ankin-{S}elberg method.
\newblock {\em J. Reine Angew. Math.}, 468:77--111, 1995.

\bibitem[Bum97]{Bump1997}
Daniel Bump.
\newblock {\em Automorphic forms and representations}, volume~55 of {\em
  Cambridge Studies in Advanced Mathematics}.
\newblock Cambridge University Press, Cambridge, 1997.

\bibitem[CFGK19]{CaiFriedbergGinzburgKaplan2019}
Yuanqing Cai, Solomon Friedberg, David Ginzburg, and Eyal Kaplan.
\newblock Doubling constructions and tensor product {$L$}-functions: the linear
  case.
\newblock {\em Invent. Math.}, 217(3):985--1068, 2019.

\bibitem[CFGK23]{CaiFriedbergGourevitchKaplan2021}
Yuanqing Cai, Solomon Friedberg, Dmitry Gourevitch, and Eyal Kaplan.
\newblock The generalized doubling method: {$(k,c)$} models.
\newblock {\em Proc. Amer. Math. Soc.}, 151(7):2831--2845, 2023.

\bibitem[CFK22]{CaiFriedbergKaplan2022}
Yuanqing Cai, Solomon Friedberg, and Eyal Kaplan.
\newblock The generalized doubling method: local theory.
\newblock {\em Geom. Funct. Anal.}, 32(6):1233--1333, 2022.

\bibitem[Gin18]{Ginzburg2018}
David Ginzburg.
\newblock Generating functions on covering groups.
\newblock {\em Compos. Math.}, 154(4):671--684, 2018.

\bibitem[GJRS11]{GinzburgJiangRallisSoudry2011}
David Ginzburg, Dihua Jiang, Stephen Rallis, and David Soudry.
\newblock {$L$}-functions for symplectic groups using {F}ourier-{J}acobi
  models.
\newblock In {\em Arithmetic geometry and automorphic forms}, volume~19 of {\em
  Adv. Lect. Math. (ALM)}, pages 183--207. Int. Press, Somerville, MA, 2011.

\bibitem[GRS98]{GinzburgRallisSoudry1998}
David Ginzburg, Stephen Rallis, and David Soudry.
\newblock {$L$}-functions for symplectic groups.
\newblock {\em Bull. Soc. Math. France}, 126(2):181--244, 1998.

\bibitem[GRS11]{GinzburgRallisSoudry2011}
David Ginzburg, Stephen Rallis, and David Soudry.
\newblock {\em The descent map from automorphic representations of {${\rm
  GL}(n)$} to classical groups}.
\newblock World Scientific Publishing Co. Pte. Ltd., Hackensack, NJ, 2011.

\bibitem[GS15]{GurevichSegal2015}
Nadya Gurevich and Avner Segal.
\newblock The {R}ankin-{S}elberg integral with a non-unique model for the
  standard {$L$}-function of {$G_2$}.
\newblock {\em J. Inst. Math. Jussieu}, 14(1):149--184, 2015.

\bibitem[GS20]{GinzburgSoudry2020}
David Ginzburg and David Soudry.
\newblock Integrals derived from the doubling method.
\newblock {\em Int. Math. Res. Not. IMRN}, (24):10553--10596, 2020.

\bibitem[GS21]{GinzburgSoudry2021}
David Ginzburg and David Soudry.
\newblock Two identities relating {E}isenstein series on classical groups.
\newblock {\em J. Number Theory}, 221:1--108, 2021.

\bibitem[GS24]{GinzburgSoudry2024}
David Ginzburg and David Soudry.
\newblock A new regularized {S}iegel-{W}eil type formula. {P}art {I}.
\newblock {\em Geom. Funct. Anal.}, 34(1):60--112, 2024.

\bibitem[Ike94]{Ikeda1994}
Tamotsu Ikeda.
\newblock On the theory of {J}acobi forms and {F}ourier-{J}acobi coefficients
  of {E}isenstein series.
\newblock {\em J. Math. Kyoto Univ.}, 34(3):615--636, 1994.

\bibitem[Jac84]{Jacquet1984}
Herv\'{e} Jacquet.
\newblock On the residual spectrum of {${\rm GL}(n)$}.
\newblock In {\em Lie group representations, {II} ({C}ollege {P}ark, {M}d.,
  1982/1983)}, volume 1041 of {\em Lecture Notes in Math.}, pages 185--208.
  Springer, Berlin, 1984.

\bibitem[KRS92]{KudlaRallisSoudry1992}
Stephen~S. Kudla, Stephen Rallis, and David Soudry.
\newblock On the degree {$5$} {$L$}-function for {${\rm Sp}(2)$}.
\newblock {\em Invent. Math.}, 107(3):483--541, 1992.

\bibitem[Li92]{LiJian-Shu1992}
Jian-Shu Li.
\newblock Nonexistence of singular cusp forms.
\newblock {\em Compositio Math.}, 83(1):43--51, 1992.

\bibitem[Pol]{Pollack2022APAW}
Aaron Pollack.
\newblock The rankin-selberg method: A user's guide.
\newblock {\em Notes available
  \url{https://mathweb.ucsd.edu/~apollack/rankin-selberg.pdf}}.

\bibitem[PS17]{PollackShah2017}
Aaron Pollack and Shrenik Shah.
\newblock On the {R}ankin-{S}elberg integral of {K}ohnen and {S}koruppa.
\newblock {\em Math. Res. Lett.}, 24(1):173--222, 2017.

\bibitem[PS18]{PollackShah2018}
Aaron Pollack and Shrenik Shah.
\newblock The spin {$L$}-function on {$\rm GSP_6$} via a non-unique model.
\newblock {\em Amer. J. Math.}, 140(3):753--788, 2018.

\bibitem[PSR87]{Piatetski-ShapiroRallis1987Doubling}
Ilya Piatetski-Shapiro and Stephen Rallis.
\newblock {\em {$L$}-functions for the classical groups}, volume 1254 of {\em
  Lecture Notes in Mathematics}.
\newblock Springer-Verlag, Berlin, 1987.

\bibitem[PSR88]{Piatetski-ShapiroRallis1988}
I.~Piatetski-Shapiro and S.~Rallis.
\newblock A new way to get {E}uler products.
\newblock {\em J. Reine Angew. Math.}, 392:110--124, 1988.

\bibitem[RS89]{RallisSoudry1989}
Stephen Rallis and David Soudry.
\newblock Analytic continuation of {A}rchimedean {W}hittaker integrals.
\newblock {\em Proc. Amer. Math. Soc.}, 105(1):42--51, 1989.

\bibitem[Sha74]{Shalika1974}
J.~A. Shalika.
\newblock The multiplicity one theorem for {${\rm GL}_{n}$}.
\newblock {\em Ann. of Math. (2)}, 100:171--193, 1974.

\bibitem[Yan21]{Yan2021}
Pan Yan.
\newblock $l$-function for $\mathrm{Sp}(4)\times\mathrm{GL}(2)$ via a
  non-unique model.
\newblock {\em arXiv:2110.05693}, 2021.

\bibitem[Yan24]{Yan2024JRMS}
Pan Yan.
\newblock A note on a {H}ecke type integral for {${\rm Sp}(2n) \times {\rm
  GL}(1)$}.
\newblock {\em J. Ramanujan Math. Soc.}, 39(1):13--20, 2024.

\end{thebibliography}

\end{document}